\documentclass[letterpaper,12pt,oneside,reqno]{amsart}
\usepackage{amsmath,amsthm,amsfonts,amssymb,bbm,color}
\usepackage{graphicx,psfrag,subfigure,url}
\usepackage{cite}
\usepackage{mathrsfs}
\usepackage{hyperref}
\usepackage[DIV13]{typearea}

\newcommand{\CwPre}[1]{\widetilde{\mathcal{C}}_{#1}} 
\newcommand{\CsPre}[1]{\widetilde{\mathcal{D}}_{#1}}
\newcommand{\Cv}[1]{\mathcal{C}_{#1}}
\newcommand{\Cs}[1]{\mathcal{D}_{#1}}
\newcommand{\CvEps}[1]{\mathcal{C}^{\epsilon}_{#1}}
\newcommand{\CvEpsR}[1]{\mathcal{C}^{\epsilon}_{#1}}
\newcommand{\CvEpsRGeq}[1]{\mathcal{C}^{\epsilon}_{#1}}
\newcommand{\CvEpsReq}[1]{\mathcal{C}^{\epsilon}_{#1}}
\newcommand{\CvRLeq}[1]{\mathcal{C}_{#1}}

\numberwithin{equation}{section}
\newcommand{\I}{{\rm i}}
\newcommand{\Id}{\mathbbm{1}}
\newcommand{\Or}{\mathcal{O}}
\newcommand{\N}{\mathbb{N}}
\newcommand{\e}[0]{\varepsilon}
\newcommand{\EE}{\ensuremath{\mathbb{E}}}
\newcommand{\PP}{\ensuremath{\mathbb{P}}}
\newcommand{\R}{\ensuremath{\mathbb{R}}}
\newcommand{\C}{\ensuremath{\mathbb{C}}}
\newcommand{\Z}{\ensuremath{\mathbb{Z}}}
\newcommand{\Rplus}{\ensuremath{\mathbb{R}_{+}}}
\newcommand{\Zgzero}{\ensuremath{\mathbb{Z}_{>0}}}

\newcommand{\Q}{\ensuremath{\mathbb{Q}}}
\newcommand{\Y}{\ensuremath{\mathbb{Y}}}
\newcommand{\MM}{\ensuremath{\mathbf{MM}}}
\newcommand{\WM}[1]{\ensuremath{\mathbf{WM}}_{(#1)}}
\newcommand{\Zsd}{\ensuremath{\mathbf{Z}}}
\newcommand{\Fsd}{\ensuremath{\mathbf{F}}}

\newcommand{\Sym}{\ensuremath{\mathrm{Sym}}}

\newcommand{\Res}[1]{\underset{{#1}}{\mathrm{Res}}}
\newcommand{\la}[0]{\lambda}

\newcommand{\T}{\mathcal{T}}

\renewcommand{\Re}{\operatorname{Re}}
\renewcommand{\Im}{\operatorname{Im}}
\DeclareMathOperator{\dist}{dist}

\DeclareMathOperator{\Ai}{Ai}
\DeclareMathOperator{\sign}{sgn}
\DeclareMathOperator*{\var}{Var}

\newtheorem{theorem}{Theorem}[section]
\newtheorem{proposition}[theorem]{Proposition}

\newtheorem{lemma}[theorem]{Lemma}
\newtheorem{corollary}[theorem]{Corollary}
\newtheorem{claim}[theorem]{Claim}

\newtheorem{remark}[theorem]{Remark}

\newtheorem{definition}[theorem]{Definition}
\newenvironment{remarks}{\begin{remark}\normalfont}{\end{remark}}

\title[Free energy fluctuations for directed polymers]{Free energy fluctuations for directed polymers in random media in $1+1$ dimension}

\author[A. Borodin]{Alexei Borodin}
\address{A. Borodin,
Massachusetts Institute of Technology,
Department of Mathematics,
77 Massachusetts Avenue, Cambridge, MA 02139-4307, USA}
\email{borodin@math.mit.edu}

\author[I. Corwin]{Ivan Corwin}
\address{I. Corwin,
Microsoft Research,
New England, 1 Memorial Drive, Cambridge, MA 02142, USA}
\email{ivan.corwin@gmail.com}

\author[P.L. Ferrari]{Patrik Ferrari}
\address{P.L. Ferrari,
  Bonn University,
Institute for Applied Mathematics,
Endenicher Allee 60, 53115 Bonn, Germany}
\email{ferrari@uni-bonn.de}

\begin{document}
\sloppy
\maketitle
\thispagestyle{empty}

\begin{abstract}
We consider two models for directed polymers in space-time independent random media (the O'Connell-Yor semi-discrete directed polymer and the continuum directed random polymer) at positive temperature and prove their KPZ universality via asymptotic analysis of exact Fredholm determinant formulas for the Laplace transform of their partition functions. In particular, we show that for large time $\tau$, the probability distributions for the free energy fluctuations, when rescaled by $\tau^{1/3}$, converges to the GUE Tracy-Widom distribution.

We also consider the effect of boundary perturbations to the quenched random media on the limiting free energy statistics. For the semi-discrete directed polymer, when the drifts of a finite number of the Brownian motions forming the quenched random media are critically tuned, the statistics are instead governed by the limiting Baik-Ben Arous-P\'{e}ch\'{e} distributions from spiked random matrix theory. For the continuum polymer, the boundary perturbations correspond to choosing the initial data for the stochastic heat equation from a particular class, and likewise for its logarithm -- the Kardar-Parisi-Zhang equation. The Laplace transform formula we prove can be inverted to give the one-point probability distribution of the solution to these stochastic PDEs for the class of initial data.
\end{abstract}

\setcounter{tocdepth}{1}
\tableofcontents
\hypersetup{linktocpage}

\section{Introduction and main results}

The main results of this paper are contained in Sections \ref{1.2}, \ref{CDRPfe} and \ref{preasymSec}. The below section introduces the study of directed polymers and motivates our main results within that field.

\subsection{Directed polymers in random media.}
This article studies the effect of quenched disorder (homogeneous and inhomogeneous) on a class of path measures introduced first by Huse and Henley~\cite{HH85} which are commonly called {\it directed polymers in a random media} (DPRM). Such polymers are directed in what is often referred to as a {\it time} direction, and then are free to configure themselves in the remaining $d$ {\it spatial} dimensions. The probability $dP_{\beta,Q}(\pi(\cdot))$ of a given configuration $\pi(\cdot)$ of the polymer is then given relative to an underlying reference path measure $dP_0$ by a \mbox{Radon-Nikodym} derivative. This derivative is written as a {\it Boltzmann weight} involving a {\it Hamiltonian} $H_{Q}$ which assigns an energy to the path\footnote{Actually, the energy is $-H_{Q}(\pi(\cdot))$ but we keep the same convention as in the models analyzed.}:
\begin{equation*}
dP_{\beta,Q}(\pi(\cdot)) = Z_{\beta,Q}^{-1} e^{\beta H_{Q}(\pi(\cdot))} dP_0(\pi(\cdot)).
\end{equation*}
In the above equation $\beta$ is the inverse temperature which balances the entropy of the underlying reference path measure with the energy of the Hamiltonian. The subscript $Q$ stands for {\it quenched} which means that this $H_{Q}(\pi(\cdot))$ is a function of some disorder $\omega$ which we think of as an element of a probability space. Hence, with respect to this probability space, $H_{Q}$ is a random function. The normalization constant $Z_{\beta,Q}$, given by
\begin{equation*}
Z_{\beta,Q}= \int e^{\beta H_{Q}(\pi(\cdot))} dP_0(\pi(\cdot)),
\end{equation*}
is the quenched partition function, and is a function of $\omega$ as well. The measure $dP_{\beta,Q}$ is a quenched polymer measure since it is also a function of $\omega$. We denote averages with respect to the disorder $\omega$ by $\EE$, so that $\EE(Z_{\beta,Q})$ is the average of the quenched partition function. We denote by $\PP$ the probability measure for the disorder $\omega$ and denote the variance with respect to the disorder as $\var{\cdot}$. From now on we will assume $dP_0$ is the path measure of a random walk with either a free end point, or a specified pinned endpoint. The latter case is called a {\it point-to-point} polymer. We will focus mainly on point-to-point polymers herein.

At infinite temperature, $\beta=0$, and under standard hypotheses on $dP_0$ (i.e., i.i.d.\ finite variance increments) the measure $dP_{\beta,Q}(\pi(\cdot))$ rescales diffusively to that of a Brownian motion and thus the polymer is purely maximizing entropy. At zero temperature, $\beta=\infty$, the polymer measure concentrates on the path (or paths) $\pi$ which maximize the polymer energy $H_Q(\pi)$. A well studied challenge is to understand the effect of quenched disorder at positive $\beta$ on the behavior of a $dP_{\beta,Q}$-typical path of the free energy
$F_{Q,\beta}:=\beta^{-1} \ln (Z_{\beta,Q})$. A rough description of the behaviour is given by the {\it transversal fluctuation exponent} $\xi$ and the {\it longitudinal fluctuation exponent} $\chi$. There are many different ways these exponents have been defined, and it is not at all obvious that they exist for a typical polymer model -- though it is believed that they do. As $n$ goes to infinity, the first exponent describes the fluctuations of the endpoint of the path $\pi$: typically $|\pi(n)| \approx n^{\xi}$. The second exponent describes the fluctuations of the free energy: $\var{F_{\beta,Q}} \approx n^{2\chi}$. Assuming the existence of these exponents, in order to have a better understanding of the system it is of essential interest to understand the statistics for the properly scaled location of the endpoint and fluctuations of the free energy.

From now on we will focus entirely on Hamiltonians which take the form of a path integral through a space-time independent noise field (the quenched disorder). In the discrete setting of $dP_0$ given as a simple symmetric random walk (SSRW) of length $n$, we consider a homogeneous noise field given by i.i.d.\ random variables $w_{t,x}$. The quenched Hamiltonian is then defined as $H_{Q}(\pi(\cdot)) = \sum_{t=0}^{n} w_{t,\pi(t)}$.

The first rigorous mathematical work on directed polymers was by Imbrie and Spencer~\cite{IS88} in 1988 where they proved that in dimensions $d\ge 3$ and with small enough $\beta$, the polymer is diffusive ($\xi=1/2$). Bolthausen~\cite{Bol89} strengthened the result to a central limit theorem for the endpoint of the walk. This means that in $d\ge 3$ entropy dominates at high enough temperature, i.e., the polymer behaves as if there were no noise. Bolthausen's work relied on the now fundamental observation that renormalized partition function (for $dP_0$ a SSRW of length $n$) $W_n=Z_{\beta,Q} /\EE(Z_{\beta,Q})$ is a martingale. By a zero-one law, the limit $W_\infty=\lim_{n\to \infty} W_n$ is either almost surely $0$ or almost surely positive.

Since at $\beta=0$ the effect of the randomness $\omega$ vanishes and one has $W_\infty=1$, one refers to the case $W_{\infty}>0$ as {\it weak disorder}, while the term {\it strong disorder} is used when $W_{\infty}=0$. There is a critical value $\beta_c$ such that weak disorder holds for $\beta<\beta_c$ and strong for $\beta>\beta_c$. It is known that $\beta_c=0$ for $d\in\{1,2\}$~\cite{CSY04} and $0<\beta_c\le\infty$ for $d\ge 3$. In $d\ge 3$ and weak disorder the walk converges to a Brownian motion~\cite{CJ06}. On the other hand, the behavior of the polymer paths in strong disorder is different since there exist (random) points at which the path $\pi$ has a uniform (in $n$) positive probability (under $dP_{\beta,Q}$) of ending (see~\cite{CSY04,CC11}).

\subsubsection*{The KPZ universality conjecture for $d=1$.}
In this paper we focus on the $d=1$ case. The universality conjecture says that the scaling exponents and limiting fluctuation statistics for the free energy exist and are universal in the sense that they do not depend of the details of the model, provided some generic requirements are satisfied (e.g.\ the variables $w$ have only local correlations,  and enough non-trivial finite moments, see also the review~\cite{Cor11}). The models we consider below are predicted to belong to the so-called the Kardar-Parisi-Zhang (KPZ) universality class~\cite{KPZ86} for which $\xi=2/3$ and $\chi=1/3$~\cite{FNS77,BKS85}. In particular, since the transition from strong to weak disorder happens at $\beta=0$, the universality conjecture predicts that for any $\beta>0$ the scaling exponents $\xi,\chi$ and the fluctuation statistics equal the corresponding zero temperature ones ($\beta=\infty$).

The first step to this conjecture is to identify $\xi,\chi$ and describe the scaling limit for polymers and their free energy. This can be done by studying any model believed to be in the universality class. The second step is to show universality, i.e., show that the results are independent of the chosen model. In this paper we show such universality for any $\beta > 0$ for two particular polymer models -- the semi-discrete and continuum directed random polymers.

Consider now $Z_{\beta,Q}^{\rm pp}(\tau,x)$ to be the point-to-point partition function of polymers ending at $x$ at time $\tau$. Denote by $F_{\beta,Q}^{\rm pp}:=\frac{1}{\beta}\ln (Z_{\beta,Q}^{\rm pp})$ the free energy and by $f_{\beta,Q}^{\rm pp}$ the free energy limit shape (law of large numbers of $F_{\beta,Q}^{\rm pp}$ along different velocities). Then, a stronger version of the KPZ universality conjecture is to claim that over all models, there exists a unique limit\footnote{More precisely, the uniqueness should be up to a space-time scaling by parameters which can be calculated from the microscopic properties of the polymer such as $dP_0$, $\beta$ and the disorder distribution (see, for example,~\cite{Spo12,KK08}).}
\begin{equation}\label{eq1.3}
\lim_{\e\to 0} \e^{\chi} \left(F_{\beta,Q}^{\rm pp}(\e^{-1}\tau,\e^{-\xi}x)-
\e^{-1}f_{\beta,Q}^{\rm pp}(\tau,\e^{1-\xi}x)\right).
\end{equation}
One issue is that, since we are in the strong disorder regime, the quenched free energy differs from the annealed one (easier to compute) and there is no general way of determining it. The conjectural space-time limit is described in~\cite{CQ11} where it is called the {\it KPZ renormalization fixed point}. Information about this fixed point has generally come from studying exactly solvable models at zero temperature, $\beta=\infty$, such as last passage percolation, TASEP or PNG (see the review~\cite{Cor11}). For these models, one indeed finds $\xi=2/3$, $\chi=1/3$, and although the existence of the full space-time limit under these scalings is not known, for fixed time $\tau>0$ the limiting spatial process is the Airy$_2$ process~\cite{PS02,Jo03b}. This extends the results of~\cite{BDJ99,Jo00b} that show that the limiting one-point distribution (fixed $(\tau,x)$) is governed by the GUE Tracy-Widom distribution~\cite{TW94}.

Some of the zero temperature models can be reinterpreted as growth models of interfaces belonging to the KPZ universality class (where the free energy plays the role of a height function). Thus, the universality conjecture says that the height function (properly rescaled) converges to the same limit as the polymer free energy. These other models require initial data (i.e., height function profile at time zero) and the description of the KPZ fixed point takes the limit of this initial data into consideration. The scalings $\xi$ and $\chi$ do not change, but the limiting statistics reflect the initial data (see the reviews~\cite{Fer07,Fer10b} and recent experiments~\cite{TS12}). For random polymers the change in statistics occurs for instance if one looks at point-to-line problem instead. We will look at a different way to change the statistics by considering point-to-point polymers with boundary perturbations.

\subsubsection*{Polymers with boundary perturbations}
The role of initial data for growth processes and particle systems can be mimicked in the case of polymers by introducing inhomogeneity into the quenched disorder so as to encourage the path measure to spend more time on the boundary of the underlying path space. For instance consider a polymer with $dP_0$ given by a SSRW (either free or point-to-point) in a quenched disorder formed by i.i.d.\ centered weights $w_{t,x}$, for $x<t$; and i.i.d.\ boundary weights $w_{t,t}$ with positive mean, for $t\geq 0$. As the mean of the boundary weights $w_{t,t}$ increases, the paths which stay along the boundary for a long time tend to have a larger Boltzmann weight, and above a critical threshold, the energy-entropy competition is dominated by the boundary attraction so that the path spends a macroscopic proportion of its time pinned along the boundary. Note that once a path leaves the boundary it can not return. This leads to Gaussian scalings and statistics for the free energy and hence takes us beyond the basin of attraction for the KPZ renormalization fixed point.

However, in a scaling window of perturbation strength around the critical value, the boundary energy is strong enough to modify the free energy statistics and polymer measure, without modifying the scaling exponents. Modulating the strength within the window one sees a transition from sub to super critical behaviors (in terms of exponents and statistics).
Under the scaling (\ref{eq1.3}) these critically perturbed polymers should converge to the KPZ fixed point with a certain type of initial data, which results in different statistics than in the homogeneous disorder case. Information about these statistics originated from the analysis of a few solvable zero temperature, $\beta=\infty$, models~\cite{BBP05,BP07,BR00,BFS09,SI04,BC09}.

\subsubsection*{Contributions of this paper}
In this work we consider two polymer models with boundary perturbations at positive temperature: the O'Connell-Yor semi-discrete directed polymer and the continuum directed random polymer (CDRP). Each of these two models are themselves scaling limits of discrete polymers with general weight distributions when the inverse temperature $\beta$ scales to 0 simultaneously with the other parameters scaling to infinity. This type of scaling is called intermediate disorder scaling as it moves between weak and strong disorder (see~\cite{AKQ10,AKQ12,AKQ12b,QM12,QMR12}). This distinguishes these two models as being somewhat universal in their own rights, and hence provides increased motivation to prove the KPZ universality conjecture in their cases.

The first contribution of this work is to rigorously establish the KPZ universality conjecture (for one-point scalings and statistics) for these two positive temperature polymer models with boundary perturbations. In particular, we determine the critical perturbation strength and show that the phase diagram of one-point scalings and statistics for these polymers match that of the solvable zero temperature models. These results significantly expand upon the previous proved results for positive temperature polymer models~\cite{ACQ10,CQ10,BC11,Sep09,SV10}.

Our results rely on an algebraic framework for the exact solvability of these polymers which comes from Macdonald processes and their limits and degenerations~\cite{BC11}. One output of that work is an exact formula for the Laplace transform of the partition function for the semi-discrete polymer with arbitrary boundary perturbation strength. From this,~\cite{BC11} showed KPZ universality for the one-point scalings and statistics for the unperturbed semi-discrete polymer at sufficiently low temperatures. The second contribution of this work is to develop a variant on the formulas of~\cite{BC11} for the semi-discrete polymer which are well-adapted for rigorous asymptotic analysis at all temperatures (and for boundary perturbations). This requires modifications at a fairly high level of the hierarchy of models arising from Macdonald processes.

The third contribution is the discovery and proof of exact formulas for the Laplace transform of the partition function for the CDRP with boundary perturbations. The semi-discrete polymer has an intermediate disorder scaling limit in which it converges to the CDRP \cite{QM12,QMR12}, and these formulas are found via rigorous asymptotics of the aforementioned variant on the formulas of~\cite{BC11}. These formulas display a striking similarity (and a clear limit transition) to those of the KPZ fixed point with initial data corresponding to the boundary perturbations~\cite{BBP05}. The CDRP free energy solves the Kardar-Parisi-Zhang (KPZ) stochastic PDE and the boundary perturbation corresponds to particular class of half Brownian-like initial data. The discovered formulas also provide a new class of statistics for this equation, which is believed to model the fluctuations of randomly growing interfaces.

Let us now explain in detail our results.

\subsection{O'Connell-Yor semi-discrete polymer free energy}\label{1.2}
The first result is about a semi-discrete directed polymer model introduced by O'Connell and Yor~\cite{OCY01}.

\begin{definition}
An {\it up/right path} in $\R\times \Z$ is an increasing path which either proceeds to the right or jumps up by one unit. For each sequence $0<s_1<\cdots<s_{N-1}<\tau$ we can associate an up/right path $\phi$ from $(0,1)$ to $(\tau,N)$ which jumps between the points $(s_i,i)$ and $(s_{i},i+1)$, for $i=1,\ldots, N-1$, and is continuous otherwise. Fix a real vector $a=(a_1,\ldots, a_N)$ and let $B(s) = (B_1(s),\ldots, B_N(s))$ for $s\geq 0$ be independent standard Brownian motions such that $B_i$ has drift $a_i$. Let $\PP$ and $\EE$ denote the probability measure and expectation with respect to the disorder $B(\cdot)$.

Define the {\it energy} of a path $\phi$ to be
\begin{equation*}
E(\phi) = B_1(s_1)+\left(B_2(s_2)-B_2(s_1)\right)+ \cdots + \left(B_N(\tau) - B_{N}(s_{N-1})\right).
\end{equation*}
Then the {\it O'Connell-Yor semi-discrete directed polymer partition function} $\Zsd^{N}(\tau)$ is given by
\begin{equation}\label{Zsd}
\Zsd^{N}(\tau) = \int e^{E(\phi)} d\phi,
\end{equation}
where the integral is with respect to Lebesgue measure on the Euclidean set of all up/right paths $\phi$ (i.e., the simplex of jumping times $0<s_1<\cdots<s_{N-1}<\tau$).
The quenched {\it free energy} $\Fsd^{N}(\tau)$ is defined as
\begin{equation}\label{Fsd}
\Fsd^{N}(\tau) = \ln \Zsd^{N}(\tau),
\end{equation}
whereas the annealed free energy is given by $\Fsd_{\rm an}^N(\tau):=\ln \EE \Zsd^{N}(\tau)$.
\end{definition}

This model is related to the discrete polymer discussed in the introduction as follows: By rotating an angle of $\pi/4$ the discrete polymer becomes a measure on up/right lattice paths starting at $(1,1)$ with weights in $(\Z_{> 0})^2$. Consider the point-to-point partition function for paths going from $(1,1)$ to $(\tau\e^{-1},N)$ and rescale the lattice weights like $\e^{1/2} w_{i,j}$. Then as $\e\to 0$, due to the invariance principle the up/right lattice paths become like $\phi$ of the above definition, and the discrete path integral energy becomes $E(\phi)$.

We focus on scaling the semi-discrete polymer as $\tau=\kappa N$ for some $\kappa>0$. Due to Brownian scaling, the parameter $\kappa$ can be changed to 1 by replacing $E(\phi)$ by $\kappa^{1/2} E(\phi)$ in the exponential of (\ref{Zsd}). In this way, $\kappa$ corresponds to an inverse temperature parameter $\beta= \kappa^{1/2}$.

It is easy to calculate the annealed free energy exactly as
\begin{equation*}
\Fsd^{N}_{\rm an}(\tau) = \frac{\tau}{2} + (N-1)\ln \tau - \ln (N-1)!
\end{equation*}
which after setting $\tau=\kappa N$ implies that the annealed free energy density converges to
\begin{equation*}
\textbf{f}_{\rm an}(\kappa):=\lim_{N\to \infty} N^{-1} \Fsd_{\rm an}^N(\kappa N) =\frac{\kappa}{2} + \ln \kappa +1.
\end{equation*}

Let us briefly relate what the law of large numbers is for the quenched free energy of the unperturbed polymer.

\begin{definition}\label{digammadef}
The Digamma function is given by $\Psi(z) = \frac{d}{dz}\ln \Gamma(z)$. For $\theta\in \R_+$, define
\begin{equation*}
\kappa_\theta:=\Psi'(\theta),\quad f_\theta:=\theta \Psi'(\theta)-\Psi(\theta),\quad c_\theta:=(-\Psi''(\theta)/2)^{1/3},
\end{equation*}
or, equivalently, for $\kappa \in\R_+$, define
\begin{equation*}
\theta^\kappa:=(\Psi')^{-1}(\kappa)\in\R_+,\quad f^\kappa:=\inf_{t>0} (\kappa t - \Psi(t))\equiv f_{\theta^\kappa},\quad c^\kappa:=c_{\theta^\kappa}.
\end{equation*}
\end{definition}

The following law of large numbers for $\Fsd^{N}(\kappa N)$ was conjectured in~\cite{OCY01} and proved in~\cite{MOC07}. Consider the semi-discrete directed polymer with drift vector \mbox{$a=(0,\ldots,0)$}. Then for all $\kappa>0$,
\begin{equation*}
\lim_{N\to \infty} N^{-1} \Fsd^{N}(\kappa N) = f^{\kappa}, \qquad \textrm{a.s.}
\end{equation*}

It follows that for all $\kappa>0$, $\textbf{f}_{\rm an}(\kappa)=\frac{\kappa}{2} + \ln \kappa +1>f^{\kappa}$, and the quenched and annealed free energy density converge in the $\kappa\to 0$ limit. This is in agreement with strong disorder.

Below we analyze the large $N$ asymptotics of the fluctuations of $\Fsd^{N}(\kappa N)$ when centered by $N f^\kappa$.

\begin{theorem}\label{ThmPosTempAsy}
Consider the semi-discrete directed polymer with drift vector \mbox{$a=(a_1,\ldots, a_m,0,\ldots,0)$} where $m\leq N$ is fixed and the $m$ non-zero real numbers $a_1,a_2,\ldots, a_m$ may depend on $N$. We can consider without loss of generality that \mbox{$a_1\geq a_2\geq \ldots \geq a_m$} as the free energy is invariant under permuting these parameters. Then for all $\kappa>0$, we have the following characterization of the limiting behavior of the free energy $\Fsd^{N}(\kappa N)$ as $N\to \infty$.
\begin{itemize}
\item[(a)] If $\limsup_{N\to\infty}N^{1/3} (a_1(N)-\theta^\kappa)=-\infty$, then
\begin{equation*}
\lim_{N\to \infty} \PP\left( \frac{ \Fsd^{N}(\kappa N) - N f^{\kappa}}{c^\kappa N^{1/3}}\leq r\right) = F_{{\rm GUE}}(r),
\end{equation*}
where $F_{{\rm GUE}}$ is the GUE Tracy-Widom distribution~\cite{TW94}.
\item[(b)] If $\lim_{N\to \infty} N^{1/3}(a_i(N) -\theta^\kappa) = b_i\in \R\cup\{-\infty\}$ for $i=1,\ldots,m$, then
\begin{equation*}
\lim_{N\to \infty} \PP\left( \frac{ \Fsd^{N}(\kappa N) - N f^{\kappa}}{c^\kappa N^{1/3}}\leq r\right) = F_{{\rm BBP},b}(r),
\end{equation*}
where $F_{{\rm BBP},b}$ is the Baik-Ben Arous-P\'{e}ch\'{e}~\cite{BBP05} distribution from spiked random matrix theory, with $b=(b_1,\ldots,b_m)$.
\end{itemize}
\end{theorem}

The definitions of $F_{{\rm GUE}}(r)$ and $F_{{\rm BBP},b}$ are provided below in Definition~\ref{GUEBBPdef}. The fact that the result is independent of the ordering of $a_1,\ldots,a_N$ is apparent from the formulas, see e.g.\ Theorem~\ref{OConYorFluctThm} below. In Section~\ref{formalOConnellYorKPZ} we reduce the proof of this result to a claim on certain asymptotics of the Fredholm determinant formula presented in Section~\ref{preasymSec}. We provide a formal critical point derivation of these asymptotics in Section~\ref{formalOConnellYorKPZ} and a rigorous proof later in Section~\ref{proofOConnellYorKPZ}.

\begin{remarks}
If $\lim_{N\to \infty} N^{1/3}(a_i(N) -\theta^\kappa)=\infty$ then the boundary perturbation overwhelms the free energy fluctuations and the scalings and statistics become Gaussian in nature. This super critical regime is proved for zero temperature polymers~\cite{BBP05,Pec05}. We do not include a proof of this regime here, but it should be readily accessible from the exact formulas via asymptotic analysis.
\end{remarks}

\begin{remarks}
In the unperturbed case, i.e., \mbox{$a=(0,\ldots,0)$}, a tight upper bound on the exponent for the free energy fluctuation scalings was determined in~\cite{SV10}. In~\cite{BC11} the full one-point scaling limit was proved for $\kappa>\kappa^*$ with $\kappa^*$ a large (enough) constant forced by some technical consideration in the asymptotic analysis. In this article we do away with those technical issues which allows us to rigorously extend the asymptotics to all positive $\kappa$, as well as to $\kappa$ tending to zero simultaneously with $N$, as we soon will consider.
\end{remarks}

Theorem~\ref{ThmPosTempAsy} is expected by universality, because the same results hold in the zero-temperature limit ($\beta=\infty$) and the phase transition is expected to be at $\beta=0$.
More precisely, the limit of the free energy (divided by $\beta$) as $\beta$ goes to infinity and $(N,\tau)$ is fixed, is described by the ground-state maximization problem
\begin{equation*}
\begin{aligned}
M^N(\tau) :=& \lim_{\beta\to \infty} \frac{1}{\beta} \ln \int e^{\beta E(\phi)} d\phi\\
=& \max_{0<s_1<\cdots<s_{N-1}<\tau} B_1(s_1)+\left(B_2(s_2)-B_2(s_1)\right)+ \cdots + \left(B_N(\tau) - B_{N}(s_{N-1})\right).
\end{aligned}
\end{equation*}
In the unperturbed case, $a_i\equiv 0$, $M^N(1)$ is distributed as the largest eigenvalue of an $N\times N$ GUE random matrix~\cite{Bar01}, see also Theorem~1.1 of~\cite{OCon09}. In fact, as a process of $\tau$, $M^N(\tau)$ has the law of the largest eigenvalue of the standard Dyson Brownian motion on Hermitian matrices (the lower eigenvalues are also connected to certain generalizations of the free energy~\cite{OCon09}). It follows from the original work of Tracy and Widom~\cite{TW94} and also~\cite{For93,NW93} that
\begin{equation*}
\lim_{N \to \infty} \PP\left( \frac{ M^N(N) - 2N}{N^{1/3}}\leq r\right) = F_{{\rm GUE}}(r).
\end{equation*}

For general drift parameters, $M^N(\tau)$ is related to Dyson Brownian motion with drifts and the distribution of $M^N(\tau)$ then coincides with the largest eigenvalue of a {\it spiked} GUE matrix, for which the analog of Theorem~\ref{ThmPosTempAsy} was proved by P\'{e}ch\'{e}~\cite{Pec05}. The first such results were for spiked LUE matrices in the work of Baik-Ben Arous-P\'{e}ch\'{e}~\cite{BBP05}.

We now record the definitions of these limiting distributions in terms of Fredholm determinants. Note that there are many equivalent ways to rewrite these formulas (cf. \cite{Baik}) and we use the most convenient for our purposes.
\begin{definition}\label{GUEBBPdef}
The GUE Tracy-Widom distribution~\cite{TW94} is defined as
\begin{equation*}
F_{{\rm GUE}}(r) =\det(\Id-K_{\rm Ai})_{L^2(r,\infty)},
\end{equation*}
where $K_{\rm Ai}$ is the Airy kernel, that has integral representations
\begin{equation*}
K_{\rm Ai}(\eta,\eta')=\frac{1}{(2\pi\I)^2} \int_{e^{-2\pi\I/3}\infty}^{e^{2\pi\I/3}\infty} dw \int_{e^{-\pi\I/3}\infty}^{e^{\pi\I/3}\infty} dz\frac{1}{z-w}\frac{e^{z^3/3-z\eta}}{e^{w^3/3-w\eta'}}=\int_{\R_+}d\lambda \Ai(\eta+\lambda)\Ai(\eta'+\lambda),
\end{equation*}
where in the first representation the contours $z$ and $w$ must not intersect.

The BBP distribution from spiked random matrix theory~\cite{BBP05} is defined as
\begin{equation*}
F_{{\rm BBP},b}(r) =\det(\Id-K_{{\rm BBP},b})_{L^2(r,\infty)},
\end{equation*}
where $b=(b_1\geq b_2\geq \ldots \geq b_m)\in\R^m$ and $K_{{\rm BBP},b}$ is given by
\begin{equation}\label{eqBBP}
K_{{\rm BBP},b}(\eta,\eta')=\frac{1}{(2\pi\I)^2} \int_{e^{-2\pi \I/3}\infty}^{e^{2\pi\I/3}\infty} dw \int_{e^{-\pi \I/3}\infty}^{e^{\pi\I/3}\infty} dz
\frac{1}{z-w}\frac{e^{z^3/3-z\eta}}{e^{w^3/3-w\eta'}}\prod_{k=1}^m\frac{z-b_k}{w-b_k},
\end{equation}
where the path for $w$ passes on the right of $b_1,\ldots,b_m$ and does not intersect with the path $z$, see Figure~\ref{PFFigPathsBBP}. It is convenient to extend this definition to allow for $b_i=-\infty$ for all $i=\ell+1,\ldots,m$. Calling $\tilde{b}=(b_1,\ldots,b_{\ell})$ the finite values of $b$, we then set $K_{{\rm BBP},b}=K_{{\rm BBP},\tilde{b}}$ and likewise define $F_{{\rm BBP},b}(r)= F_{{\rm BBP},\tilde{b}}(r)$. Notice that if $m=0$ then $K_{{\rm BBP},b}=K_{\rm Ai}$. For representations of this kernel in terms of Airy functions see~\cite{BBP05}.
\begin{figure}[ht]
\begin{center}
\psfrag{Cw}[lb]{$\mathcal{C}_w$}
\psfrag{Cz}[lb]{$\mathcal{C}_z$}
\psfrag{bbar}[cb]{$\bar b$}
\includegraphics[height=5cm]{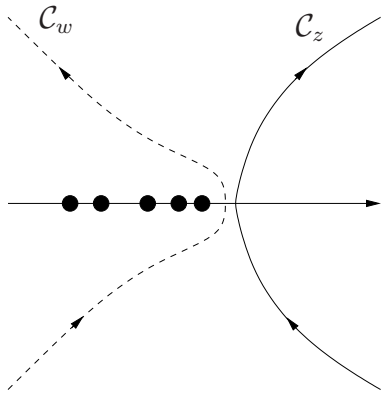}
\caption{Integration contours for $w$ (dashed) and $z$ (solid) of the kernel $K_{{\rm BBP},b}$. The black dots are $b_1,\ldots,b_m$.}
\label{PFFigPathsBBP}
\end{center}
\end{figure}
\end{definition}

\subsection{Continuum directed random polymer free energy}\label{CDRPfe}

Just as Brownian motion serves as a paradigm for and universal scaling limit of random walks, the {\it continuum directed random polymer} (CDRP) serves a similar role for $1+1$ dimensional directed random polymers~\cite{CDR10,AKQ10}. It is proved that the CDRP is the limit of (general weight distribution) discrete~\cite{AKQ12} or semi-discrete polymers~\cite{QM12} under intermediate disorder scaling in which the inverse temperature (or noise strength) is scaled to zero as the system size grows (so as to converge to space-time white noise). One such result is quoted as Theorem~\ref{scalinglimit}.

The directed polymers considered earlier are measures on random walk trajectories reweighted by a Boltzmann weight with Hamiltonian given in terms of a path integral through a space-time independent noise field along the walk's trajectory. In the continuum, random walks are replaced by Brownian motions and the space-time noise field becomes space-time Gaussian white-noise. It is possible to define a path measure in the continuum~\cite{AKQ12b}. Here we will only focus on the continuum limit of the polymer partition function (\ref{Zsd}), not the path measure. We define the CDRP partition function with respect to general boundary perturbations. These perturbations are the limit of critically tuned boundary perturbations of the discrete polymers.

\begin{definition}
The {\it partition function} for the continuum directed random polymer with boundary perturbation $\ln\mathcal{Z}_0(X)$ is given by the solution to the stochastic heat equation with multiplicative Gaussian space-time white noise and $\mathcal{Z}_0(X)$ initial data:
\begin{equation}\label{SHE}
\partial_T \mathcal{Z} = \tfrac{1}{2}\partial_X^2 \mathcal{Z} +\mathcal{Z}\dot{\mathscr{W}}, \qquad \mathcal{Z}(0,X)=\mathcal{Z}_0(X).
\end{equation}
The initial data $\mathcal{Z}_0(X)$ may be random but is assumed to be independent of the Gaussian space-time white noise $\dot{\mathscr{W}}$ and is assumed to be almost surely a sigma-finite positive measure. Observe that even if $\mathcal{Z}_0(X)$ is zero in some regions, the stochastic PDE makes sense and hence the partition function is well-defined.
\end{definition}

The stochastic heat equation (\ref{SHE}) is really short-hand for its integrated (mild) form
\begin{equation*}
\mathcal{Z}(T,X) = \int_{-\infty}^{\infty} p(T,X-Y) \mathcal{Z}_0(Y) dY + \int_{0}^{T} \int_{-\infty}^\infty p(T-S,X-Y)\mathcal{Z}(S,Y) \dot{\mathscr{W}}(dS,dY)
\end{equation*}
where (formally) Gaussian space-time white noise has covariance
\begin{equation*}
\EE\left[\dot{\mathscr{W}}(T,X)\dot{\mathscr{W}}(S,Y)\right] = \delta(T-S)\delta(Y-X),
\end{equation*}
and $p(T,X) = (2\pi T)^{-1/2} \exp(-X^2/2T)$ is the standard heat kernel. A detailed description of this stochastic PDE and the class of initial data for which it is well-posed can be found in~\cite{ACQ10,BG97}, or the review~\cite{Q12}.

As long as $\mathcal{Z}_0$ is an almost surely sigma-finite positive measure, it follows from work of M\"{u}ller~\cite{Mue91} that, almost surely, $\mathcal{Z}(T,X)$ is positive for all $T>0$ and $X\in \R$. Hence we can take its logarithm.

\begin{definition}
For $\mathcal{Z}_0$ an almost surely sigma-finite positive measure define the {\it free energy} for the continuum directed random polymer with boundary perturbation $\ln\mathcal{Z}_0(X)$ as
\begin{equation*}
\mathcal{F}(T,X) = \ln\mathcal{Z}(T,X).
\end{equation*}
\end{definition}

The random space-time function $\mathcal{F}$ is also the Hopf-Cole solution to the Kardar-Parisi-Zhang equation with initial data $\mathcal{F}_0(X)=\ln\mathcal{Z}_0(X)$~\cite{ACQ10,BG97}.

The solution to the stochastic heat equation is interpreted as a polymer partition function due to a version of the Feynman-Kac representation (see~\cite{BC95} or the review~\cite{Cor11}):
\begin{equation*}
\mathcal{Z}(T,X) = \EE_{B(T)=X}\left[\mathcal{Z}_0(B(0)):\,\exp\,: \left\{\int_0^{T} \dot{\mathscr{W}}(t,B(t))dt\right\}\right]
\end{equation*}
where the expectation $\EE$ is taken over the law of a Brownian motion $B$ which is run {\it backwards} from time $T$ position $X$. The $:\,\exp\,:$ is the {\it Wick exponential} which accounts for the fact that one can not na\"{\i}vely integrate white noise along a Brownian motion (various equivalent ways exist to define this exponential~\cite{BC95,ACQ10,Cor11}). By time reversal we may consider this as the partition function for Brownian bridges which can depart at time 0 from any location $B(0)\in \R$ with an energetic cost of $\ln \left(p(T,X-B(0))\mathcal{Z}_0(B(0))\right)$ and then must end at $X$ at time $T$. This departure energy cost is the limit of the boundary perturbations for the discrete and semi-discrete polymers.

Let us introduce the class of boundary perturbations which arise in the limit of the semi-discrete directed polymer partition function with the first few drift parameters $a_1,\ldots, a_m$ tuned as a function of $N$ in a critical way (and all other drifts zero).

\begin{definition}\label{spikedICdef}
The continuum directed random polymer partition function with $m$-{\it spiked} boundary perturbation corresponds to choosing initial data for (\ref{SHE}) as follows: Fix $m\geq 1$ and a real vector $b=(b_1,\ldots, b_m)$; then
\begin{equation*}
\mathcal{Z}_0(X) = \Zsd^m(X)\mathbf{1}_{X\geq 0}
\end{equation*}
where $\Zsd^m(X)$ is defined as in (\ref{Zsd}) with drift vector $b$, and where $\mathbf{1}_{X\geq 0}$ is the indicator function for $X\geq 0$. When $m=0$ we will define $0$-{\it spiked} initial data as corresponding to $\mathcal{Z}_0(X)=\mathbf{1}_{X=0}$, where $\mathbf{1}_{X=0}$ is the indicator function that $X=0$.
\end{definition}

\begin{theorem}\label{ThmIntDisAsy}
Fix $m\geq 0$ and a real vector $b=(b_1,\ldots, b_m)$. Consider the free energy of the continuum directed random polymer with $m$-spiked boundary perturbation with drift vector $b$ (Definition~\ref{spikedICdef}).
\begin{itemize}
\item[(a)] If $m=0$, then for any $T>0$ and any $S$ with positive real part,
\begin{equation*}
\EE\left[e^{-S \exp\left(\mathcal{F}(T,0) + T/4!\right)}\right] = \det(\Id-K_{{\rm CDRP}})_{L^2(\R_+)}.
\end{equation*}
\item[(b)] If $m\geq 1$, then for any $T>0$ and any $S$ with positive real part,
\begin{equation*}
\EE\left[e^{-S \exp\left(\mathcal{F}(T,0) + T/4!\right)}\right] = \det(\Id-K_{{\rm CDRP},b})_{L^2(\R_+)}.
\end{equation*}
\end{itemize}
\end{theorem}

The kernel in the above theorem are given now.

\begin{definition}\label{2.18def}
Fix $m\geq 0$ and a real vector $b=(b_1,\ldots, b_m)$. The integral kernel \begin{equation}\label{eq2.18}
K_{{\rm CDRP},b}(\eta,\eta')=\frac{1}{(2\pi\I)^2}\int_{\mathcal{C}_w} dw \int_{\mathcal{C}_z} dz \frac{\sigma \pi S^{(z-w)\sigma}}{\sin(\pi(z-w)\sigma)}  \frac{e^{z^3/3-z\eta'}}{e^{w^3/3-w\eta}} \prod_{k=1}^m\frac{\Gamma(\sigma w-b_k) }{\Gamma(\sigma z-b_k)},
\end{equation}
where $\sigma=(2/T)^{1/3}$. When $m=0$ the product of Gamma function ratios is replaced by 1 and the resulting kernel is denoted $K_{{\rm CDRP}}$. The $w$ contour $\mathcal{C}_w$ is from $-\frac{1}{4\sigma}-\I\infty$ to $-\frac{1}{4\sigma}+\I\infty$ and crosses the real axis on the right of $b_1/\sigma,\ldots,b_m/\sigma$. The $z$ contour $\mathcal{C}_z$ is taken as $\mathcal{C}_w$ shifted to the right by $\frac{1}{2\sigma}$ (see Figure~\ref{PFFigPathsCDRP} for an illustration). Just as in Definition~\ref{GUEBBPdef} there exist integral representations for these kernels involving Airy functions. In particular,
\begin{equation}\label{PFeqKPZ}
K_{{\rm CDRP}}(\eta,\eta')=\int_{\R} dt \frac{S}{S+e^{-t/\sigma}} \Ai(t+\eta)\Ai(t+\eta').
\end{equation}
Similar formulas exist for $K_{{\rm CDRP},b}$ involving Gamma deformed Airy functions \cite{CQ10,SI11}.

\begin{figure}[ht]
\begin{center}
\psfrag{Cw}[lb]{$\mathcal{C}_w$}
\psfrag{Cz}[lb]{$\mathcal{C}_z$}
\psfrag{s}[cb]{$\frac{1}{4\sigma}$}
\psfrag{0}[cb]{$0$}
\includegraphics[height=5cm]{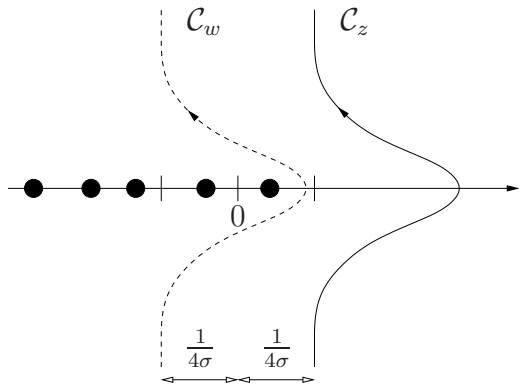}
\caption{Integration contours for $w$ (dashed) and $z$ (solid) of the kernel $K_{{\rm CDRP},b}$. The black dots are $b_1/\sigma,\ldots,b_m/\sigma$. The path $\mathcal{C}_z$ equals $\frac{1}{2\sigma}+\mathcal{C}_w$, so as to avoid the zeros of the sine function.}
\label{PFFigPathsCDRP}
\end{center}
\end{figure}
\end{definition}

\begin{remarks}
To recover the case of $(m-1)$-spiked from the $m$-spiked boundary perturbation case, one needs to take $b_m \to -\infty$ and simultaneously replace $S$ by $-b_m S$.
\end{remarks}

In Section~\ref{critpointCDRP}, Theorem~\ref{scalinglimit} we explain how the semi-discrete polymer partition function limits to that of the CDRP. This reduces the proof of the above theorem to a claim on the asymptotics of the Fredholm determinant formula presented in Section~\ref{preasymSec}. We provide a formal critical point derivation of these asymptotics in Section~\ref{critpointCDRP} and a rigorous proof later in Section~\ref{proofCDRP}.

\begin{remarks}
The above Laplace transform can be inverted via a contour integral in $S$ so as to give the probability distribution for the free energy. Note that the branch cut in $S^{(z-w)\sigma}$ in the integrand in (\ref{eq2.18}) should be taken as the negative real axis.

For $m=0$ the free energy probability distribution was discovered simultaneously and independently in both~\cite{ACQ10,SS10} and rigorously proved in~\cite{ACQ10} via Tracy and Widom's ASEP formulas~\cite{TW08,TW08b,TW08c,TW08cErratum}; the above Laplace transform formula was soon after (non-rigorously) derived from the replica trick approach in~\cite{CDR10,Dot10}. For $m=1$ the free energy probability distribution was discovered and rigorously proved in~\cite{CQ10} also via ASEP~\cite{TW09b}; the above Laplace transform formula was later (non-rigorously) derived from the replica trick approach in~\cite{SI11}. The general $m\geq 2$ result above is, to our knowledge, new and it is not clear how one would derive or prove it from ASEP.

There are other ways to write the kernel as well as the Fredholm determinant in the theorem, as can be seen in the above mentioned citations.
\end{remarks}

\begin{remarks}\label{perturbedshiftREM}
It is not necessary to focus just on the free energy at $(T,0)$. When \mbox{$m=0$}, for $T$ fixed, $\mathcal{F}(T,X) - \ln p(T,X)$ is a stationary process in $X$ \cite{ACQ10} due to the fact that space-time white noise is statistically invariant under affine shifts. For $m\geq 1$, $\mathcal{F}(T,X)$ is no longer stationary, however, a calculation given in Section~\ref{nonstatSec} shows that if we let \mbox{$\mathcal{\tilde{Z}}_0(X)\stackrel{(d)}{=}\Zsd^m(X)\mathbf{1}_{X\geq 0}$} for a shifted drift vector $b=(b_1+X/T,\ldots, b_m+X/T)$, then
\begin{equation*}
\mathcal{Z}(T,X) =  e^{-\frac{X^2}{2T}} \mathcal{\tilde{Z}}(T,0),
\end{equation*}
where $ \mathcal{\tilde{Z}}(T,X)$ solves the stochastic heat equation with initial data $\mathcal{\tilde{Z}}_0(X)$.
\end{remarks}

A corollary of the above theorem is the large $T$ asymptotics of the free energy fluctuations for the CDRP with $m$-spiked boundary perturbation.

\begin{corollary}\label{CDRPtoKPZ}
Fix $m\geq 0$ and a real vector $b=(b_1,\ldots, b_m)$. Consider the free energy of the continuum directed random polymer with $m$-spiked boundary perturbation with drift vector $\sigma b$, with $\sigma=(2/T)^{1/3}$ (Definition~\ref{spikedICdef}).
\begin{itemize}
\item[(a)] If $m=0$, then for any $r\in \R$,
\begin{equation*}
\lim_{T\to \infty} \PP\left(\frac{ \mathcal{F}(T,0) +T/4!}{(T/2)^{1/3}} \leq r \right) = F_{{\rm GUE}}(r),
\end{equation*}
where $F_{{\rm GUE}}$ is the GUE Tracy-Widom distribution~\cite{TW94} (see Definition~\ref{GUEBBPdef}).
\item[(b)] If $m\geq 1$, then for any $r\in \R$,
\begin{equation*}
\lim_{T\to \infty} \PP\left(\frac{ \mathcal{F}(T,0) +T/4!}{ (T/2)^{1/3}} \leq r \right) = F_{{\rm BBP},b}(r),
\end{equation*}
where $F_{{\rm BBP};b}$ is the Baik-Ben Arous-P\'{e}ch\'{e}~\cite{BBP05} distribution from spiked random matrix theory (see Definition~\ref{GUEBBPdef}).\\
\end{itemize}
\end{corollary}

For $m=0,1$ the above corollary is proved in~\cite{ACQ10} and~\cite{CQ10} (respectively). Given the new $m\geq 2$ formulas, it is straightforward to prove the full corollary as is done in Section~\ref{Corproof}.

\subsection{Fredholm determinant formula for semi-discrete polymer free energy}\label{preasymSec}

Theorems~\ref{ThmPosTempAsy} and~\ref{ThmIntDisAsy} are proved via asymptotic analysis of a Fredholm determinant formula for the Laplace transform of the O'Connell-Yor semi-discrete directed polymer partition function which we now give as Theorem~\ref{OConYorFluctThm}. The formula is written in terms of a Fredholm determinant. One of the surprising aspects of the known exactly solvable positive temperature directed random polymers is that the Laplace transform of their partition functions are given by Fredholm determinants. Here we write the Laplace transform as a double-exponential transform of the free energy.

\begin{definition}\label{CaCsdef}
For $\alpha\in \R$ and $\varphi\in (0,\pi/4)$, we define a contour $\Cv{\alpha,\varphi}$ that surrounds the portion of the real axis with values less than $\alpha$ by
$\Cv{\alpha,\varphi}=\{\alpha+e^{\I (\pi+\varphi)}y\}_{y\in \Rplus}\cup \{\alpha+e^{\I (\pi-\varphi)}y\}_{y\in \Rplus}$. The contour is oriented so as to have increasing imaginary part.

For every $v\in \Cv{\alpha,\varphi}$ we choose $R=-\Re(v)+\alpha+1$, $d>0$, and define a contour $\Cs{v}$ as follows:
$\Cs{v}$ goes by straight lines from $R-\I \infty$, to $R-\I d$, to $1/2-\I d$, to $1/2+\I d$, to $R+\I d$, to $R+\I\infty$. The parameter $d$ is taken small enough so that $v+\Cs{v}$ do not intersect $\Cv{\alpha,\varphi}$. We also call $\Cs{v,\vert}$ the portion of $\Cs{v}$ with real part $R$ and $\Cs{v,\sqsubset}$ the remaining part. See Figure~\ref{contours} for an illustration.
\end{definition}
\begin{figure}
\begin{center}
\psfrag{Cv}[lb]{$\Cv{\alpha,\varphi}$}
\psfrag{v+Cs}[lb]{$v+\Cs{v}$}
\psfrag{Cs}[lb]{$\Cs{v}$}
\psfrag{v}[cb]{$v$}
\psfrag{R}[cb]{$R$}
\psfrag{2d}[lb]{$2d$}
\psfrag{alpha}[cb]{$\alpha$}
\psfrag{0}[cb]{$0$}
\psfrag{phi}[lb]{$\varphi$}
\includegraphics[height=5cm]{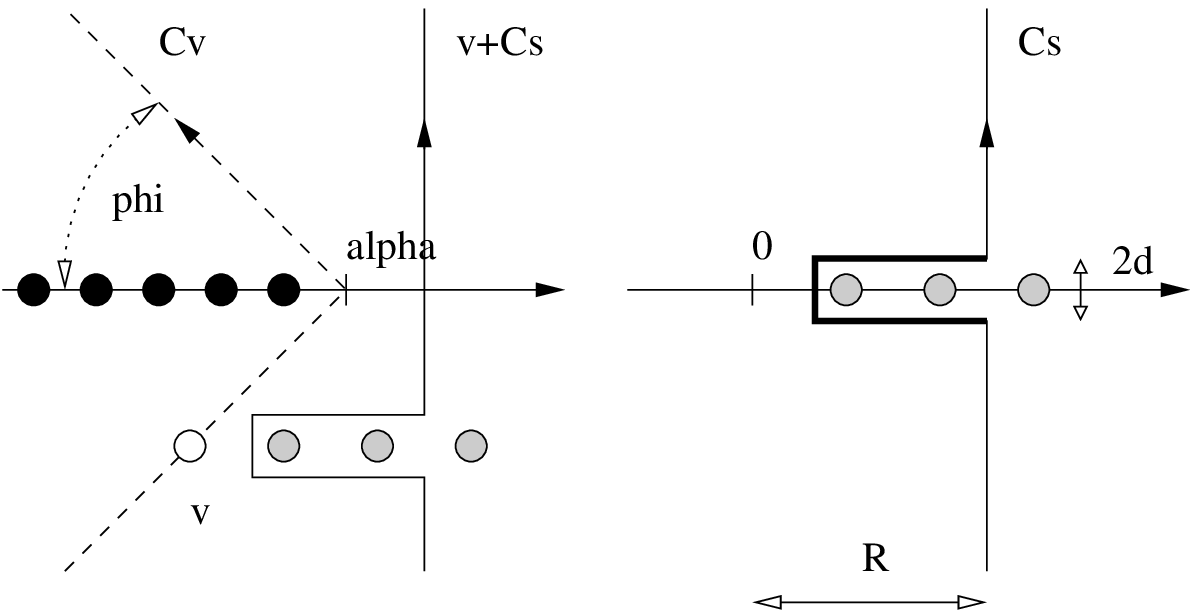}
\end{center}
\caption{Left: The contour $\Cv{\alpha,\varphi}$ (dashed), where the black dots are $a_1,\ldots,a_N$ and $\alpha>\max\{a_1,\ldots,a_N\}$. The contour $v+\Cs{v}$ is the solid line. Right: The contour $\Cs{v}$, where the thick part is $\Cs{v,\sqsubset}$ and the thin part is $\Cs{v,\vert}$. The grey dots are at $\{1,2,\ldots\}$}
\label{contours}
\end{figure}

\begin{theorem}\label{OConYorFluctThm}
For fixed $N\geq 1$, $\tau\geq 0$, a drift vector $a=(a_1,\ldots, a_N)$, and \mbox{$\alpha>\max\{a_1,\ldots,a_N\}$}, the Laplace transform of the partition function for the O'Connell-Yor semi-discrete directed polymer with drift $a$ is given by the Fredholm determinant formula
\begin{equation*}
\EE\left[ e^{-u \Zsd^{N}(\tau)}\right] = \det(\Id+ K_{u})_{L^2(\Cv{\alpha,\varphi})}
\end{equation*}
where $\Cv{\alpha,\varphi}$ is given in Definition~\ref{CaCsdef} for any $\varphi\in (0,\pi/4)$. The operator $K_u$ is defined in terms of its integral kernel
\begin{equation}\label{kvvprime}
K_u(v,v') = \frac{1}{2\pi \I}\int_{\Cs{v}}ds \Gamma(-s)\Gamma(1+s) \prod_{m=1}^{N}\frac{\Gamma(v-a_m)}{\Gamma(s+v-a_m)} \frac{ u^s e^{v\tau s+\tau s^2/2}}{v+s-v'}.
\end{equation}
\end{theorem}

Theorem~5.2.10 of~\cite{BC11} gives a similar formula for the Laplace transform of the semi-discrete directed polymer partition function. The difference between the two formulas is the contours involved -- both the contour of the $L^2$ space and the contour involved in defining the kernel. The contours of the above formula are unbounded. This somewhat technical modification is important since it enables us to perform rigorous steepest descent analysis as necessary to prove Theorems~\ref{ThmPosTempAsy} and~\ref{ThmIntDisAsy}. In~\cite{BC11}, the corresponding contours were bounded, thus limiting the asymptotic analysis of the semi-discrete polymer to a low temperature regime.

In order to prove this modified formula we modify the choice of contours very early in the proof of Theorem~5.2.10 of~\cite{BC11}. In Section~\ref{formalMainFormula} we provide a proof of the above formula (see in particular the end of Section~\ref{WhitMeasSec}) with the more detailed and technical pieces of the proof delayed until Section~\ref{MacSec}. The occurrence of unbounded contours introduce some new considerations in proving this theorem. It is not presently clear how to derive the above theorem directly from Theorem~5.2.10 of~\cite{BC11}.

\subsection{Acknowledgements}
We wish to thank the American Institute of Mathematics, since this work originated at the workshop on: The Kardar-Parisi-Zhang equation and universality class. In particular, we appreciate early discussions with G\'{e}rard Ben Arous, Tomohiro Sasamoto and Thomas Weiss. We have also benefited from correspondence with Jeremy Quastel, Gregorio Moreno Flores, and Daniel Remenik on their work. We area also grateful to B\'alint Vet{\H o} for careful reading of part of the manuscript. AB was partially supported by the NSF grant DMS-1056390. IC was partially supported by the NSF grant DMS-1208998; as well as the Clay Mathematics Institute through a Clay Research Fellowship and Microsoft Research through the Schramm Memorial Fellowship. PF was supported by the German Research Foundation via the SFB611--A12 project.
\section{Free energy fluctuations for the O'Connell-Yor semi-discrete polymer}\label{formalOConnellYorKPZ}

In this section we reduce the proof of Theorems~\ref{ThmPosTempAsy} to a statement about the asymptotics of a Fredholm determinant (Theorem~\ref{PFThmF2pert} below). We then provide a formal critical point derivation of the asymptotics, delaying the rigorous proof until Section~\ref{proofOConnellYorKPZ}.

The starting point for the proof of Theorem~\ref{ThmPosTempAsy} is the Fredholm determinant formula given in Theorem~\ref{OConYorFluctThm} for $\EE [e^{-u\Zsd^{N}(t)}]$. We rely on the fact that under the scalings we consider, the Laplace transform of the partition function converges to the asymptotic probability distribution of the free energy (a similar approach is used in the proof of Corollary~\ref{CDRPtoKPZ}). Towards this aim, define a sequence of functions $\{\Theta_N\}_{N\geq 1}$ by $\Theta_N(x) = \exp\left(-\exp\left(c^\kappa N^{1/3}x\right)\right)$, where $c^\kappa$ is given in Definition \ref{digammadef}.
Recall also that we are scaling $\tau=\kappa N$ for $\kappa>0$ fixed.

Assume that the drift vector $a=(a_1,\ldots,a_N)$ is as specified in the statement of Theorem~\ref{ThmPosTempAsy}.
Set
\begin{equation*}
u=u(N,r,\kappa)=e^{-Nf^\kappa - r c^\kappa N^{1/3}}
\end{equation*}
where $f^\kappa$ is as in Definition \ref{digammadef} and observe that
\begin{equation}\label{uZdpTheta}
\EE\left[e^{-u\Zsd^{N}(\kappa N)}\right] = \EE\left[\Theta_N\left(\frac{\Fsd^{N}(\kappa N) - N f^{\kappa}}{c^\kappa N^{1/3}}-r\right)\right].
\end{equation}
The $N\to\infty$ asymptotics of the left-hand side of (\ref{uZdpTheta}) can be computed from taking asymptotics of the Fredholm determinant formula given in Theorem~\ref{OConYorFluctThm} which states that
\begin{equation}\label{uZdpThetaFRED}
\EE\left[e^{-u\Zsd^{N}(\kappa N)}\right] =  \det(\Id+ K_{u})_{L^2(\Cv{\alpha,\varphi})}
\end{equation}
for \mbox{$\alpha>\max\{a_1,\ldots,a_N\}$} and $\varphi\in (0,\pi/4)$. This asymptotic result is stated as Theorem~\ref{PFThmF2pert} below and proved in Section~\ref{proofOConnellYorKPZ}. After explaining how it implies Theorem~\ref{ThmPosTempAsy} we provide a formal critical point derivation of the asymptotics.

\begin{theorem}\label{PFThmF2pert}
Consider a vector \mbox{$a=(a_1,\ldots, a_m,0,\ldots,0)$} where $m\leq N$ is fixed and the $m$ non-zero real numbers $a_1,a_2,\ldots, a_m$ may depend on $N$. We can consider without loss of generality that \mbox{$a_1\geq a_2\geq \ldots \geq a_m> 0$}. Then for all $\kappa>0$,
\begin{itemize}
\item[(a)] For the unperturbed case, $m=0$,
\begin{equation*}
\lim_{N\to\infty} \det(\Id+ K_{u})_{L^2(\Cv{\alpha,\varphi})}=\det(\Id-K_{\rm Ai})_{L^2(r,\infty)}=F_{{\rm GUE}}(r),
\end{equation*}
where $F_{{\rm GUE}}$ is the GUE Tracy-Widom distribution~\cite{TW94}.
\item[(b)] If $\lim_{N\to \infty} N^{1/3}(a_i(N) -\theta^\kappa) = b_i\in \R\cup \{-\infty\}$ for $i=1,\ldots,m$, then
\begin{equation*}
\lim_{N\to\infty} \det(\Id+ K_{u})_{L^2(\Cv{\alpha,\varphi})}=\det(\Id-K_{{\rm BBP},b})_{L^2(r,\infty)}=F_{{\rm BBP},b}(r),
\end{equation*}
where $F_{{\rm BBP},b}$ is the Baik-Ben Arous-P\'{e}ch\'{e}~\cite{BBP05} distribution from spiked random matrix theory.
\end{itemize}
\end{theorem}
We remind from Definition~\ref{GUEBBPdef} that when $b_i=-\infty$, $1\leq i \leq m$, then $F_{{\rm BBP},b}=F_{\rm GUE}$.

The above result implies the the right-hand side of (\ref{uZdpTheta}) has a limit $p(r)$ that is a continuous probability distribution function. Here $p(r)$ is the limiting distribution function in cases (a) or (b) of Theorem~\ref{PFThmF2pert}. The functions $\Theta_N(x-r)$ approximate $\mathbf{1}(x\leq r)$ in the sense necessary to apply Lemma~\ref{problemma1}, and hence $p(r)$ also describes the limiting probability distribution of the free energy:
\begin{equation*}
\lim_{N\to \infty} \PP\left( \frac{ \Fsd^{N}(\kappa N) - N f^{\kappa}}{c^\kappa N^{1/3}}\leq r\right) = p(r).
\end{equation*}
This implies Theorem~\ref{ThmPosTempAsy}.

\subsection{Formal critical point asymptotics for Theorem~\ref{PFThmF2pert}}\label{formalcalc1}
We provide a formal analysis of the asymptotics of the Fredholm determinant $\det(\Id+ K_{u})_{L^2(\Cv{\alpha,\varphi})}$. In particular, we only focus on the limit of the kernel $K_u$, and even in that pursuit, we only consider the pointwise limit of the kernel in (\ref{kvvprime}). We also disregard issues respecting the choice of contours. All of these issues are considered in the rigorous proof contained in Section~\ref{proofOConnellYorKPZ}.

The first set of manipulations to $K_u$ that we make are to rewrite $\Gamma(-s)\Gamma(1+s) = -\pi / \sin(\pi s)$, and to factor out the ratios of Gamma functions involving those $a_i$ for $1\leq i \leq m$. With a change of variable $\tilde z=s+v$, we obtain
\begin{equation}\label{zetaKeqn}
K_{u}(v,v') = \frac{-1}{2\pi \I}\int d\tilde z\, \frac{\pi}{\sin (\pi(\tilde z-v))}\frac{\exp\left(N G(v)+ r  c^\kappa N^{1/3} v\right)}{\exp\left(N G(\tilde z)+ r c^\kappa N^{1/3} \tilde z\right)}  \frac{1}{\tilde z-v'} \prod_{k=1}^{m}\frac{\Gamma(v-a_k) \Gamma(\tilde z)}{\Gamma(\tilde z-a_k)\Gamma(v)}\, ,
\end{equation}
where
\begin{equation}\label{Geqn}
G(z) = \ln \Gamma(z) - \kappa \frac{z^2}{2} + f^\kappa z.
\end{equation}

The problem is now prime for steepest descent analysis of the integral defining the kernel above. The idea of steepest descent is to find critical points for the function in the exponential, and then to deform contours so as to go close to the critical point. The contours should be engineered so that away from the critical point, the real part of the function $G$ in the exponential decays and hence as $N$ gets large, has negligible contribution. This then justifies localizing and rescaling the integration around the critical point. The order of the first non-zero derivative (here third order) determines the rescaling in $N$ (here $N^{1/3}$) which in turn corresponds with the scale of the fluctuations in the problem we are solving. It is exactly this third order nature that accounts for the emergence of Airy functions and hence the Tracy Widom (GUE) distribution as well as the BBP transition distributions.

The critical point equation for $G$ is given by $G'(z)=0$ with
\begin{equation*}
G'(z) = \Psi(z) - \kappa z + f^\kappa.
\end{equation*}
The Digamma function $\Psi(z)=\frac{d}{dz}\ln(\Gamma(z))$ is given in Definition~\ref{digammadef}. Also given in that definition is $\theta^\kappa\in \R_+$ which is the critical point, i.e., $G'(\theta^\kappa) = 0$. At the critical point $G''(\theta^\kappa)=0$ and $G^{(3)}(\theta^\kappa)=\Psi''(\theta^\kappa)=-2 (c^\kappa)^3$ so that Taylor expansion at the critical point gives (up to higher order terms)
\begin{equation*}
G(v) \simeq G(\theta^\kappa)- \frac{(c^\kappa)^3}{3} (v-\theta^\kappa)^3,\quad G(\tilde z) \simeq G(\theta^\kappa)- \frac{(c^\kappa)^3}{3} (\tilde z-\theta^\kappa)^3.
\end{equation*}
This cubic behavior suggests rescaling around $\theta^\kappa$ by the change of variables
\begin{equation*}
w = c^\kappa N^{1/3}(v-\theta^\kappa), \qquad w' = c^\kappa N^{1/3}(v'-\theta^\kappa), \qquad z = c^\kappa N^{1/3}(\tilde z-\theta^\kappa).
\end{equation*}
Under the above change of variables we find that as $N\to \infty$,
\begin{equation*}
\frac{\exp\left(N G(v)+ r c^\kappa N^{1/3} v\right)}{\exp\left(N G(\tilde z)+ r c^\kappa N^{1/3} \tilde z\right)}
\to \frac{\exp\left(z^3/3-r z\right)}{\exp\left(w^3/3-r w\right)}\,.
\end{equation*}
Note that since the $v,v'$ variables were scaled, there is a Jacobian factor of $1/(c^\kappa N^{1/3})$ introduced into the kernel. Grouping this with the reciprocal sine function we see that as $N\to \infty$,
\begin{equation*}
\frac{1}{c^\kappa N^{1/3}}\frac{\pi}{\sin (\pi(v-\tilde z))} \to \frac{1}{w-z}, \qquad \frac{d\tilde z}{\tilde z-v'} \to \frac{dz}{z-w'}\,.
\end{equation*}
It remains to study the ratio of Gamma functions. This is where the subcriticality (a) versus criticality (b) becomes important. First notice that the factor $\prod_{k=1}^{m}\Gamma(\tilde z)/\Gamma(v)\to 1$. The fact that the critical value for the $a_i$'s is $\theta^\kappa$ coincides with the centering of the change of variables is not an accident as we now explain.
After the above change of variables
\begin{equation*}
\frac{\Gamma(v-a_k)}{\Gamma(\tilde z-a_k)} = \frac{\Gamma(\theta^\kappa-a_k+ w/(c^\kappa N^{1/3}))}{\Gamma(\theta^\kappa-a_k+z/(c^\kappa N^{1/3}))}\,.
\end{equation*}
As long as $\limsup_{N\to\infty}N^{1/3} (a_k(N)-\theta^\kappa)=-\infty$, as $N\to \infty$, the numerator and denominator both converge to $\Gamma(\theta^\kappa-a_k)$ and hence their ratio is 1. Thus for the subcritical case (a) the limiting kernel is given by
\begin{equation*}
\frac{1}{2\pi \I}\int dz\, \frac{1}{(w-z)(z-w')}\frac{\exp\left(z^3/3-r z\right)}{\exp\left(w^3/3-r w\right)}\,.
\end{equation*}

In the critical case, $\limsup_{N\to\infty}N^{1/3} (a_k(N)-\theta^\kappa) = b_k$ for $1\leq k\leq m$. This means that after the change of variables
\begin{equation*}
\frac{\Gamma(v-a_k)}{\Gamma(\tilde z-a_k)} = \frac{\Gamma((w-b_k)/(c^\kappa N^{1/3}))}{\Gamma((z-b_k)/(c^\kappa N^{1/3}))}\simeq \frac{z-b_k}{w-b_k}
\end{equation*}
for large $N$, since $\Gamma(z)\simeq 1/z$ near $0$. Therefore, in the critical case (b) the limiting kernel is given by
\begin{equation*}
\widetilde K_{{\rm BBP},b}(w,w')=\frac{1}{2\pi \I}\int dz\, \frac{1}{(w-z)(z-w')}\frac{\exp\left(z^3/3-r z\right)}{\exp\left(w^3/3-r w\right)}\prod_{k=1}^{m}\frac{z-b_k}{w-b_k}\,.
\end{equation*}

The variables $w,w'$ are integrated along a contour $\mathcal{C}$ from $e^{-2\pi\I/3}\infty$ to $e^{2\pi\I/3}\infty$, passing on the right of $b_1,\ldots,b_m$. The variable $z$ is integrated along a contour which  goes from $e^{-\pi\I/3}\infty$ to $e^{\pi\I/3}\infty$ without crossing $\mathcal{C}$.

The limiting Fredholm determinants we have derived can be rewritten as Fredholm determinants on $L^2(r,\infty)$ as shown in Section~\ref{AppFredDet}, and hence one sees their equivalence to those given in Definition~\ref{GUEBBPdef}.

This completes the formal critical point derivation of Theorem~\ref{PFThmF2pert}. Let us note, that in order to make this formal manipulations into a proof it is necessary (among other things) to be careful about the contours. First one has to find a steep descent path\footnote{For an integral $I=\int_\gamma dz\, e^{t f(z)}$, we say that $\gamma$ is a steep
descent path if (1) $\Re(f(z))$ reaches the maximum at some $z_0\in\gamma$: $\Re(f(z))< \Re(f(z_0))$ for $z\in\gamma\setminus\{z_0\}$, and (2) $\Re(f(z))$ is monotone along $\gamma$ except at its maximum point $z_0$ and, if $\gamma$ is closed, at a point $z_1$ where the minimum of $\Re(f)$ is reached.} for $G(v)$, which might not be obvious due to the Digamma function (it turns out a useful representation for the real part of the Digamma function is as an infinite sum, see Section~\ref{sect5.1}). Secondly, one would like to find a steep descent path for $-G(\tilde{z})$, but because the path for $\tilde{z}$ has to include all the poles at $v+1,v+2,\ldots$, such path does not exists. The way out we used was to find a steep descent path for $-G(\tilde{z})$ and then add the contributions of the poles at the $v+1,\ldots,v+\ell$ which lie on the left of the path (see Figures~\ref{PFFigPathsTW} and~\ref{PFFigPathsBBPproof}). Finally, one needs to get estimates so that not only the kernel converges, but also the Fredholm determinant. Further technical details are presented in the proof, see Section~\ref{proofOConnellYorKPZ}.

\section{Laplace transform of the CDRP partition function}\label{critpointCDRP}

In this section we reduce the proof of Theorem~\ref{ThmIntDisAsy} to a statement about the asymptotics of a Fredholm determinant (Theorem~\ref{PFThmF2pert} below). We then provide a formal critical point derivation of the asymptotics, delaying the rigorous proof until Section~\ref{proofCDRP}. We also include two brief calculations delayed from the introduction.

The CDRP occurs as limits of discrete and semi-discrete polymers under what has been called {\it intermediate disorder} scaling. This means that the inverse temperature should be scaled to zero in a critical way as the system size scales up. For the discrete directed polymer it was observe independently by Calabrese, Le Doussal and Rosso~\cite{CDR10} and by Alberts, Khanin and Quastel~\cite{AKQ10} that if one scaled the inverse temperature to zero at the right rate while diffusively scaling time and space, then the discrete polymer partition function converges to CDRP partition function. Using convergence of discrete to continuum chaos series,~\cite{AKQ12} provide a proof of this result that is universal with respect to the underlying i.i.d.\ random variables which form the random media (subject to certain moment conditions).

Concerning the semi-discrete directed polymer,~\cite{QM12} prove convergence of the partition function to that of the CDRP.  The results of~\cite{QM12} deal with zero drift vector, and in~\cite{QMR12} the convergence is extended to deal with a finite number of non-zero drifts, critically tuned so as to result in an boundary perturbation for the CDRP partition function.

\begin{theorem}[\cite{QM12,QMR12}]\label{scalinglimit}
Fix $T>0$, $X\in \R$, $m\geq 0$, and a real vector $b=(b_1,\ldots,b_m)$. Set $\kappa=\sqrt{T/N}$ from which $\tau=\kappa N =\sqrt{TN}$, and $\theta=\theta^\kappa\simeq \sqrt{N/T}+\frac12$. For each $N\geq m$ define a drift vector $a=(a_1,\ldots,a_m,0,\ldots,0)$ where $a_k:=\theta+b_k$ for $1\leq k\leq m$. Consider the O'Connell-Yor semi-discrete polymer partition function $\Zsd^{N}(\tau)$ with drift vector $a$. For $\sqrt{T N}+X>0$, define its rescaling as
\begin{equation*}
\mathcal{Z}^N(T,X)=\frac{\Zsd^{N}(\sqrt{T N}+X)}{C(N,T,X)}
\end{equation*}
with scaling constant
\begin{equation*}
C(N,T,X) = \exp\left(N+\tfrac12 N\ln(T/N)+\tfrac12(\sqrt{T N}+X) + X\sqrt{N/T}\right) \exp\left(-\tfrac12 m\ln(T/N)\right).
\end{equation*}
Then as $N\to \infty$, $\mathcal{Z}^N(T,X)$ converges in distribution to $\mathcal{Z}(T,X)$ which is the partition function for the continuum directed random polymer with $m$-spiked boundary perturbation corresponding to drift vector $b$.
\end{theorem}
\begin{remarks}
For $S$ with positive real part, due to the almost sure positivity of $\mathcal{Z}(T,X)$, the above weak convergence also implies convergence of Laplace transforms:
\begin{equation*}
\EE\left[e^{-S \mathcal{Z}(T,X)}\right]=\lim_{N\to\infty}\EE\left[e^{-S \mathcal{Z}^N(T,X)}\right]
\end{equation*}
\end{remarks}

Let us focus on the semi-discrete polymer with drift vector as specified in the above theorem, and $X=0$. Then, rewrite the CDRP partition function in terms of the free energy, the previous remark implies
\begin{equation*}
\EE\left[e^{-S \exp\left(\mathcal{F}(T,0) + T/4!\right)}\right] = \lim_{N\to\infty} \EE\left[e^{-u \Zsd^{N}(\sqrt{TN})}\right] =\lim_{N\to\infty} \det(\Id+ K_{u})_{L^2(\Cv{\alpha,\varphi})}
\end{equation*}
where
\begin{equation}\label{usec3}
u=S e^{-N-\frac12 N \ln(T/N)-\frac12 \sqrt{TN}+T/4!} e^{\frac12 m \ln(T/N)},
\end{equation}
\mbox{$\alpha>\max\{a_1,\ldots,a_N\}$} and $\varphi\in(0,\pi/4)$. This reduces the proof of Theorem~\ref{ThmIntDisAsy} to the following.

\begin{theorem}\label{ThmCDRPdets}
Fix $S$ with positive real part, $T>0$, $m\geq 0$, and a real vector \mbox{$b=(b_1,\ldots,b_m)$}. Set $\kappa=\sqrt{T/N}$ from which $\tau=\kappa N =\sqrt{TN}$, and $\theta=\theta^\kappa\simeq \sqrt{N/T}+\frac12$. For each $N\geq m$ define a drift vector \mbox{$a=(a_1,\ldots,a_m,0,\ldots,0)$} where $a_k:=\theta+b_k$ for $1\leq k\leq m$, and define $u$ by (\ref{usec3}). Fix any \mbox{$\alpha>\max\{a_1,\ldots,a_N\}$} and $\varphi\in (0,\pi/4)$. Then, recalling the kernel $K_u$ from (\ref{kvvprime}), it holds that:
\begin{itemize}
\item[(a)] In the unperturbed case, $m=0$,
\begin{equation*}
\lim_{N\to\infty} \det(\Id+ K_{u})_{L^2(\Cv{\alpha,\varphi})}=\det(\Id-K_{{\rm CDRP}})_{L^2(\R_+)}
\end{equation*}
with $K_{{\rm CDRP}}$ given in Definition~\ref{2.18def}.
\item[(b)] For the perturbed case, $m\geq 1$,
\begin{equation*}
\lim_{N\to\infty} \det(\Id+ K_{u})_{L^2(\Cv{\alpha,\varphi})}=\det(\Id-K_{{\rm CDRP},b})_{L^2(\R_+)}
\end{equation*}
with $K_{{\rm CDRP},b}$ given in Definition~\ref{2.18def}.
\end{itemize}
\end{theorem}

This theorem is proved in Section~\ref{proofCDRP}. However, we now include a formal critical point derivation of the asymptotics.

\subsection{Formal critical point aysmptotics for Theorem~\ref{ThmCDRPdets}.}

We provide a formal analysis of the asymptotics of the Fredholm determinant $\det(\Id+ K_{u})_{L^2(\Cv{\alpha,\varphi})}$ under the prescribed scalings. In particular, we only focus on the limit of the kernel $K_u$, and even in that pursuit, we only consider the pointwise limit of the kernel in (\ref{kvvprime}). We also disregard issues respecting the choice of contours. All of these issues are considered in the rigorous proof contained in Section~\ref{proofCDRP}.

Set $\kappa=\sqrt{T/N}$ and note that $u= (S/\theta^m) e^{-N f^\kappa+\Or(N^{-1/2})}$ (recall Definition~\ref{digammadef}) so that the result is the same if we just set $u=(S/\theta^m) e^{-N f^\kappa}$. It is convenient to do the change of variable $v\to \theta + \sigma w$ and $\tilde z\to \theta + \sigma z$. In these new variables, the kernel is (up to the approximations) given by
\begin{equation*}
K_\theta(w,w')=\frac{-1}{2\pi \I}\int dz \frac{\sigma \pi S^{(z-w)\sigma}}{\sin(\pi (z-w)\sigma)}
\frac{e^{N G(\theta +\sigma w)-N G(\theta + \sigma z)}}{z-w'}\prod_{k=1}^m\frac{\Gamma(\sigma w-b_k) \Gamma(\theta + \sigma z)\theta^{\sigma w}}{\Gamma(\sigma z-b_k)\Gamma(\theta + \sigma w)\theta^{\sigma z}},
\end{equation*}
where $G$ is given by (\ref{Geqn}).

$G$ has a double critical point at $\theta$, as $G'(\theta)=G''(\theta)=0$. Therefore, using the large $N$ Taylor expansion of $G$ one sees that
\begin{equation*}
N G(\theta +\sigma w) \simeq N G(\theta) - \frac{1}{3} w^3.
\end{equation*}
Also, as $\theta$ is going to infinity with $N$, it is immediate that $\Gamma(\theta + \sigma z)\theta^{\sigma w}/\Gamma(\theta + \sigma w)\theta^{\sigma z}\to 1$. Thus, as $N\to\infty$, the kernel $K_{\theta}$ goes to
\begin{equation*}
\widetilde K_{{\rm CDRP},b}(w,w')=\frac{-1}{2\pi \I}\int dz \frac{\sigma \pi S^{(z-w)\sigma}}{\sin(\pi (z-w)\sigma)} \frac{e^{z^3/3-w^3/3}}{z-w'}\prod_{k=1}^m\frac{\Gamma(\sigma w-b_k)}{\Gamma(\sigma z-b_k)}.
\end{equation*}
The variables $w,w'$ are on the contour $\mathcal{C}_w$ and $z$ on $\mathcal{C}_z$ of Figure~\ref{PFFigPathsCDRP}. The Fredholm determinant with this kernel can be rewritten as a Fredholm determinant on $L^2(\R_+)$ as shown in Section~\ref{AppFredDet} and one checks that this corresponds with $\det(\Id-K_{{\rm CDRP},b})_{L^2(\R_+)}$ as desired.

This completes the brief and formal critical point derivation of Theorem~\ref{ThmCDRPdets}. The technical challenges in order to produce a rigorous proof are similar to those explained at the end of Section~\ref{formalOConnellYorKPZ}. There is a new difficulty with regards to contours which arises in this case, however. The variables $w,w',z$ arose from the change of variables of $v,v',\tilde z$. Besides a shift by $\theta$ (which goes to infinity when $N$), the variables are just scaled by $\sigma$, which is of order 1 (here $T$ is fixed). This means that all of the singularities of  $1/\sin(\pi(\tilde z-v))$ remain in the $N\to \infty$ scaling. This necessitates significant care in the choice of contours along which to take asymptotics. The complete proof of these asymptotics is given in Section~\ref{proofCDRP}.

\subsection{Affine shifting of the perturbed CDRP partition function}\label{nonstatSec}
The following calculation shows how the claim (for $m\geq 1$) in Remark~\ref{perturbedshiftREM} is derived. As explained in Section~\ref{CDRPfe}, we can write
\begin{equation*}
\mathcal{Z}(T,X) = \int_{-\infty}^{\infty} p(T,Y-X) \mathcal{Z}_0(Y) \EE_{\substack{B(0)=Y\\ B(T)=X}}\left[:\,\exp\,: \left\{\int_0^{T} \dot{\mathscr{W}}(t,B(t))dt\right\}\right] dY,
\end{equation*}
where $\EE_{\substack{B(0)=Y\\ B(T)=X}}$ denotes the expectation over a Brownian bridge $B$ starting at $B(0)=Y$ and ending at $B(T)=X$.
Let $\tilde{B}(s)=B(s)-s X/T$, then observe
\begin{equation*}
\mathcal{Z}(T,X) = \int_{-\infty}^{\infty} \frac{p(T,Y)}{p(T,Y)} p(T,Y-X) \mathcal{Z}_0(Y)  \EE_{\substack{\tilde{B}(0)=Y\\ \tilde{B}(T)=0}}\left[:\,\exp\,: \left\{\int_0^{T} \dot{\mathscr{W}}(t,\tilde{B}(t))dt\right\}\right] dY,
\end{equation*}
where the space-time white noise here is the affine shift of the previous one (but, in any case, equal in law). We have also inserted a factor of 1 so that we can now rewrite
\begin{equation*}
\mathcal{Z}(T,X) =  e^{-\frac{X^2}{2T}} \int_{-\infty}^{\infty}  p(T,Y)\mathcal{Z}_0(Y)e^{XY/T} \EE_{\substack{\tilde{B}(0)=Y\\ \tilde{B}(T)=0}}\left[:\,\exp\,: \left\{\int_0^{T} \dot{\mathscr{W}}(t,\tilde{B}(t))dt\right\}\right] dY.
\end{equation*}
If $\mathcal{Z}_0(Y) = \Zsd^m(Y)\mathbf{1}_{Y\geq 0}$ for drift vector $b=(b_1,\ldots, b_m)$ then $\mathcal{\tilde{Z}}_0(Y) = \mathcal{Z}_0(Y) e^{XY/T}= \Zsd^m(Y)\mathbf{1}_{Y\geq 0}$ for drift vector $b=(b_1+X/T,\ldots, b_m+X/T)$. Letting $\mathcal{\tilde{Z}}$ correspond to the solution to the stochastic heat equation with $\mathcal{\tilde{Z}}_0(Y)$ initial data, we find that
\begin{equation*}
\mathcal{Z}(T,X) =  e^{-\frac{X^2}{2T}} \mathcal{\tilde{Z}}(T,0).
\end{equation*}
Thus, after a parabolic shift, and an addition of drift into the boundary perturbation, the distribution of the general $X$ free energy can also be determined from Theorem~\ref{ThmIntDisAsy} as well.

\subsection{Proof of Corollary~\ref{CDRPtoKPZ}}\label{Corproof}
Recall $\sigma = (2/T)^{1/3}$. Let us focus on $m\geq 1$ and note that $m=0$ is proved identically. Define a sequence of function $\{\Theta_T\}_{T\geq 0}$ by $\Theta_T(x) = e^{-e^{x/\sigma}}$. Then if we set $S=e^{-r/\sigma}$
\begin{equation}\label{lhseqn}
\EE\left[e^{-S \exp\left(\mathcal{F}(T,0) + T/4!\right)}\right]=\EE\left[\Theta_{T}\left(\frac{ \mathcal{F}(T,0) +T/4!}{ \sigma^{-1}}-r\right)\right].
\end{equation}
Let us first calculating the $T\to \infty$ limit of the Fredholm determinant expression for the left-hand side of the above equality. It is easy to see that in this limit
\begin{equation*}
K_{{\rm CDRP},\sigma b}(\eta,\eta') \to K_{{\rm BBP},b}(\eta+r,\eta'+r).
\end{equation*}
This is because $T\to \infty$ corresponds to $\sigma\to 0$ and thus
\begin{equation*}
\frac{\sigma \pi S^{(z-w)\sigma}}{\sin(\pi(z-w)\sigma)} \to \frac{e^{-r(z-w)}}{z-w},\quad
\frac{\Gamma(\sigma(w-b_k)) }{\Gamma(\sigma(z-b_k))} \to \frac{z-b_k}{w-b_k}.
\end{equation*}
One readily identifies the resulting expression with that of Definition~\ref{GUEBBPdef}. The tail bounds necessary to justify this are not hard (see~\cite{ACQ10,CQ10} for instance).
Going back to (\ref{lhseqn}), this implies that
\begin{equation}
\lim_{T\to\infty} \EE\left[\Theta_{T}\left(\frac{ \mathcal{F}(T,0) +T/4!}{ \sigma^{-1}}-r\right)\right] = F_{{\rm BBP},b}(r).
\end{equation}
Since (away from $x=r$) $\Theta_{T}(x-r)$ is converging to $\mathbf{1}_{x\leq r}$ (and in particular due to Lemma~\ref{problemma1}),
\begin{equation*}
\lim_{T\to\infty} \EE\left[\Theta_{T}\left(\frac{ \mathcal{F}(T,0) +T/4!}{ \sigma^{-1}}-r\right)\right]= \lim_{T\to \infty} \PP\left(\frac{ \mathcal{F}(T,0) +T/4!}{\sigma^{-1}} \leq r \right)
\end{equation*}
and hence the corollary is proved.

\section{Proof of Theorem~\ref{OConYorFluctThm}: the main formula}\label{formalMainFormula}

The proof of Theorem~\ref{OConYorFluctThm} follows closely the proof of Theorem~5.2.10 in~\cite{BC11}. The major difference is that the formula resulting from Theorem~5.2.10 had bounded contours which were unsuitable for the full scope of asymptotic analysis necessary to prove Theorems~\ref{ThmPosTempAsy} and~\ref{ThmIntDisAsy}. This somewhat minor modification to the final formula requires us to modify the proof quite early on and in fact the unboundedness of contours results in a fair number of new technical steps in the proof. We produce in this section the complete proof. Results used along the way which are stated and proved in~\cite{BC11} are not reproved. Some of the more technical points in the proof are delayed until Section~\ref{MacSec}.

To prove Theorem~\ref{OConYorFluctThm} we use the theory of Macdonald processes as developed in~\cite{BC11}. As we explain below in Section~\ref{WhitMeasSec}, due to O'Connell's work~\cite{OCon09} on a continuum version of tropical RSK correspondence, the partition function $\Zsd^N(t)$ arises as a marginal of the Whittaker process (or measure). Macdonald processes sit above Whittaker processes due to the hierarchy of symmetric functions. Under suitable scaling, Macdonald processes converge weakly to Whittaker processes (and hence a suitable marginal converges to $\Zsd^N(t)$). Due to the Macdonald difference operators and the Macdonald version of the Cauchy identity it is possible to compute simple formulas for expectations of a large class of observables of Macdonald processes. In particular it is possible to compute $q$-moments for the marginal random variable which converges to $\Zsd^N(t)$. A $q$-Laplace transform can be rigorously computed by taking an appropriate generating function of the $q$-moments, and switching the expectations and summation is rigorously justified. For the $q$-Laplace transform we find a nice Fredholm determinant. Taking the degeneration of Macdonald processes to Whittaker processes, the $q$-Laplace transform becomes the usual Laplace transform, and due to the weak convergence, we recover the desired formula for the Laplace transform of $\Zsd^N(t)$.

We do not provide an introduction to the theory of Macdonald symmetric functions here. Instead we refer readers to Section 2.1 of~\cite{BC11} for all of the relevant details, or to Chapter VI of~\cite{Mac79}. We also do not define Macdonald processes in their full generality but content ourselves with studying a certain set of marginals of the processes which are called Macdonald measures, and a certain Plancherel specialization.

\subsection{O'Connell's Whittaker measure and its relation to polymers}\label{WhitMeasSec}

In order to state the pre-asymptotic Laplace transform formula and set up its derivation, it is useful to introduce a few concepts. Initially they may seem a little out of place, but due to O'Connell's work~\cite{OCon09} (some of which is recorded in Theorem~\ref{OConthm}) the connection to the semi-discrete polymer becomes clear. This connection is analogous to the relation between last passage percolation and the Schur process (see e.g.~\cite{Jo05}).

The {\it class-one $\mathfrak{gl}_{N}$-Whittaker functions} are basic objects of representation theory and integrable systems~\cite{Kos79,Eti99}. One of their properties is that they are eigenfunctions for the quantum $\mathfrak{gl}_{N}$-Toda chain. As showed by Givental~\cite{Giv97}, they can also be defined via the following integral representation
\begin{equation*}
\psi_{\lambda}(x_{N,1},\ldots,x_{N,N})=\int_{\R^{N(N-1)/2}} e^{\mathcal{F}_\lambda(x)} \prod_{k=1}^{N-1}\prod_{i=1}^k dx_{k,i},
\end{equation*}
where $\lambda=(\lambda_1,\ldots,\lambda_N)$ and
\begin{equation*}
\mathcal{F}_{\lambda}(x)=\I\sum_{k=1}^{N} \lambda_k\left(\sum_{i=1}^k x_{k,i}-\sum_{i=1}^{k-1} x_{k-1,i}\right)-\sum_{k=1}^{N-1}\sum_{i=1}^k \left(e^{x_{k,i}-x_{k+1,i}}+e^{x_{k+1,i+1}-x_{k,i}}\right).
\end{equation*}

For any $\tau>0$ set
\begin{equation*}
\theta_{\tau}(x_1,\ldots,x_N)=\int_{\R^N} \psi_{\nu}(x_1,\ldots,x_{N}) e^{-\tau\sum_{j=1}^N\nu_j^2/2} m_N(\nu)\prod_{j=1}^{N} d\nu_j
\end{equation*}
with the Skylanin measure
\begin{equation*}
m_{N}(\nu)=\frac1{(2\pi)^{N} (N)!}\prod_{j\ne k} \frac 1{\Gamma(\I \nu_k-\I \nu_j)}\,.
\end{equation*}

Note that our definition of Whittaker functions differs by factors of $\I$ from those considered by O'Connell \cite{OCon09}.

\begin{definition}
For $N\geq 1$, define the {\it Whittaker measure} with respect to a vector \mbox{$a=(a_1,\ldots, a_N)\in \R^N$} via the density function (with respect to Lebesgue) given by
\begin{equation}\label{WMdef}
\WM{a_1,\ldots, a_N;\tau}\big(\{T_{i}\}_{1\le i\le N}\big)=e^{-\tau\sum_{j=1}^N a_j^2/2}\psi_{\iota a}(T_{1},\dots,T_{N})\,{\theta_{\tau}(T_{1},\dots,T_{N})}.
\end{equation}
\end{definition}

The fact that this measure integrates to one follows from analytic continuation of the orthogonality relation for Whittaker functions (see~\cite{BC11} Proposition~4.1.17). We write expectations with respect to Whittaker measures as $\langle\cdot \rangle_{\WM{a_1,\ldots, a_N;\tau}}$.

Let us remark that in~\cite{OCon09,BC11} the Whittaker measure is a level $N$ marginal of a measure on triangular arrays called the Whittaker process, which is also defined via Whittaker functions, but whose precise definition will not be important for us in this work.

The connection between the Whittaker measure and the semi-discrete polymer is best explained by introducing an extension to the polymer partition function and free energy.

\begin{definition}
The {\it hierarchy of partition functions} $\Zsd^{N}_{n}(\tau)$ for $0\leq n\leq N$ is defined so that $\Zsd^{N}_{0}(\tau)=1$ and for $n\geq 1$,
\begin{equation*}
\Zsd^{N}_{n}(\tau) = \int_{D^N_{n}(\tau)} e^{\sum_{i=1}^{n} E(\phi_i)} d\phi_1 \cdots d\phi_n,
\end{equation*}
where the integral is with respect to the Lebesgue measure on the Euclidean set $D^N_n(\tau)$ of all $n$-tuples of non-intersecting (disjoint) up/right paths with initial points $(0,1),\ldots, (0,n)$ and endpoints $(\tau,N-n+1),\ldots, (\tau,N)$. For a path $\phi_i$ which starts at $(0,i)$, ends at $(\tau,N-n+i)$ and jumps between levels $j$ and $j+1$ at times $s_j$, the energy $E(\phi_i)$ is defined as
\begin{equation*}
E(\phi_i) = B_i(s_i)+\left(B_{i+1}(s_{i+1})-B_{i+1}(s_i)\right)+ \cdots + \left(B_N(\tau) - B_{N}(s_{N-1})\right),
\end{equation*}
with $B_1,\ldots,B_N$ independent Brownian motions. It follows that $\Zsd^{N}_{1}(\tau) = \Zsd^{N}(\tau)$ as defined in (\ref{Zsd}).

The {\it hierarchy of free energies} $\Fsd^{N}_{n}(\tau)$ for $1\leq n\leq N$ is defined via
\begin{equation*}
\Fsd^{N}_{n}(\tau) = \ln\left(\frac{ \Zsd^{N}_{n}(\tau)}{ \Zsd^{N}_{n-1}(\tau)}\right).
\end{equation*}
It follows that $\Fsd^{N}_{1}(\tau) = \Fsd^{N}(\tau)$ as defined in (\ref{Fsd}).
\end{definition}

The following result is shown in~\cite{OCon09} by utilizing a continuous version of the tropical RSK correspondence (see also~\cite{COSZ11} for a discrete analog) and certain Markov kernel intertwining relations.

\begin{theorem}\label{OConthm}
Fix $N\geq 1$, $\tau\geq 0$ and a vector of drifts $a=(a_1,\ldots, a_N)$, then $\left\{\Fsd^{N}_n(\tau)\right\}_{1\leq n\leq N}$ is distributed according to the Whittaker measure $\WM{-a_1,\ldots, -a_N;\tau}$ of  (\ref{WMdef}).
\end{theorem}

More is shown in~\cite{OCon09} including the fact that the collection of free energies evolves as a Markov diffusion with infinitesimal generator given in terms of the quantum $\mathfrak{gl}_N$ Toda lattice Hamiltonian; as well as the fact that the entire triangular array $\Fsd(\tau)$ is distributed according to the Whittaker process which we briefly mentioned earlier.

\begin{remarks}\label{remsym}
It is useful to note the following symmetry: The transformation \mbox{$T_{i}\leftrightarrow -T_{N+1-i}$} maps $\WM{-a_1,\ldots, -a_N;\tau}$ to $\WM{a_1,\ldots, a_N;\tau}$ (the sign of $a_j$'s changes). This easily follows from the definition of $\WM{a_1,\ldots, a_N;\tau}$.
\end{remarks}

The result below is similar to the formula found in Theorem~4.1.40 of~\cite{BC11} except that the contours have been modified. This technical change requires us to go through the proof of this theorem and perform a number of estimates which are harder than in~\cite{BC11}, due to the unboundedness of the contours.

\begin{theorem}\label{NeilPolymerFredDetThm}
Fix $N\geq 1$, $\tau>0$ and a vector $a=(a_1,\ldots, a_N)\in \R^N$ and $\alpha>\max\{a_i\}$. Then for all $u\in \C$ with positive real part
\begin{equation*}
\left\langle e^{-u e^{-T_{N}}}  \right\rangle_{\WM{a_1,\ldots, a_N;\tau}} = \det(\Id+ K_{u})_{L^2(\Cv{\alpha,\varphi})}
\end{equation*}
where the operator $K_u$ is defined in terms of its integral kernel
\begin{equation}\label{kvvprimeC}
K_{u}(v,v') = \frac{1}{2\pi \I}\int_{\Cs{v}}ds \Gamma(-s)\Gamma(1+s) \prod_{m=1}^{N}\frac{\Gamma(v-a_m)}{\Gamma(s+v-a_m)} \frac{ u^s e^{v\tau s+\tau s^2/2}}{v+s-v'}.
\end{equation}
The paths $\Cv{\alpha,\varphi}$ and $\Cs{v}$ are as in Definition~\ref{CaCsdef}, where $\varphi$ is any angle in $(0,\pi/4)$.
\end{theorem}

This theorem is proved over the course of the rest of this section, with the more technical aspects of the proof relegated to the later Section~\ref{MacSec}. Before going into this, let us note that in view Theorem~\ref{OConthm}, the proof of Theorem~\ref{OConYorFluctThm} is now immediate.

\begin{proof}[Proof of Theorem~\ref{OConYorFluctThm}]
By appealing to Theorem~\ref{OConthm} and Remark~\ref{remsym} we relate $\Zsd^{N}(\tau)$ to $T_{N}$ as
\begin{equation*}
\left\langle e^{-u e^{-T_{N}}}  \right\rangle_{\WM{a_1,\ldots, a_N;\tau}} = \EE\left[ e^{-u \Zsd^{N}(\tau)}\right].
\end{equation*}
Applying Theorem~\ref{NeilPolymerFredDetThm} gives the claimed Laplace transform and completes the proof.
\end{proof}

\subsection{Macdonald measures}
We now turn to the proof of Theorem~\ref{NeilPolymerFredDetThm}. Before giving the proof (as we do in Section~\ref{Neilpolyproof}) we need to recall and develop a variation on the solvability framework of Macdonald measures \cite{BC11}.

The Macdonald measure is defined as (following the notation of~\cite{BC11} where it is introduced and studied)
\begin{equation*}
\MM(\tilde a_1,\ldots,\tilde a_N;\rho)(\lambda)= \frac{P_{\lambda}(\tilde a_1,\ldots,\tilde a_N) Q_{\lambda}(\rho)} {\Pi(\tilde a_1,\ldots,\tilde a_N;\rho)}\,.
\end{equation*}
Here $\lambda$ is a partition of length $\ell(\lambda)\leq N$ (i.e., $\lambda=(\lambda_1\geq \lambda_2\geq \cdots \lambda_N\geq 0)$) and $P_{\lambda}$ and $Q_\lambda$ are Macdonald symmetric functions which are defined with respect to two additional parameters\footnote{The parameter $t$ does not correspond to time in any of the processes we have considered. For time we have either used $\tau$ in the semi-discrete polymer or $T$ in the CDRP} $q,t\in [0,1)$. The notation $P_{\lambda}(\tilde a_1,\ldots,\tilde a_N)$ means to specialize the symmetric function to an $N$ variable symmetric polynomial, and then evaluate those $N$ variables at the $\tilde a_i$'s\footnote{We use tildes presently since after a limit transition of Macdonald measures to Whittaker measures the $\tilde a_i$'s will become $a_i$'s}. The notation $Q_{\lambda}(\rho)$ means to apply a {\it specialization} $\rho$ to $Q_{\lambda}$. A specialization of a symmetric function $f$ in the space $\Sym$ of symmetric functions of an infinite number of indeterminants is an algebra homomorphism $\rho:\Sym\to\C$, and its application to $f$ is written as $f(\rho)$. We will focus here on a single class of {\it Plancherel} specializations $\rho$ which are defined with respect to a parameter $\gamma>0$ via the relation
\begin{equation}\label{tag1}
\sum_{n\ge 0} g_n(\rho) u^n= \exp(\gamma u) =: \Pi(u;\rho).
\end{equation}
Here $u$ is a formal variable and $g_n=Q_{(n)}$ is the $(q,t)$-analog of the complete homogeneous symmetric function $h_n$. The above generating function in $u$ is denoted $\Pi(u;\rho)$. Since $g_n$ forms a $\Q[q,t]$ algebraic basis of $\Sym$, this uniquely defines the specialization $\rho$. The Plancherel specialization has the property that $Q_{\lambda}(\rho)\geq 0$ for all partitions $\lambda$.

On account of the Cauchy identity for Macdonald polynomials
\begin{equation*}
\sum_{\la:\ell(\la)\le N} P_{\lambda}(\tilde a_1,\ldots,\tilde a_N) Q_{\lambda}(\rho) =: \Pi(\tilde a_1,\ldots,\tilde a_N;\rho) = \Pi(\tilde a_1;\rho)\cdots \Pi(\tilde a_N;\rho).
\end{equation*}

Our ability to compute expectations of observables for Macdonald measures comes from the following idea: Assume we have a linear operator $\mathcal{D}$ on the space of functions in $N$ variables whose restriction to the space of symmetric polynomials diagonalizes in the basis of Macdonald polynomials: $\mathcal{D} P_\la=d_\lambda P_\la$ for any partition $\la$ with $\ell(\la)\le N$. Then we can apply $\mathcal{D}$ to both sides of the identity
\begin{equation*}
\sum_{\la:\ell(\la)\le N} P_{\lambda}(\tilde a_1,\dots,\tilde a_N) Q_{\lambda}(\rho)=\Pi(\tilde a_1,\dots,\tilde a_N;\rho).
\end{equation*}
Dividing the result by $\Pi(\tilde a_1,\dots,\tilde a_N;\rho)$ we obtain
\begin{equation}\label{tag8}
\langle d_\lambda \rangle_{\MM(\tilde a_1,\dots,\tilde a_N;\rho)}=\frac{\mathcal{D}\Pi(\tilde a_1,\dots,\tilde a_N;\rho)} {\Pi(\tilde a_1,\dots,\tilde a_N;\rho)}\,,
\end{equation}
where $\langle \cdot\rangle_{\MM(\tilde a_1,\dots,\tilde a_N;\rho)}$ represents averaging $\cdot$ over the specified Macdonald measure.
If we apply $\mathcal{D}$ several times we obtain
\begin{equation*}
\langle d_\lambda^k \rangle_{\MM(\tilde a_1,\dots,\tilde a_N;\rho)}=\frac{\mathcal{D}^k \Pi(\tilde a_1,\dots,\tilde a_N;\rho)} {\Pi(\tilde a_1,\dots,\tilde a_N;\rho)}\,.
\end{equation*}
If we have several possibilities for $\mathcal{D}$ we can obtain formulas for averages of the observables equal to products of powers of the corresponding eigenvalues. One of the remarkable features of Macdonald polynomials is that there exists a large family of such operators for which they form the eigenbasis (and this fact can be used to define the polynomials). These are the Macdonald difference operators. We need only consider the first of these operators $D_N^1$ which acts on functions $f$ of $N$ independent variables as
\begin{equation*}
(D_N^1  f)(x_1,\ldots, x_N) = \sum_{i=1}^{N} f(x_1,\ldots,x_{i-1},qx_{i},x_{i+1},\ldots,x_N) \prod_{j\neq i} \frac{ tx_i-x_j}{x_i-x_j}.
\end{equation*}

This difference operator takes symmetric polynomials to symmetric polynomials and acts diagonally on Macdonald polynomials.
\begin{proposition}[VI(4.15) of~\cite{Mac79}]\label{prop5}
For any partition $\lambda=(\lambda_1\geq \lambda_2\geq \cdots)$ with $\lambda_m=0$ for $m>N$
\begin{equation*}
D_N^1 P_\lambda(x_1,\cdots,x_N)=\left(q^{\la_1}t^{N-1}+q^{\la_2}t^{N-2}+\cdots+q^{\la_N}\right) P_\lambda(x_1,\cdots,x_N).
\end{equation*}
\end{proposition}
Although the operator $D_N^1$ does not look particularly simple, it can be represented by contour integrals by properly encoding the shifts in terms of residues.

\begin{proposition}[Proposition~2.2.13 of~\cite{BC11}]\label{prop8}
Fix $k\ge 1$. Assume that $F(u_1,\dots,u_N)=f(u_1)\cdots f(u_N)$. Take $x_1,\dots,x_N >0$ and assume that $f(u)$ holomorphic and nonzero in a complex neighborhood of an interval in $\R$ that contains $\{q^ix_j\mid i=0,\ldots,k,j=1\ldots,N\}$.
Then
\begin{equation*}
\frac{\bigl({(D_n^1)}^k F\bigr)(x)}{F(x)}=\frac {(t-1)^{-k}}{(2\pi \I)^k} \oint\! \cdots\! \oint \! \! \prod_{1\le a<b\le k} \! \frac{(tz_a-qz_b)(z_a-z_b)}{(z_a-qz_b)(tz_a-z_b)} \prod_{c=1}^k\! \left(\prod_{m=1}^N \frac{(tz_c-x_m)}{(z_c-x_m)}\right)\!
\frac{f(qz_c)}{f(z_c)}\frac{dz_c}{z_c},
\end{equation*}
where the $z_c$-contour contains $\{qz_{c+1},\ldots,qz_k,x_1,\ldots,x_N\}$ and no other singularities for \mbox{$c=1,\dots,k$}.
\end{proposition}

Observe that $\Pi(\tilde a_1,\dots,\tilde a_N;\rho)=\prod_{i=1}^N \Pi(\tilde a_i;\rho)$. Hence, Proposition~\ref{prop8} is suitable for evaluating the right-hand side of equation (\ref{tag8}) and hence computing the associated observable of the Macdonald process.

\subsection{The emergence of a Fredholm determinant}\label{emergfreddet}

Macdonald polynomials in $N$ variables with $t=0$ are also known of as {\it $q$-deformed $\mathfrak{gl}_{N}$ Whittaker functions}~\cite{GLO11}. We now denote the Macdonald measure as $\MM_{t=0}(\tilde a_1,\dots,\tilde a_N;\rho)$ and refer to these as {\it $q$-Whittaker measures}.

The partition function for the corresponding $q$-Whittaker measure $\MM_{t=0}(\tilde a_1,\dots,\tilde a_N;\rho)$ when $\rho$ is Plancherel simplifies as
\begin{equation*}
\sum_{\lambda\in \Y(N)} P_\la(\tilde a_1,\dots,\tilde a_N)Q_\la(\rho)=\Pi(\tilde a_1,\dots,\tilde a_N;\rho)=\prod_{j=1}^N \exp(\gamma \tilde a_j).
\end{equation*}

\subsubsection{Formulas for q-moments}
Let us take the limit $t\to 0$ of Proposition~\ref{prop8}. Write
\begin{equation*}
\mu_k=\left\langle q^{k\lambda_N}\right\rangle_{\MM_{t=0}(\tilde a_1,\dots,\tilde a_N;\rho)}.
\end{equation*}
Then
\begin{equation}\label{mukdef}
\mu_k= \frac{(-1)^k q^{\frac{k(k-1)}{2}}}{(2\pi \I)^k} \oint \cdots \oint \prod_{1\leq A<B\leq k} \frac{z_A-z_B}{z_A-qz_B} \frac{f(z_1)\cdots f(z_k)}{z_1\cdots z_k} dz_1\cdots dz_k,
\end{equation}
where $f(z) = \prod_{i=1}^{N}\frac{\tilde a_i}{\tilde a_i-z}\exp\{(q-1)\gamma z\}$ and where the $z_j$-contours contain $\{q z_{j+1},\ldots, q z_k\}$ as well as $\{\tilde a_1,\ldots, \tilde a_N\}$ but not 0. For example when $k=2$ and all $\tilde a_i\equiv 1$, $z_2$ can be integrated along a small contours around 1, and $z_1$ is integrated around a slightly larger contour which includes 1 and the image of $q$ times the $z_2$ contour.

We have encountered the point at which we diverge from~\cite{BC11}. In particular, we will now deform our contours from bounded curves to infinite contours (as justified by Cauchy's theorem). This may not seem so significant presently, but in the final formula this frees us up to deform contours to steepest descent contours and hence to rigorously prove the various asymptotics desired for our limit theorems. However, the unboundedness of contours introduces new complications which we must address.

First, however, let us define two important unbounded contours we will soon encounter.
\begin{definition}\label{CwPredef}
For an illustration, see Figure~\ref{complexcontours}. For any $\tilde \alpha>0$ and $\varphi\in (0,\pi/4)$ define $\CwPre{\tilde \alpha,\varphi}=\{\tilde \alpha+e^{-\I \varphi\sign(y)}y,y\in \R\}$ (note that it is oriented so as to have decreasing imaginary part).
For $w\in \CwPre{\tilde \alpha,\varphi}$ for some choice of parameters, define a contour $\CsPre{w}$ in the same way as $\Cs{w}$ of Definition~\ref{CaCsdef} but with $R$ and $d$ taken such that the following holds: For all $s\in\CsPre{w}$ (i) if $b=\frac{\pi}{4}-\frac{\varphi}{2}$, then $\arg(w(q^s-1))\in (\pi/2+b,3\pi/2-b)$; (ii) $q^s w$ lies to the left of $\CwPre{\tilde \alpha,\varphi}$.
\end{definition}

\begin{remarks}\label{contoursrem}
Let us check that the contours $\CsPre{w}$ in above definition actually exist. Fix $\varphi\in (0,\pi/4)$ and fix a contour $\CwPre{\tilde \alpha,\varphi}$ as above. For $w\in \CwPre{\tilde \alpha,\varphi}$ we must show that there exists $R$ and $d$ as desired to satisfy (i) and (ii). The existence of these contains is clear for $|w|$ small, so let us consider when $|w|$ is sufficiently large. Then the argument of $w$ is roughly $\varphi$ (actually, $\pm \varphi$ but let us focus on $w$ with positive imaginary part). Then, for $R>R_0$ large enough and $d<d_0$ small enough (though positive) it follows from basic geometry that the argument of $w(q^s-1)$ can be bounded in $(\pi/2+b,3\pi/2-b)$. In order that $q^s w$ avoid $\CwPre{\tilde \alpha,\varphi}$. it suffices that $d$ be less than a constant times $|w|^{-1}$ and that $|w|q^R$ be small (but still of order~1). This implies that $R$ can be chosen roughly as $-\ln_q|w|$.
\end{remarks}

\begin{figure}
\begin{center}
\psfrag{a}[cb]{(a)}
\psfrag{b}[cb]{(b)}
\psfrag{c}[cb]{(c)}
\psfrag{dd}[cb]{(d)}
\psfrag{0}[cb]{$0$}
\psfrag{v}[cb]{$w$}
\psfrag{alpha}[cb]{$\tilde \alpha$}
\psfrag{phi}[rb]{$\varphi$}
\psfrag{qsw}[rb]{$q^s w, s\in \CsPre{w}$}
\psfrag{q12w}[lb]{$q^{1/2} w$}
\psfrag{Cv}[lb]{$\CwPre{\tilde \alpha,\varphi}$}
\psfrag{Cs}[lb]{$\CsPre{w}$}
\psfrag{R}[cb]{$R$}
\psfrag{d}[lb]{$d$}
\psfrag{qR}[cb]{$q^R$}
\psfrag{q12}[lb]{$q^{1/2}$}
\psfrag{2d}[lb]{$2d$}
\psfrag{12}[cb]{$1/2$}
\psfrag{1}[cb]{$1$}
\psfrag{mb}[cb]{$<b=\tfrac{\pi}{4}-\tfrac{\varphi}{2}$}
\includegraphics[height=10cm]{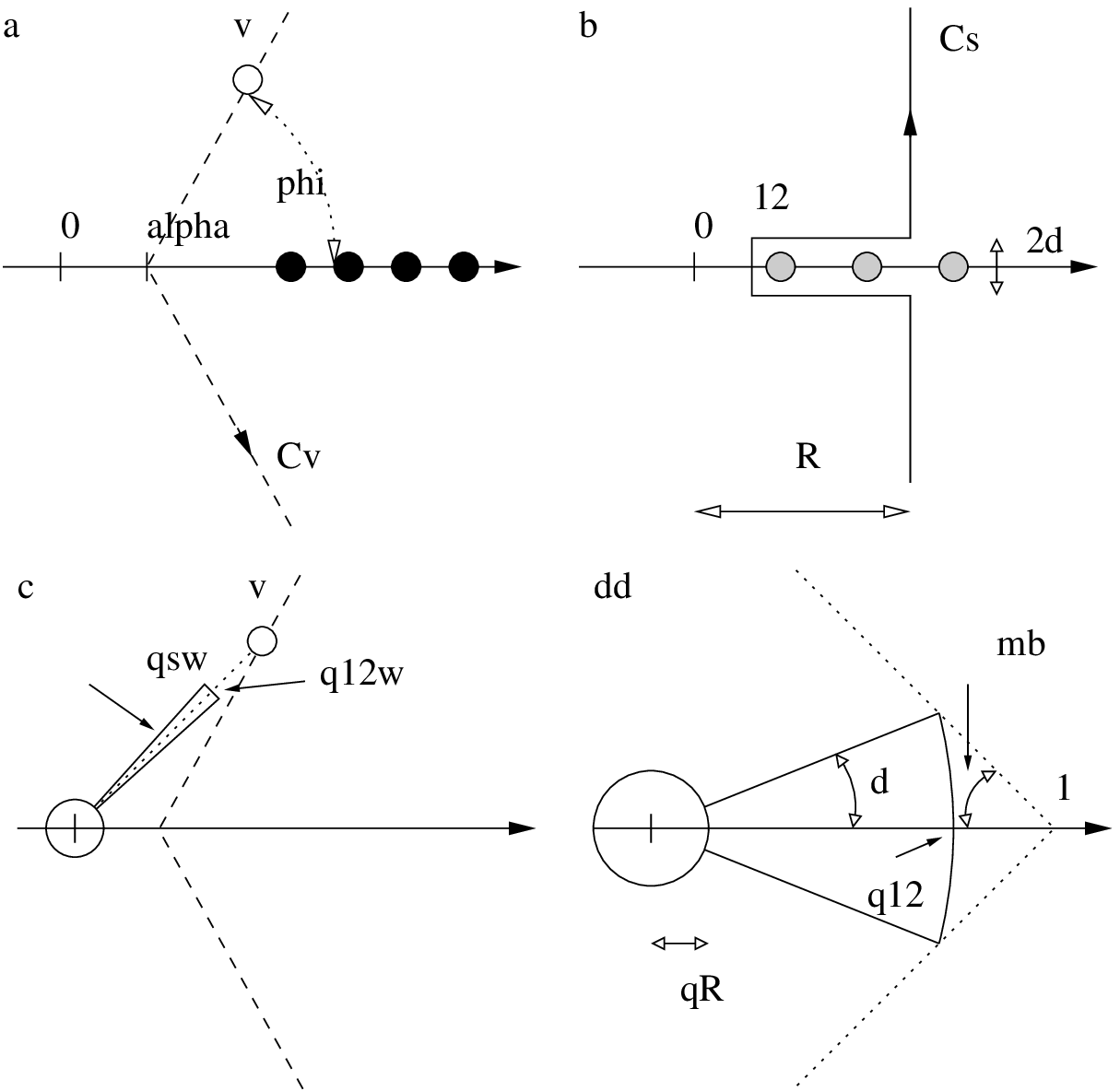}
\end{center}
\caption{(a) $w$-contour $\CwPre{\tilde \alpha,\varphi}$. The black dots are $\tilde a_1,\ldots,\tilde a_N$; (b) $s$-contour $\CsPre{w}$ which depends on the value of $w\in \CwPre{\tilde \alpha,\varphi}$; (c) The contour $\CsPre{w}$ is chosen so that $q^sw$ sits entirely to the left of $\CwPre{\tilde \alpha,\varphi}$ as is demonstrated here; (d) The image of $q^s$ for $s\in \CsPre{w}$. The value of $d$ and $R$ are also chosen so that the argument of $w(q^s-1)$ is contained in $(\pi/2+b,3\pi/2-b)$. This amounts to making sure that the angle $q^s-1$ makes with the negative real axis is within $(-b,b)$.}\label{complexcontours}
\end{figure}

For the moment we will only make use of the contour $\CwPre{\tilde \alpha,\varphi}$.

\begin{lemma}\label{vertLines}
Fix $k\geq 1$ and consider positive $\tilde a_1,\ldots, \tilde a_N$. Then, for any $\varphi\in (0,\pi/4)$,
\begin{equation*}
\mu_k= \frac{(-1)^k q^{\frac{k(k-1)}{2}}}{(2\pi \I)^k} \int \cdots \int \prod_{1\leq A<B\leq k} \frac{z_A-z_B}{z_A-qz_B} \frac{f(z_1)\cdots f(z_k)}{z_1\cdots z_k} dz_1\cdots dz_k,
\end{equation*}
with $f(z) = \exp((q-1)\gamma z)\prod_{i=1}^{N}\frac{\tilde a_i}{\tilde a_i-z}$ and the $z_j$-contours given by $\CwPre{\tilde \alpha_j,\varphi}$ where $\{\tilde \alpha_j\}_{j=1}^{N}$ are positive real numbers such that $\tilde \alpha_j<q\tilde \alpha_{j+1}$ for all $1\leq j\leq N-1$ and \mbox{$\tilde \alpha_N<\min_i\{\tilde a_i\}$}.
\end{lemma}
\begin{proof}
The formula in (\ref{mukdef}) for $\mu_k$ involves closed contours for the $z_j$ which can be chosen to be concentric circles centered on the real axis and with their left crossing point with the real axis a $\tilde \alpha_j$ (the $z_A$ contours containing $q$ times the $z_B$ contour for all $A<B$ as well as the poles at the $\tilde a_i$'s). We now proceed to expand these contours to the infinite contours $\CwPre{\tilde \alpha_j,\varphi}$, one at a time. We start with $z_1$. We may freely deform $z_1$ to a pie slice shape with radius $r\gg 1$ made up of the portion of $\CwPre{\tilde \alpha_1,\varphi}$ with distance from $\alpha_1$ less than $r$, and then a circular arc through the real axis: precisely the contour is $\{\tilde \alpha_1+e^{\I \varphi}t\}_{0\leq t\leq r}\cup \{\tilde \alpha_1+e^{-\I \varphi}t\}_{0\leq t\leq r}\cup \{\tilde\alpha + e^{\I \sigma}r\}_{-\varphi\leq \sigma\leq \varphi}$. We would like to replace this by the infinite contour $\CwPre{\tilde \alpha_1,\varphi}$. To justify this observe that for $z$ with positive real part, $f(z)$ decays exponentially with respect to the real part of $z$. Since $\varphi\in (0,\pi/4)$, as $r$ goes to infinity, the contributions of the contour integral away from the origin along both the pie slice and the limiting contour become negligible, and hence the replacement is justified. This procedure can be repeated for $z_2$ and so on until the contours are as claimed in the lemma. Note that as the original contours were positively oriented circles, and the expanded sections were on the left of the circles, the contours are now oriented so as to have decreasing imaginary part.
\end{proof}

In order to combine the $q$-moments $\mu_k$ into a Fredholm determinant it is useful to have all of the contours coincide. This, however, involves keeping track of the residues which result from such a deformation. First, define the $q$-Pochhammer symbol and $q$-factorial as 
\begin{equation*}
(a;q)_{n} = \prod_{i=0}^{n-1} (1-q^i a), \qquad k_q!:= \frac{(q;q)_{k}}{(1-q)^k}.
\end{equation*}

\begin{proposition}\label{mukprop}
Fix $k\geq 1$ and consider a meromorphic function $f(z)$ with $N$ poles $\tilde a_{1},\ldots,\tilde a_N$ with positive real part. Then setting
\begin{equation*}
\mu_k=\frac{(-1)^k q^{\frac{k(k-1)}{2}}}{(2\pi \I)^k} \int \cdots \int \prod_{1\leq A<B\leq k} \frac{z_A-z_B}{z_A-qz_B} \frac{f(z_1)\cdots f(z_k)}{z_1\cdots z_k} dz_1\cdots dz_k,
\end{equation*}
we have
\begin{equation*}
\begin{aligned}
\mu_k =& k_q! \sum_{\substack{\lambda\vdash k\\ \lambda=1^{m_1}2^{m_{2}}\cdots}} \frac{1}{m_1!m_2!\cdots} \\
 & \times \frac{(1-q)^{k}}{(2\pi \I)^{\ell(\lambda)}} \int \cdots \int \det\left[\frac{1}{w_i q^{\lambda_i}-w_j}\right]_{i,j=1}^{\ell(\lambda)} \prod_{j=1}^{\ell(\lambda)}  f(w_j)f(qw_j)\cdots f(q^{\lambda_j-1}w_j) dw_j,
\end{aligned}
\end{equation*}
with the $z_j$-contours given by $\CwPre{\tilde \alpha_j,\varphi}$ where $\{\tilde \alpha_j\}_{j=1}^{k}$ are positive real numbers such that \mbox{$\tilde \alpha_j<q\tilde \alpha_{j+1}$} for all $1\leq j\leq k-1$ and $\tilde \alpha_k<\min_i\{\tilde a_i\}$, and where the $w_j$-contours are all the same and given by $\CwPre{\tilde \alpha,\varphi}$ for any $\tilde \alpha\in (0,\min_i\{\tilde a_i\})$. The notation $\lambda \vdash k$ above means that $\lambda$ partitions $k$ (i.e., if $\lambda=(\lambda_1,\lambda_2,\ldots)$ then $k=\sum \lambda_i$), and the notation $\lambda = 1^{m_1}2^{m_2}\cdots$ means that $i$ shows up $m_i$ times in the partition $\lambda$.
\end{proposition}
\begin{proof}
Word-for-word repetition of that for Proposition~3.2.1 of~\cite{BC11}.
\end{proof}

\subsubsection{A first Fredholm determinant}

Before stating our first Fredholm determinant in this derivation, let us recall how such determinants are defined. Fix a Hilbert space $L^2(X,\mu)$ where $X$ is a measure space and $\mu$ is a measure on $X$. When $X=\mathcal{C}$ is a simple smooth contour in $\C$, we write $L^2(\mathcal{C})$ where $\mu$ is understood to be the path measure along $\mathcal{C}$ divided by $2\pi \I$. When $X$ is the product of a discrete set $D$ and a contour $\mathcal{C}$, $\mu$ is understood to be the product of the counting measure on $D$ and the path measure along $\mathcal{C}$ divided by $2\pi \I$. Let $K$ be an {\it integral operator} acting on $f(\cdot)\in L^2(X,\mu)$ by $Kf(x) = \int_{X} K(x,y)f(y) d\mu(y)$. $K(x,y)$ is called the {\it kernel} of $K$. One way of defining the Fredholm determinant of $K$, for trace class operators $K$, is via the Fredholm series
\begin{equation}\label{eqFredholm}
\det(\Id+K)_{L^2(X,\mu)} = 1+\sum_{n=1}^{\infty} \frac{1}{n!} \int_{X} \cdots \int_{X} \det\left[K(x_i,x_j)\right]_{i,j=1}^{n} \prod_{i=1}^{n} d\mu(x_i).
\end{equation}
In fact, if an operator $K$ is such that the above right-hand side series is absolutely convergent, then we still write it as $\det(\Id+K)$ as short-hand even if $K$ is not trace class. This is sufficient for our purposes as we will deal directly with the expansion. We will often write only $\det(\Id+K)_{L^2(X)}$ for $\det(\Id+K)_{L^2(X,\mu)}$.

Our first Fredholm determinant formula now follows.

\begin{proposition}\label{abgCor}
Fix $k\geq 1$ and consider positive $\tilde a_1,\ldots,\tilde a_N$. Then for any $\varphi\in (0,\pi/4)$ and any $\tilde\alpha \in (0,\min_{i}\{\tilde a_i\})$, there exists a positive constant $C\geq 1$ such that for all $|\zeta|<C^{-1}$,
\begin{equation}\label{deteqn}
\left\langle \frac{1}{\left(\zeta q^{\lambda_N};q\right)_{\infty}}\right\rangle_{\MM_{t=0}(\tilde a_1,\ldots,\tilde a_N;\rho)} = \det(\Id+K)_{L^2(\Zgzero\times\CwPre{\tilde \alpha,\varphi})}
\end{equation}
where the kernel $K$ is given by
\begin{equation}\label{abgKernel}
K(n_1,w_1;n_2,w_2) = \frac{\zeta^{n_1} f(w_1)f(qw_1)\cdots f(q^{n_1-1}w_1)}{q^{n_1} w_1 - w_2}
\end{equation}
with
\begin{equation*}
f(w) =  \exp((q-1)\gamma w)\prod_{m=1}^{N} \frac{\tilde a_m}{\tilde a_m-w}.
\end{equation*}
Recall that the contour $\CwPre{\tilde \alpha,\varphi}$ is given in Definition~\ref{CwPredef}.
\end{proposition}

The left-hand-side of (\ref{deteqn}) is the $q$-Laplace transform of the random variable $q^{\lambda_N}$ (with respect to the $e_q$ exponential) --- see Section 3.1.1 of~\cite{BC11}. The $q$-Binomial theorem is analogous to Taylor expansion of the $e_q$ exponential. Because $q^{\lambda_N}<1$ it we can justify interchanging the summation (of the series expansion resulting from the $q$-Binomial theorem) and the expectation and hence we find that the $q$-Laplace transform can be written as a generating series of $q$-moments. Using the expression for these moments coming from Proposition~\ref{mukprop} it is an easy formal manipulation of terms in a summation to turn this generating series into the desired Fredholm determinant. A little extra arguing shows that these formal manipulations are numerical equalities and proves the claimed result. The details are given in Section~\ref{abgCorproof}.

\subsubsection{A Fredholm determinant suitable for asymptotics}
By using a Mellin-Barnes type integral representation and analytic continuation we can reduce our Fredholm determinant to that of an operator acting on a single contour. The above developments all lead to the following result.

\begin{theorem}\label{PlancherelfredThm}
Fix $\rho$ a Plancherel (see equation (\ref{tag1})) Macdonald nonnegative specialization and positive $\tilde a_1,\ldots, \tilde a_N$. Then for all $\zeta\in \C\setminus \Rplus$
\begin{equation}\label{thmlaplaceeqn}
\left\langle \frac{1}{\left(\zeta q^{\lambda_N};q\right)_{\infty}}\right\rangle_{\MM_{t=0}(\tilde a_1,\ldots,\tilde a_N;\rho)} = \det(\Id+\tilde K_{\zeta})_{L^2(\CwPre{\tilde \alpha,\varphi})}
\end{equation}
where $\CwPre{\tilde \alpha,\varphi}$ as in Definition~\ref{CwPredef} with any $\tilde \alpha\in (0,\min_i\{\tilde a_i\})$ and any $\varphi\in (0,\pi/4)$. The operator $\tilde K_{\zeta}$ is defined in terms of its integral kernel
\begin{equation*}
\tilde K_{\zeta}(w,w') = \frac{1}{2\pi \I}\int_{\CsPre{w}} \Gamma(-s)\Gamma(1+s)(-\zeta)^s g_{w,w'}(q^s)ds
\end{equation*}
where
\begin{equation}\label{gwwprimeeqn}
g_{w,w'}(q^s) = \frac{\exp\big(\gamma w(q^{s}-1)\big)}{q^s w - w'} \prod_{m=1}^{N} \frac{(q^{s}w/\tilde a_m;q)_{\infty}}{(w/\tilde a_m;q)_{\infty}},
\end{equation}
and the contour $\CsPre{w}$ is as in Definition~\ref{CwPredef}.
\end{theorem}

The details of the proof of this result are given in Section~\ref{PlancherelfredThmproof}.

\subsection{Weak convergence to the Whittaker measure}

We are now almost prepared to relate the above discussion to the Fredholm determinant in Theorem~\ref{NeilPolymerFredDetThm}. The connection relies on the following weak convergence of probability measures result.

\begin{theorem}[Theorem~4.1.20 of~\cite{BC11}]\label{theorem26}
Fix \mbox{$N\geq 1$}, a drift vector \mbox{$a=(a_1,\ldots, a_N)\in \R^N$}, and a time parameter $\tau>0$. For positive $\tilde a_1,\dots,\tilde a_N$ and $\gamma>0$, consider a partition \mbox{$\lambda=(\lambda_1,\ldots,\lambda_N)$} distributed according to the $q$-Whittaker measure (i.e., Macdonald measure at $t=0$) $\MM_{t=0}(\tilde a_1,\dots,\tilde a_N;\rho)$ with Plancherel specialization $\rho$ determined by $\gamma$. For a small parameter $\e>0$ let
\begin{equation*}
q=e^{-\epsilon},\quad \gamma=\tau \epsilon^{-2}, \quad \tilde a_j=e^{-\epsilon a_j}, \quad  \lambda_j=\tau\epsilon^{-2}-(N+1-2j)\epsilon^{-1}\ln\epsilon+T_{j}\epsilon^{-1}, \quad 1\le j\le N.
\end{equation*}
Consider the $\e$-indexed measure induced on $\{T_{j}\}_{1\leq j\leq N}$. This measure weakly converges, as $\e\to 0$, to the Whittaker measure $\WM{a_1,\ldots, a_N,\tau}$ on $\{T_{j}\}_{1\leq j\leq N}$.
\end{theorem}

\subsection{Proof of Theorem~\ref{NeilPolymerFredDetThm}}\label{Neilpolyproof}
We may now combine the above results to provide a proof of Theorem~\ref{NeilPolymerFredDetThm}. We follow approach taken in proving Theorem~4.1.40 of~\cite{BC11} but due to the unboundedness of the contours with which we are dealing, there are a number of extra and somewhat involved estimates we must make. These are stated as propositions and proved in Sections~\ref{prop1sec} and~\ref{prop2sec}.

The proof splits into two parts. \emph{Step 1:} We prove that the left-hand side of equation (\ref{thmlaplaceeqn}) of Theorem~\ref{PlancherelfredThm} converges to $\left\langle e^{-u e^{-T_{N}}}  \right\rangle_{\WM{a_1,\ldots, a_N;\tau}}$. This relies on combining Theorem~\ref{theorem26} (which provides weak convergence of the $q$-Whittaker measure to the Whittaker measure) with Lemma~\ref{problemma2} and the fact that the $q$-Laplace transform converges to the usual Laplace transform. \emph{Step 2:} We prove that the Fredholm determinant expression coming from the right-hand side of Theorem~\ref{PlancherelfredThm} converges to the Fredholm determinant given in the theorem we are presently proving (see~\ref{kvvprimeC}).

In accordance with the scalings of Theorem~\ref{theorem26}, we scale the parameters of Theorem~\ref{PlancherelfredThm} as
\begin{equation}\begin{aligned}\label{scalings}
&q =e^{-\e}, \qquad\gamma =\tau\e^{-2}, \qquad \tilde a_k=e^{-\e a_k},\quad 1\leq k\leq N\\
&w= q^v, \qquad \zeta = -\e^{N} e^{\tau \e^{-1}} u, \qquad \lambda_N = \tau \e^{-2} + (N-1) \e^{-1} \ln \e + T_N \e^{-1}.
\end{aligned}\end{equation}

\subsubsection*{Step 1:}
We assume throughout that $u\in \C$ with positive real part. Rewrite the left-hand side of equation (\ref{thmlaplaceeqn}) in  Theorem~\ref{PlancherelfredThm} as
\begin{equation*}
\left\langle \frac{1}{\left(\zeta q^{\lambda_N};q\right)_{\infty}}\right\rangle_{\MM_{t=0}(\tilde a_1,\ldots,\tilde a_N;\rho)}  = \left\langle e_q(x_q)\right\rangle_{\MM_{t=0}(\tilde a_1,\ldots,\tilde a_N;\rho)}
\end{equation*}
where
\begin{equation*}
x_q=(1-q)^{-1}\zeta q^{\lambda_N} =-ue^{-T_N} \e/(1-q)
\end{equation*}
and
\begin{equation*}
e_q(x)=\frac{1}{((1-q)x;q)_{\infty}}
\end{equation*}
is a $q$-exponential. Combine this with the fact that $e_q(x)\to e^{x}$ uniformly on $\Re(x)<0$ to show that, considered as a function of $T_N$, $e_q(x_q)\to e^{-u e^{-T_N}}$ uniformly for $T_N\in \R$. By Theorem~\ref{theorem26}, the measure on $T_N$ (induced from the $q$-Whittaker measure on $\lambda_N$) converges weakly in distribution as $\e\to 0$ to the marginal of the Whittaker measure $\WM{a_1,\ldots,a_N;\tau}$ on the $T_N$ coordinate. Combining this weak convergence with the uniform convergence of $e_q(x_q)$ and Lemma~\ref{problemma2} gives that
\begin{equation*}
\left\langle \frac{1}{\left(\zeta q^{\lambda_N};q\right)_{\infty}}\right\rangle_{\MM_{t=0}(\tilde a_1,\ldots,\tilde a_N;\rho)} \to \left\langle e^{-u e^{-T_{N}}}  \right\rangle_{\WM{a_1,\ldots,a_N;\tau}}
\end{equation*}
as $q\to 1$, or equivalently $\e \to 0$.

\subsubsection*{Step 2:}
Recall the kernel in the right-hand side of equation (\ref{thmlaplaceeqn}) in  Theorem~\ref{PlancherelfredThm}. It can be rewritten as a Fredholm determinant of a kernel with the variables $v$ and $v'$ (recall $w=q^v$) as follows:
\begin{equation*}
\det(\Id+\tilde K_{\zeta}) = \det(\Id+K_u^{\e}).
\end{equation*}
Here the $L^2$ space with respect to which this determinant is defined is that of the contour specified in Definition~\ref{imagedef} below. The kernel is denoted by $K_u^{\e}$ to denote the dependence on $u$ (through $\zeta = -\e^{N} e^{\tau \e^{-1}} u$) and $\e$ (through $q^{-\e}$). It is given by
\begin{equation}\label{445}
K_u^{\e}(v,v') = \frac{1}{2\pi \I}\int_{\CsPre{q^v}}h^q(s) ds,
\end{equation}
where
\begin{equation}\label{446}
h^q(s)=\Gamma(-s)\Gamma(1+s)\left(\frac{-\zeta}{(1-q)^N}\right)^s \frac{q^v \ln q}{q^{s+v} - q^{v'}}  e^{\gamma q^v(q^{s}-1)} \prod_{m=1}^{N} \frac{\Gamma_q(v- a_m)}{\Gamma_q(s+v- a_m)}
\end{equation}
where the new term $q^v \ln q$ came from the Jacobian of changing $w$ to $v$ and where the \mbox{$q$-Gamma} function is defined as
\begin{equation*}
\Gamma_q(x) = \frac{(q;q)_{\infty}}{(q^x;q)_{\infty}} (1-q)^{1-x}.
\end{equation*}

The contour on which this kernel $K_u^{\e}$ acts is the image of the contour $\CwPre{\tilde \alpha,\varphi}$ under the map $x\mapsto \ln_q x$. There is a fair amount of freedom in specifying this contour, so we will fix a particular such pre-image contour.
\begin{definition}\label{imagedef}
Let $\alpha = 1+\max a_i$ then we define the contour $\CvEps{\alpha,\varphi}$ as the image of \mbox{$q^{\alpha} + e^{\pm \varphi\I}\Rplus $} under the map $x\mapsto \ln_{q} x$. This contour is illustrated in Figure~\ref{whitasymcontours}.
\end{definition}
We will assume that $K_u^{\e}$ acts on the contour $\CvEps{\alpha,\varphi}$. Note that as $\e\to 0$ this contour converges locally uniformly to $\alpha + e^{(\pi \pm \varphi)\I}\Rplus $ as can readily be seen by Taylor expanding the map $x\mapsto \ln_{q} x$.

It follows from the above observation that the contour on which the kernel $K_u^{\e}$ is defined converges as $\e\to 0$ to the contour a $\Cv{\alpha,\varphi}$ on which the kernel in Theorem~\ref{NeilPolymerFredDetThm} is defined. Let us likewise demonstrate the pointwise convergence of the integrand in the integral (\ref{445}) defining kernel $K_u^{\e}$ to that of the kernel $K_u$.

Consider the behavior of each term as $q\to 1$ (or equivalently as $\e\to 0$ as $q=e^{-\e}$):
\begin{align}
e^{-\tau s \e^{-1}}\left(\frac{-\zeta}{(1-q)^N}\right)^s & \to u^s,\label{pwlimits1}\\
\frac{q^v \ln q}{q^{s+v} - q^{v'}} &\to \frac{1}{v+s-v'}, \label{pwlimits2}\\
\frac{\Gamma_q(v-a_m)}{\Gamma_q(v+s-a_m)} & \to \frac{\Gamma(v-a_m)}{\Gamma(s+v-a_m)},\label{pwlimits3}\\
e^{\tau s \e^{-1}} \exp\left(\gamma q^v(q^{s}-1)\right) & \to e^{v \tau s + \tau s^2/2}.\label{pwlimits4}
\end{align}
Combining these pointwise limits together gives the integrand of the kernel $K_u$ given in (\ref{kvvprimeC}). However, in order to prove convergence of the determinants, or equivalently the Fredholm expansion series, one needs more than just this straightforward pointwise convergence.

There are four things we must do to complete \emph{Step 2} and prove convergence of the determinants. In proving convergence of Fredholm determinants it is convenient to have the contour on which the operator acts be fixed with $\e$.

In \emph{Step 2a} we deform $\CvEps{\alpha,\varphi}$ to a contour $\CvEpsR{\alpha,\varphi,r}$ with a portion $\CvRLeq{\alpha,\varphi,<r }$ (of distance $<r$ to the origin) which coincides with the limiting contour $\Cv{\alpha,\varphi}$.

Then in \emph{Step 2b} we show that for any fixed $\eta>0$, by choosing $\e_0$ small enough and $r_0$ large enough, for all $\e<\e_0$ and $r>r_0$ the determinant restricted to $L^2(\CvRLeq{\alpha,\varphi,<r})$ differs from the entire determinant on $L^2(\CvEpsR{\alpha,\varphi,r})$ by less than $\eta$. Thus, at an arbitrarily small cost of $\eta$ we can restrict to a sufficiently large radius on which the contour is independent of $\e$.

In \emph{Step 2c} we show that for any $\eta>0$, for $\e$ small, the Fredholm determinant of $K_u^{\e}$ restricted to $L^2(\CvRLeq{\alpha,\varphi,<r})$ differs by at most $\eta$ from the Fredholm determinant of $K_u$ restricted to the same space.

Finally, \emph{Step 2d} shows that for $r_0$ large enough, for all $r>r_0$ the Fredholm determinant of $K_u$ restricted to $L^2(\CvRLeq{\alpha,\varphi,<r})$ differs from the Fredholm determinant of $K_u$ on $L^2(\Cv{\alpha,\varphi})$ by at most $\eta$. Summing up the steps, we deform the contour, we cut the contour to be finite, we take the $\e\to 0$ limit, and then we repair the contour to its final form -- all at cost $3\eta$ for $\eta$ arbitrarily small.

\subsubsection*{Step 2a:} We must define the contour to which we want to deform $\CvEps{\alpha,\varphi}$, and then justify this deformation as not changing the value of the Fredholm determinant.

\begin{definition}\label{cpctcontdef}
Fix $\varphi\in (0,\pi/4)$ and $r>0$. For $\alpha\in \R$, define the finite contour $\CvRLeq{\alpha,\varphi,<r}$ to be $\{\alpha+te^{(\pi\pm \varphi)\I}:0\leq t\leq r\}$. The maximal imaginary part along $\CvRLeq{\alpha,\varphi,<r}$ is $r\sin(\varphi)$. Define the infinite contour $\CvEpsR{\alpha,\varphi,r}$ to be the union of $\CvRLeq{\alpha,\varphi,<r}$ with $\CvEpsRGeq{\alpha,\varphi,>r}$ and  $\CvEpsReq{\alpha,\varphi,=r}$. Here, the contour
$\CvEpsRGeq{\alpha,\varphi,>r}$ is the portion of the contour $\CvEps{\alpha,\varphi}$ which has imaginary part exceeding $r\sin(\varphi)$ in absolute value; and the contour $\CvEpsReq{\alpha,\varphi,=r}$ is composed of the two horizontal line segments which join $\CvRLeq{\alpha,\varphi,<r}$ with $\CvEpsRGeq{\alpha,\varphi,>r}$. These contours are illustrated in Figure~\ref{whitasymcontours}.
\end{definition}

\begin{figure}
\begin{center}
\psfrag{alpha}[cb]{$\alpha$}
\psfrag{C}[lb]{$\Cv{\alpha,\varphi}$}
\psfrag{Ce}[lb]{$\CvEps{\alpha,\varphi}$}
\psfrag{C1}[lb]{$\CvEpsR{\alpha,\varphi,>r}$}
\psfrag{C2}[lb]{$\CvEpsR{\alpha,\varphi,=r}$}
\psfrag{C3}[lb]{$\CvEpsR{\alpha,\varphi,<r}$}
\includegraphics[height=5cm]{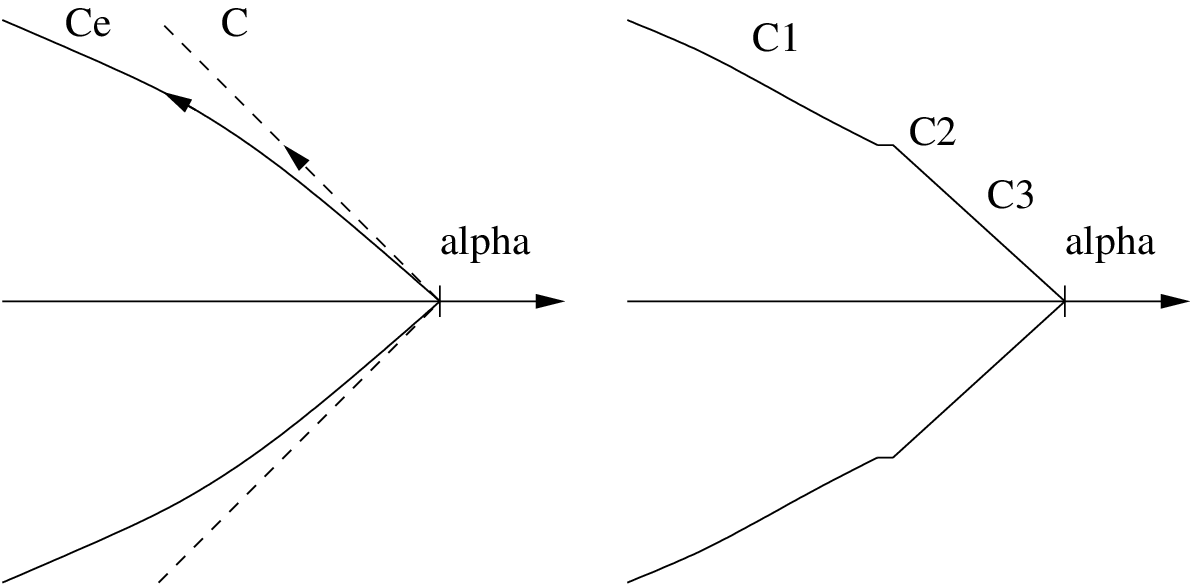}
\caption{Left: The infinite contour $\CvEps{\alpha,\varphi}$ and the limiting contour $\Cv{\alpha,\varphi}$. Right: The infinite contour $\CvEpsR{\alpha,\varphi,r}$ (which we deform from $\CvEps{\alpha,\varphi}$).}
\label{whitasymcontours}
\end{center}
\end{figure}

Now we justify replacing the contour $\CvEps{\alpha,\varphi}$ by $\CvEpsR{\alpha,\varphi,r}$.

\begin{lemma}
For any $r>0$ there exists $\e_0>0$ such that for all $\e<\e_0$,
\begin{equation*}
\det(\Id+K_u^{\e})_{L^2(\CvEps{\alpha,\varphi})} = \det(\Id+K_u^{\e})_{L^2(\CvEpsR{\alpha,\varphi,r})}.
\end{equation*}
\end{lemma}
\begin{proof}
The two contours differ only by a finite length modification. We can continuously deform between the two contours. We will employ Lemma~\ref{TWprop1} which says that as long as the kernel is analytic in a neighborhood of the contour as we continuously deform then the Fredholm determinant remains unchanged throughout the deformation. The only things which could threaten the analyticity of the kernel are the poles coming from the left-hand side terms of (\ref{pwlimits2}) and (\ref{pwlimits3}).
On account of the condition satisfied by the contour $\CsPre{q^v}$ (see Definition~\ref{CwPredef}), it follows that these poles are avoided. By choosing $\e$ small enough, the two contours we are deforming between can be made as close as desired. Taking them close enough ensures it is possible then to deform between them while avoiding poles of the kernel in $v$ or $v'$ -- hence proving the lemma.
\end{proof}

\subsubsection*{Step 2b:}
We must now show that we can, with small error, restrict our Fredholm determinant to acting on the finite, fixed contour $\CvRLeq{\alpha,\varphi,<r}$. This requires us choosing $r>r_0$ for $r_0$ large enough, and also choosing $\e<\e_0$ for $\e_0$ small enough.

\begin{proposition}\label{compactifyprop}
Fix $\varphi\in (0,\pi/4)$. For any $\eta>0$ there exists $r_0>0$ and $\e_0>0$ such that for all $r>r_0$ and $\e<\e_0$
\begin{equation*}
\left|\det(\Id+ K_u^{\e})_{L^2(\CvEpsR{\alpha,\varphi,r})} - \det(\Id+ K_u^{\e})_{L^2(\CvRLeq{\alpha,\varphi,<r})}\right| \leq \eta.
\end{equation*}
\end{proposition}
The proof of this proposition is fairly technical and is given in Section~\ref{prop2sec}.

\subsubsection*{Step 2c:} Having restricted our attention to the finite, unchanging (with $\e$) contour $\CvRLeq{\alpha,\varphi,<r}$ we may now take the limit of Fredholm determinant on the restricted $L^2$ space as $\e\to 0$.

\begin{proposition}\label{finiteSprop}
Fix $\varphi\in(0,\pi/4)$. For any $\eta>0$ and any $r>0$ there exists $\e_0>0$ such that for all $\e<\e_0$
\begin{equation*}
\left|\det(\Id+ K_u^{\e})_{L^2(\CvRLeq{\alpha,\varphi,<r})} - \det(\Id+ K_u)_{L^2(\CvRLeq{\alpha,\varphi,<r})}\right|\leq \eta
\end{equation*}
where $K_u(v,v')$ is given by the integral (\ref{kvvprime}).
\end{proposition}
The proof of this proposition is also fairly technical and is given in Section~\ref{prop1sec}.

\subsubsection*{Step 2d:} Finally, we show that post-asymptotics we can return to the simple infinite contour $\Cv{\alpha,\varphi}$.

\begin{proposition}\label{postasymlemma}
Fix $\varphi\in(0,\pi/4)$. For any $\eta>0$ there exists $r_0>0$ such that for all $r>r_0$
\begin{equation*}
\left|  \det(\Id+ K_u)_{L^2(\CvRLeq{\alpha,\varphi,<r})} -  \det(\Id+ K_u)_{L^2(\Cv{\alpha,\varphi})}\right|\leq \eta.
\end{equation*}
\end{proposition}
The proof of this proposition is given in Section~\ref{prop3sec}. It is a fair amount more straight forward than the previous two proofs and hence is given first.

Having completed the four substeps we may combine Propositions~\ref{compactifyprop},~\ref{finiteSprop} and~\ref{postasymlemma} to show that for any $\eta>0$, there exists $\e_0>0$ such that for all $\e<\e_0$,
\begin{equation*}
\left|\det(\Id+\tilde K_{\zeta})_{L^2(\CwPre{\tilde \alpha,\varphi})} -  \det(\Id+ K_u)_{L^2(\Cv{\alpha,\varphi})} \right| \leq 3\eta
\end{equation*}
where $\det(\Id+\tilde K_{\zeta})$ is as in the right-hand side of equation (\ref{thmlaplaceeqn}) in  Theorem~\ref{PlancherelfredThm}, subject to the scalings given in (\ref{scalings}). Since $\eta$ is arbitrary this shows that
\begin{equation*}
\lim_{\e \to 0} \det(\Id+\tilde K_{\zeta})_{L^2(\CwPre{\tilde \alpha,\varphi})} = \det(\Id+ K_u)_{L^2(\Cv{\alpha,\varphi})}.
\end{equation*}

The above result completes the proof of Theorem~\ref{NeilPolymerFredDetThm} modulo proving Propositions~\ref{compactifyprop},~\ref{finiteSprop} and~\ref{postasymlemma}.

\section{Details in the proof of Theorem~\ref{ThmPosTempAsy}}\label{proofOConnellYorKPZ}
As discussed in Section~\ref{formalOConnellYorKPZ}, to finish the proof of Theorem~\ref{ThmPosTempAsy} we need to show Theorem~\ref{PFThmF2pert}. For $\kappa>0$, Definition~\ref{digammadef} associates the scaling parameters $\theta^\kappa$, $f^\kappa$ and $c^\kappa$ which appear in the statement of this result. The variable $\kappa$ and $\theta$ are dual in the sense that one could instead start with some fixed $\theta>0$ and then associated scaling parameters $\kappa_{\theta}$, $f_{\theta}$ and $c_{\theta}$. In particular, for $\theta=\theta^{\kappa}$ one recovers $f_{\theta}=f^\kappa$ and $c_{\theta}=c^\kappa$. In the proof it is more natural to parameterize everything by $\theta$ instead of $\kappa$, so we will do it.

First we prove the convergence to the GUE Tracy-Widom distribution without boundary perturbations, since the proof with boundary perturbations is a small modification of it.

\subsection{Proof of Theorem~\ref{PFThmF2pert}(a)}\label{sect5.1}
We first give explicit expansions for some of the functions from Definition~\ref{digammadef}.
Let $\Psi(z)=\frac{d}{dz} \ln(\Gamma(z))$ be the Digamma function and fix $\theta\in (0,\infty)$. Then
\begin{equation*}
\Psi(z)=-\gamma_{\rm E}+\sum_{n=0}^\infty \left(\frac{1}{n+1}-\frac{1}{n+z}\right),
\end{equation*}
where $\gamma_{\rm E}$ is the Euler constant. Hence,
\begin{align}
\kappa_\theta&=\Psi'(\theta)=\sum_{n=0}^\infty \frac{1}{(\theta+n)^2}, \label{PFeqKappa} \\
f_\theta&=\theta \Psi'(\theta)-\Psi(\theta)=\gamma_{\rm E} +\sum_{n=0}^\infty \left(\frac{n+2\theta}{(n+\theta)^2}-\frac{1}{n+1}\right), \label{PFeqF}\\
c_\theta&=(-\Psi''(\theta)/2)^{1/3}=\bigg(\sum_{n=0}^\infty \frac{1}{(n+\theta)^3}\bigg)^{1/3}. \label{PFeqCtheta}
\end{align}
Under the scaling limit
\begin{equation*}
\tau=\kappa_\theta N,\quad u=e^{-N f_\theta - r c_\theta N^{1/3}}.
\end{equation*}
we have to show the following: For $K_u$ as in (\ref{kvvprime}) and a contour $\mathcal{C}_v:=\Cv{0,\varphi}$,
\begin{equation*}
\lim_{N\to\infty}\det(\Id+K_u)_{L^2(\mathcal{C}_v)}= \det(\Id-K_{\rm Ai})_{L^2(r,\infty)}.
\end{equation*}

To show this we start with the kernel (\ref{kvvprime}), replace $\Gamma(-s)\Gamma(1+s) = -\pi / \sin(\pi s)$ and then perform the change of variable $\tilde z=s+w$ to obtain
\begin{equation*}
K_{u}(v,v') = \frac{-1}{2\pi \I}\int d\tilde z\, \frac{\pi}{\sin (\pi(\tilde z-v))}\frac{e^{N G(v)-N G(\tilde z)} e^{r N^{1/3} (v-\tilde z)}}{\tilde z-v'}.
\end{equation*}
where
\begin{equation*}
G(z) = \ln \Gamma(z) - \kappa \frac{z^2}{2} + f^\kappa z.
\end{equation*}
We will show that the leading contribution to the Fredholm determinant comes for $v,v'$ in a $N^{-1/3}$-neighborhood of $\theta$. Now let us specify the exact choice for the contour $\mathcal{C}_v$ as well as the contour along which $\tilde z$ is integrated. We choose\footnote{Theorem~\ref{OConYorFluctThm} is stated for $\varphi\in (0,\pi/4)$ since one uses the quadratic decay (\ref{eq6.9}) to control the linear term in the bound (\ref{6.12}). For $\varphi=\pi/4$ one gets a linear decay instead of (\ref{eq6.9}) whose strength depends on the parameter $\alpha$ too, it would not strong enough general $\alpha$. However, in our case, with $\alpha=\theta$, it still works, as can be seen from the bound obtained in Proposition~\ref{PFPropBound}. The proof could also be adapted to any other asymptotic direction $0<\varphi<\pi/4$ by simply modifying the path away at a distance greather than some (arbitrary but fixed with $N$) value $R_0$ (one can not employ any angle $\varphi\in (0,\pi/4)$ right away from the critical point since some steep descent properties are then locally not satisfied).}
\begin{equation*}
\mathcal{C}_v:=\{\theta-|y|+\I y, y\in \R\}.
\end{equation*}
$\mathcal{C}_v$ is a steep descent path (see the footnote in Section~\ref{formalcalc1}) for the function $\Re(G(v))$. The path for $\tilde z$ is dependent on $v$, since it has to pass to the left of, or contain the simple poles $v+1,v+2,\ldots$, see Figure~\ref{PFFigPathsTW} (left). Consider the sequence of points $S=\{\Re(v)+1,\Re(v)+2,\ldots\}$. There are three possibilities\label{PFChoiceOfPathZ}:
\begin{itemize}
  \item[(1)] If the sequence $S$ does not contain points in $[\theta,\theta+3c_\theta^{-1} N^{-1/3}]$, then let $\ell\in \N_0$ be such that $\Re(v)+\ell \in [\theta-1,\theta]$ and we set $\tilde\e=c_\theta^{-1} N^{-1/3}$.
  \item[(2)] If the sequence $S$ contains a point in $[\theta,\theta+2c_\theta^{-1} N^{-1/3}]$, then let $\ell\in \N$ such that $\Re(v)+\ell \in [\theta,\theta+2c_\theta^{-1} N^{-1/3}]$ and set $\tilde\e=3 c_\theta^{-1} N^{-1/3}$.
  \item[(3)] If the sequence $S$ contains a point in $(\theta+2c_\theta^{-1} N^{-1/3},\theta+3c_\theta^{-1} N^{-1/3}]$, then let $\ell\in \N$ such that $\Re(v)+\ell \in (\theta-1+2c_\theta^{-1} N^{-1/3},\theta-1+3c_\theta^{-1} N^{-1/3}]$ and set $\tilde\e= c_\theta^{-1} N^{-1/3}$.
\end{itemize}
With this choice, the singularity of the sine along the line $\theta+\tilde\e+\I\R$ is not present, since the poles are at a distance at least $c_\theta^{-1} N^{-1/3}$ from it. Then, the path for $\tilde z$ is given by
\begin{equation*}
\mathcal{C}_{\tilde z}:=\{\theta+\tilde\e+\I y, y\in \R \}\cup \bigcup_{k=1}^\ell B_{v+k},
\end{equation*}
and $B_{v+k}$ denotes a small circle (radius smaller than $1/2$) around $v+k$ and clockwise oriented. If $\ell=0$ then the small circles are simply not present.
The idea behind this choice of the path $\mathcal{C}_{\tilde z}$ is that the $z$-contour consists of a fixed line that is (almost) independent of kernel arguments, and an additional number of little circles (i.e., poles) as needed. Moreover, the leading contribution of the kernel comes only from the cases where $\ell=0$ (i.e., situation (1)) for which $\tilde\e=c_\theta^{-1} N^{-1/3}$.

\begin{figure}
\begin{center}
\psfrag{Phi1}[cb]{$\Phi^{-1}$}
\psfrag{0}[cb]{$0$}
\psfrag{theta}[rb]{$\theta$}
\psfrag{-thetat}[cb]{$-\theta c_\theta N^{1/3}$}
\psfrag{theta+epsilon}[lb]{$\theta+\tilde\e$}
\psfrag{p}[lb]{$p$}
\psfrag{w}[cb]{$w$}
\psfrag{Cw}[lb]{$\mathcal{C}_w$}
\psfrag{v}[cb]{$v$}
\psfrag{Cv}[lb]{$\mathcal{C}_v$}
\psfrag{Cztilde}[lb]{$\mathcal{C}_{\tilde z}$}
\psfrag{Cz}[lb]{$\mathcal{C}_z$}
\includegraphics[height=6cm]{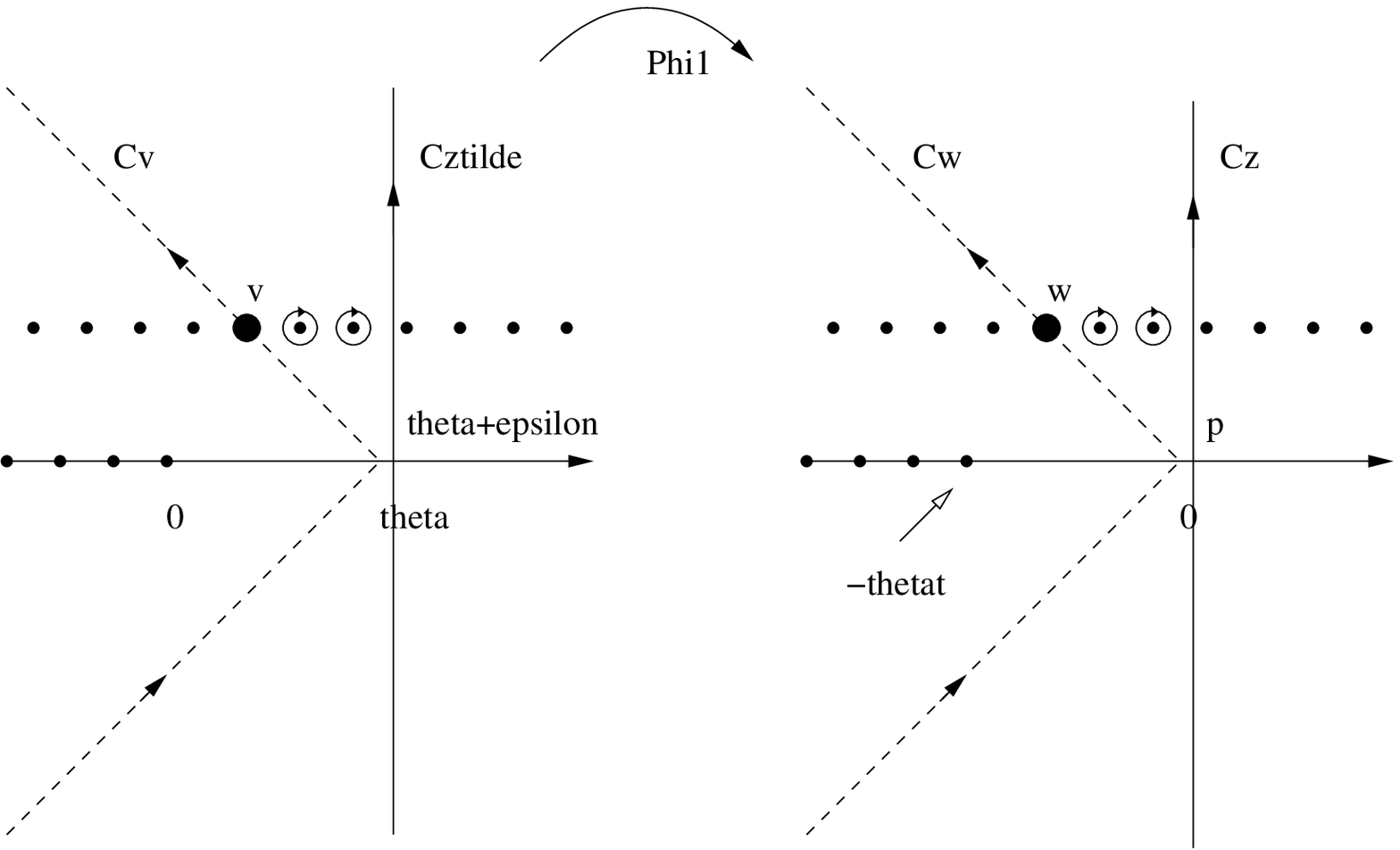}
\caption{Left: Integration paths $\mathcal{C}_v$ (dashed) and $\mathcal{C}_{\tilde z}$ (the solid line plus circles at $v+1,\ldots,v+\ell$). The small black dots are poles either of the sine or of the gamma function. Right:
Integration paths after the change of variables $\mathcal{C}_w$ (dashed) and $\mathcal{C}_z$ (the solid line plus circles at $w+1,\ldots,w+\ell$), with \mbox{$p=p(w)\in\{1,3\}$}.}
\label{PFFigPathsTW}
\end{center}
\end{figure}

Also, we do the change of variable
\begin{equation*}
\{v,v',\tilde z\}=\{\Phi(w),\Phi(w'),\Phi(z)\}\quad \textrm{with}\quad\Phi(z):=\theta+z c_\theta^{-1} N^{-1/3}.
\end{equation*}
After this change of variable, $\det(\Id+K_u)_{L^2(\mathcal{C}_v)}=\det(\Id+K_N)_{L^2(\mathcal{C}_w)}$, the path $\mathcal{C}_v$ becomes (see Figure~\ref{PFFigPathsTW} (right))
\begin{equation}\label{PFeqCw}
\mathcal{C}_w:=\{-|y|+\I y,y\in\R\}
\end{equation}
and the accordingly rescaled kernel
\begin{equation*}
\begin{aligned}
K_N(w,w')&:=c_\theta^{-1} N^{-1/3} K_u(\Phi(w),\Phi(w')) \\
&=\frac{-c_\theta^{-1} N^{-1/3}}{2\pi \I}\int_{\mathcal{C}_z:=\Phi^{-1}(\mathcal{C}_{\tilde z})}dz \frac{\pi e^{N G(\Phi(w))-N G(\Phi(z))}}{\sin(\pi (z-w)c_\theta^{-1} N^{-1/3})} \frac{e^{r (w-z)}}{z-w'}
\end{aligned}
\end{equation*}
where
\begin{equation*}
G(w)=\ln(\Gamma(w))+ f_\theta w - \kappa_\theta w^2/2.
\end{equation*}

In Proposition~\ref{PFPropBound} we show that for any $w,w'\in \mathcal{C}_w$, there exists a constant $C\in (0,\infty)$ such that
\begin{equation*}
|K_N(w,w')|\leq C e^{-|\Im(w)|}
\end{equation*}
uniformly for all $N$ large enough. Therefore,
\begin{equation*}
\left|\det(K_N(w_i,w_j))_{1\leq i,j\leq n}\right|\leq n^{n/2} C^n \prod_{i=1}^n e^{-|\Im(w_i)|}
\end{equation*}
where the factor $n^{n/2}$ is Hadamard's bound. From this bound, it follows that the Fredholm expansion of the determinant,
\begin{equation*}
\det(\Id+K_N)_{L^2(\mathcal{C}_w)} =\sum_{n=0}^\infty \frac{1}{n!} \int_{\mathcal{C}_w} dw_1 \cdots \int_{\mathcal{C}_w} dw_n \det(K_N(w_i,w_j))_{1\leq i,j\leq n},
\end{equation*}
is absolutely integrable and summable. Thus we can by dominated convergence take the $N\to\infty$ limit inside the series, i.e., replace $K_N$ by its pointwise limit,
\begin{equation}\label{PFeq27}
\lim_{N\to\infty} K_N(w,w') = \widetilde K_{\rm Ai}(w,w'):=\frac{-1}{2\pi\I} \int_{e^{-\pi \I/4}\infty}^{e^{\pi \I/4}\infty}dz \frac{e^{z^3/3-w^3/3}e^{r w-r z}}{(z-w)(z-w')},
\end{equation}
derived in Proposition~\ref{PFPropPtwiseConv}, i.e., we have shown that
\begin{equation*}
\lim_{N\to\infty} \det(\Id+K_N)_{L^2(\mathcal{C}_w)} = \det(\Id+\widetilde K_{\rm Ai})_{L^2(\mathcal{C}_w)}.
\end{equation*}
The last part is a standard reformulation, which we report in Lemma~\ref{PFLemTWreformuation}, see also~\cite{TW08b}. This ends the proof of Theorem~\ref{PFThmF2pert}.

\subsection{Pointwise convergence and bounds}
The function $G$ satisfies
\begin{equation}\label{PFeqG123}
G'(\theta)=G''(\theta)=0,\quad G^{(3)}(\theta)=-2\sum_{n=0}^\infty \frac{1}{(n+\theta)^3}=-2c_\theta^{3},\quad G^{(4)}(\theta)=\sum_{n=0}^\infty \frac{6}{(n+\theta)^4},
\end{equation}
therefore $G$ has a double critical point at $\theta$.
For the steep descent analysis we need to analyze the function \mbox{$g(x,y)=\Re(G(x+\I y))$}. It holds\footnote{See for example $\texttt{http://functions.wolfram.com/06.11.19.0001.01}$}
\begin{equation*}
\begin{aligned}
\Re(\ln\Gamma(x+\I y)) & = \sum_{n=1}^\infty \left(\frac{x}{n} - \frac12 \ln\left(\frac{(x+n)^2+y^2}{n^2}\right)\right)
- \gamma_{\rm E} x -\frac12 \ln(x^2+y^2) \\
&=\sum_{n=0}^\infty \left(\frac{x}{n+1} - \frac12 \ln\left((x+n)^2+y^2\right)+\ln(n)\mathbf{1}_{n\geq 1}\right)
- \gamma_{\rm E} x.
\end{aligned}
\end{equation*}
Together with (\ref{PFeqKappa}) and (\ref{PFeqF}) we get
\begin{equation}\label{PFeq21}
\begin{aligned}
g(x,y)&=\Re(\ln\Gamma(x+\I y))+ f_\theta x - \frac12 \kappa_\theta (x^2-y^2)\\
&=\sum_{n=0}^\infty\left(\frac{(n+2\theta)x-(x^2-y^2)/2}{(n+\theta)^2}-\frac12 \ln\left((x+n)^2+y^2\right)+\ln(n)\mathbf{1}_{n\geq 1}\right).
\end{aligned}
\end{equation}
It follows that
\begin{equation}\label{PFeqg1}
g_1(x,y):=\frac{\partial g(x,y)}{\partial x}
=\sum_{n=0}^\infty\left(\frac{n+2\theta-x}{(n+\theta)^2}-\frac{x+n}{(x+n)^2+y^2}\right)
\end{equation}
and
\begin{equation}\label{PFeqg2}
g_2(x,y):=\frac{\partial g(x,y)}{\partial y} =\sum_{n=0}^\infty\left(\frac{y}{(\theta+n)^2}-\frac{y}{(x+n)^2+y^2}\right).
\end{equation}

\begin{proposition}\label{PFPropPtwiseConv}
Uniformly for $w,w'$ in a bounded set of $\mathcal{C}_w$,
\begin{equation}\label{PFeqKernelPtF2}
\lim_{N\to\infty} K_N(w,w') = \frac{-1}{2\pi\I} \int_{e^{-\pi \I/4}\infty}^{e^{\pi \I/4}\infty}dz \frac{e^{z^3/3-w^3/3}e^{r w-r z}}{(z-w)(z-w')}.
\end{equation}
\end{proposition}
\begin{proof}
Consider $w,w'$ in a bounded set of $\mathcal{C}_w$, i.e., the original variables $v,v'$ of order $N^{-1/3}$ around the critical point $\theta$. For $N$ large enough and $w$ bounded, \mbox{$\Re((-w)c_\theta^{-1} N^{-1/3}) \in (0,1)$}, and
$\mathcal{C}_z:=\Phi^{-1}(\mathcal{C}_{\tilde z})=\{1+\I y,y\in \R\}$.
Using (\ref{PFeqG123}) we have the expansion
\begin{equation}\label{PFeq23}
\begin{aligned}
N G(\Phi(w)) &= N G(\theta)-\frac{1}{3} w^3 + \Or(w^4 N^{-1/3})\\
-N G(\Phi(z)) &= -N G(\theta)+\frac13 z^3 -\mu_\theta z^4 N^{-1/3}+\Or(z^5 N^{-2/3})
\end{aligned}
\end{equation}
with $\mu_\theta=G^{(4)}(\theta) c_\theta^{-4}/24>0$ and
\begin{equation}\label{PFeq23b}
\frac{\pi}{\sin(\pi(z-w)c_\theta^{-1} N^{-1/3})}=\frac{c_\theta N^{1/3}}{z-w}(1+\Or((z-w)^2 N^{-1/3})).
\end{equation}
It is also easy to control the $w'$-dependence because $|z-w'|\geq 1$.

Now we divide the integral over $z$ into \mbox{(a) $|\Im(z)|> \delta N^{1/3}$} and \mbox{(b) $|\Im(z)|\leq \delta N^{1/3}$} for some $\delta>0$ which can be taken as small as desired (but independent of $N$).

\smallskip
\emph{(a) Contribution of the integration over $|\Im(z)|> \delta N^{1/3}$.}
We need to estimate
\begin{equation}\label{PFeq33}
\left|\frac{-c_\theta^{-1} N^{-1/3}}{2\pi \I}
\int_{1+\I y,\\|y|> \delta N^{1/3}}
dz \frac{\pi e^{N G(\Phi(w))-N G(\Phi(z))}}{\sin(\pi (z-w)c_\theta^{-1} N^{-1/3})} \frac{e^{r (w-z)}}{z-w'}\right|.
\end{equation}
From (\ref{PFeq23}), (\ref{PFeq23b}), and the fact that $w$ is in a bounded neighborhood of $0$, we have
\begin{equation}\label{PFeq34}
(\ref{PFeq33})\leq \Or(1) \int_{|y|\geq \delta N^{1/3}}
dy\, e^{N \Re(G(\Phi(0))-G(\Phi(1+\I y)))}.
\end{equation}
Setting $\tilde\e=c_\theta^{-1} N^{-1/3}$ and doing the change of variable $\tilde y=y c_\theta^{-1} N^{-1/3}$ we obtain
\begin{equation}\label{PFeq35}
(\ref{PFeq34})\leq \Or(N^{1/3}) \int_{\delta/c_\theta}^\infty d\tilde y\, e^{N (g(\theta+\tilde\e,\tilde y)-g(\theta,0))}
\end{equation}
The function $g(x,y):=\Re(G(x+\I y))$ is given in (\ref{PFeq21}). Finally, in Lemma~\ref{PFlemPathZ} we show that the path $\theta+\tilde\e+\I\R$ is steep descent for the function $-G(\tilde z)$ with derivative of $-\Re(G(\tilde z))$ going to $-\infty$ linearly in $\Im(\tilde z)$. It then follows that (\ref{PFeq35}) is of order $N^{1/3} e^{Ng(\theta,0)-N g(\theta+\tilde\e,0)}e^{-c_1(\delta) N}$ for some positive constant $c_1(\delta)\sim \delta^4$ for small $\delta$.
But
\begin{equation*}
Ng(\theta,0)-N g(\theta+\tilde\e,0)= \frac13+\Or(N^{-1/3}).
\end{equation*}
Thus the contribution of the integration over $|\Im(z)|> \delta N^{1/3}$ is $\Or(e^{-c_2(\delta) N})$ for some positive  constant $c_2(\delta)\sim \delta^4$ for small $\delta$.

\smallskip
\emph{(b) Contribution of the integration over $|\Im(z)|\leq \delta N^{1/3}$.} We need to determine the asymptotics of
\begin{equation}\label{PFeq37}
\frac{-c_\theta^{-1} N^{-1/3}}{2\pi \I}
\int_{1+\I y,\\|y|\leq \delta N^{1/3}}
dz \frac{\pi e^{N G(\Phi(w))-N G(\Phi(z))}}{\sin(\pi (z-w)c_\theta^{-1} N^{-1/3})} \frac{e^{r (w-z)}}{z-w'}.
\end{equation}
Using the expansion (\ref{PFeq23}) and (\ref{PFeq23b}) we get
\begin{equation}\label{PFeq32}
(\ref{PFeq37})=\frac{-1}{2\pi \I} \int_{1-\I\delta N^{1/3}}^{1+\I\delta N^{1/3}}dz \frac{e^{-\mu_\theta z^4 N^{-1/3}}}{(z-w)(z-w')} \frac{e^{z^3/3-rz}}{e^{w^3/3-rw}}(1+\Or((z-w)^2 N^{-1/3}))e^{\Or(w^4 N^{-1/3}; z^5 N^{-2/3})}.
\end{equation}
Denoting $z=1+\I y$ we have
\begin{equation*}
\Re(z^3/3)= -3y^2+1,\quad \Re(z^4)=y^4-6 y^2+1.
\end{equation*}
The convergence of the integral is controlled by $e^{-\mu_\theta y^4 N^{-1/3}-3 y^2}$. One employs the bound \mbox{$|e^{x}-1|\leq |x| e^{|x|}$} with $x=\Or(w^4 N^{-1/3}; z^5 N^{-2/3})$ to control the error terms. Altogether they are only of order $\Or(N^{-1/3})$, i.e., we have obtained
\begin{equation*}
(\ref{PFeq32})=\Or(N^{-1/3})+\frac{-1}{2\pi \I}
\int_{1-\I\delta N^{1/3}}^{1+\I\delta N^{1/3}}dz \frac{e^{-\mu_\theta z^4 N^{-1/3}}}{(z-w)(z-w')} \frac{e^{z^3/3-rz}}{e^{w^3/3-rw}}
\end{equation*}
Finally, we deform the integration contour to the following one: from $\delta N^{1/3} (1-\I)$ to \mbox{$\delta N^{1/3} (1+\I)$}. The error term is again of order $e^{-c_1(\delta) N}$. However, with the new contour, using again $|e^{x}-1|\leq |x| e^{|x|}$ but with \mbox{$x=-\mu_\theta z^4 N^{-1/3}$} one sees that the eliminating the quartic power in $z$ amounts in an error of order $\Or(N^{-1/3})$. The last step is to replace $\delta$ by $\infty$ in the integration boundaries. This leads to an extra error $\Or(e^{-c_3(\delta) N})$ with some positive constant $c_3(\delta)\sim \delta^3$ for small $\delta$.

To summarize, we first choose $\delta$ small enough so that all the $c_k(\delta)>0$. Then for all $N$ large enough we have shown that the contribution of the integration over $|\Im(z)|> \delta N^{1/3}$ is of order $\Or(e^{-c_1(\delta) N})$ and the integration over $|\Im(z)|\leq \delta N^{1/3}$ is given by
\begin{equation*}
\Or(N^{-1/3})+\frac{-1}{2\pi \I}
\int_{e^{-\pi \I/4}\infty}^{e^{\pi\I/4}\infty}dz \frac{1}{(z-w)(z-w')} \frac{e^{z^3/3-rz}}{e^{w^3/3-rw}}.
\end{equation*}
Taking the $N\to\infty$ limit we obtain the result.
\end{proof}

\begin{proposition}\label{PFPropBound}
For any $w,w'\in \mathcal{C}_w$, there exists a constant $C\in (0,\infty)$ such that
\begin{equation*}
|K_N(w,w')|\leq C e^{-|\Im(w)|}
\end{equation*}
uniformly for all $N$ large enough.
\end{proposition}
\begin{proof}
Since the $z$-contour can be chosen such that $|z-w'|\geq 1/2$, we can estimate the absolute value of the factor $(z-w')^{-1}$ by $2$ and discard it from further considerations. For $w$ in a bounded set of $\mathcal{C}_w$, the statement is a consequence of the computations in the proof of Proposition~\ref{PFPropPtwiseConv}. Thus, it is enough to consider $w=-|y|+\I y$ for $y\geq L$, for $L$ which will be chosen large enough (but independent of $N$). In the original variables $v,v'$, this means that we need to consider $v=\theta-|y|+\I y$ for $y\geq L c_\theta^{-1} N^{-1/3}$. Let $v=\Phi(w)$, $v'=\Phi(w')$, then the kernel $K_N$ is given by
\begin{equation*}
K_N(w,w')=\frac{e^{N (G(v)-G(\theta))+r(v-\theta) c_\theta N^{1/3}}}{c_\theta N^{1/3}\, 2\pi\I}\int_{\mathcal{C}_{\tilde z}}d\tilde z \frac{\pi e^{N G(\theta)-N G(\tilde z)} e^{r(\theta-\tilde z) c_\theta N^{1/3}}}{\sin(\pi(\tilde z-v))(\tilde z-v')}.
\end{equation*}
We divide the bound dividing in two contributions: (a) integration over $\theta+\tilde\e+\I\R$, with \mbox{$\tilde\e=p c_\theta^{-1} N^{-1/3}$} (with $p\in\{1,3\}$ depending on the value of $v$, see the proof of Theorem~\ref{PFThmF2pert}(a) above), and (b) integration over the circles $B(v+k)$, $k=1,\ldots,\ell(v)$.

\smallskip
\emph{(a) Integration over $\theta+\tilde\e+\I\R$.} The relevant dependence on $v$ is in the prefactor $e^{N (G(v)-G(\theta))+r(v-\theta) c_\theta N^{1/3}}$ and in the sine. The dependence of $\tilde\e$ on $v$ is marginal, as the needed bounds can be made for any $\tilde\e$ small enough. The estimates as in the proof of Proposition~\ref{PFPropPtwiseConv} imply that this contribution is bounded by
\begin{equation*}
C e^{N (\Re(G(v))-G(\theta))+r(\Re(v)-\theta) c_\theta N^{1/3}} = C e^{N (g(\theta-y,y)-g(\theta,0))-r y c_\theta N^{1/3}},
\end{equation*}
where we used the parametrization $v=\theta-|y|+\I y$ and, by symmetry, considered only $y>0$.
In Lemma~\ref{PFlemPathW} we show that $g(\theta-y,y)$ is strictly decreasing as $y$ increases and for large $y$ the derivative goes to $-\infty$ (logarithmically). Thus, for any fixed $\delta>0$, there exists a constant $c_1>0$ such that for all $y\geq \delta$, $\partial_y g(\theta-y,y)\leq -c_1$. In Lemma~\ref{PFlemPathW} we also show that for small $y$, $g(\theta-y,y)-g(\theta,0)=-\tfrac23 c_\theta^3 y^3+\Or(y^4)$. Therefore, we can choose $\delta>0$ small enough such that:
\begin{enumerate}
\item[(1)] for $L c_\theta^{-1} N^{-1/3}\leq y\leq \delta$, $g(\theta-y,y)\leq g(\theta,0) -\tfrac13 c_\theta^3 y^3$,
\item[(2)] for $y>\delta/2$, $\partial_y g(\theta-y,y)\leq -2 c_1$. It follows $g(\theta-y,y)\leq g(\theta,0)-c_1 y$ for all $y>\delta$.
\end{enumerate}
Replacing $y=\Im(v)=\Im(w)/(c_\theta N^{1/3})$ we get the bounds:
\begin{enumerate}
\item[(1)] for $L\leq \Im(w) \leq \delta c_\theta N^{1/3}$,
\begin{equation*}
C e^{- \Im(w)^3/3-r \Im(w)}\leq C e^{-\Im(w)^3/6}\leq 3 C e^{-\Im(w)}
\end{equation*}
for $L$ large enough (depending on $r$ only).
\item[(2)] for $\Im(w)\geq \delta c_\theta N^{1/3}$,
\begin{equation*}
C e^{-\Im(w) (N^{2/3} c_1/c_\theta+r)}\leq C e^{-\Im(w)}
\end{equation*}
for $N$ large enough.
\end{enumerate}

\smallskip
\emph{(b) Integration over the circles $B(v+k)$, $k=1,\ldots,\ell(v)$.} This happens only if \mbox{$y+3 c_\theta^{-1} N^{-1/3} \geq 1$}, where $y=\Im(v)=\Im(w)/(c_\theta N^{1/3})$. The contribution of the integration over $B_{v+k}$ (up to a $\pm$ sign depending on $k$) is
\begin{equation*}
\frac{e^{N G(v)-N G(v+k)} e^{-rk c_\theta N^{1/3}}}{(v+k-v')}.
\end{equation*}
We have $|v+k-v'|\geq 1/\sqrt{2}$, thus the contribution from the pole at $v+k$ is bounded by
\begin{equation*}
2 e^{N (g(\theta-v,v)-g(\theta-v+k,v))}e^{|r| k c_\theta N^{1/3}}.
\end{equation*}
Define the function $h(v,k):=g(\theta-v,v)-g(\theta-v+k,v)$. In Lemma~\ref{PFlemPoles} we show that $h(v,k)$ is strictly decreasing as a function of $k$, for $k\in [0,y+\tilde\e]$ (we have a positive $\delta$ instead of $\tilde\e$, but for $N$ large enough $\tilde\e<\delta$). Also, $k\leq \ell(v) = \lfloor y+\tilde\e \rfloor$, so that the contribution of the poles at $v+1,\ldots,v+\ell(v)$ is bounded by
\begin{equation}\label{PFeq45}
2 \ell(v) e^{N h(v,1) +|r| \ell(v) c_\theta N^{1/3}}.
\end{equation}
We consider separately the cases (1) $y\leq\theta$ (i.e., $\Re(v)\geq 0$) and (2) $y>0$ (i.e., $\Re(v)<0$).
\begin{enumerate}
\item[(1)] For $y\leq \theta$, from the bound on $\partial_k h(v,k)$, see Lemma~\ref{PFlemPoles}, we get $h(v,1)\leq -y c_\theta^3/4$.
\item[(2)] For $y>\theta$, we know that $h(v,1)<0$ for all $y$ and when $y\to\infty$, $\partial_k h(v,k)|_{k=0}\simeq -y \kappa_\theta$. Since the function $h(v,1)$ is continuous in $y$, there exists a positive constant $c>0$ such that $h(v,1)\leq -c y$ for all $y>\theta$.
\end{enumerate}
Thus, with $c'=\min\{c,c_\theta^3/4\}$ we get
\begin{equation*}
(\ref{PFeq45})\leq e^{-\Im(w) N^{2/3}c'/c_\theta + \Or(1)} \Or(\Im(w) N^{-1/3})\leq C e^{-\Im(w)}
\end{equation*}
for $N$ large enough. This ends the proof of the Proposition.
\end{proof}

Finally let us collect the lemmas on the steep descent properties used in the propositions above.
\begin{lemma}\label{PFlemPathW}
The function $g(\theta-y,y)$ is strictly decreasing for $y>0$.\\
For $y\to\infty$ it holds $\partial_y g(\theta-y,y)\sim -\ln(y)$.\\
For $y\searrow 0$ we have $g(\theta-y,y)=g(\theta,0)-\tfrac23 c_\theta^3 y^3+\Or(y^4)$.
\end{lemma}
\begin{proof}
Using (\ref{PFeqg1}) and (\ref{PFeqg2}) we have
\begin{equation}\label{PFeq49}
\begin{aligned}
&\frac{\partial g(\theta-y,y)}{\partial y} = g_2(\theta-y,y)-g_1(\theta-y,y)\\
&=-\sum_{n=0}^\infty\left(\frac{1}{\theta+n}-\frac{n+\theta-2y}{(n+\theta-y)^2+y^2}\right)=-\sum_{n=0}^\infty\frac{2 y^2}{(\theta+n)((\theta+n-y)^2+y^2)}
\end{aligned}
\end{equation}
which is $0$ for $y=0$ and strictly negative for $y>0$. The asymptotics for large $y$ are obtained by writing (\ref{PFeq49}) as
\begin{equation*}
\frac{\I}2\Psi(-y+\theta+\I y)-\frac{\I}2\Psi(-y+\theta-\I y)-\frac12\Psi(-y+\theta-\I y)-\frac12\Psi(-y+\theta+\I y)+\Psi(\theta)
\end{equation*}
and using the large-$z$ expansion
\begin{equation}\label{PFeq51}
\Psi(z)=\ln(z)-\frac{1}{2z}+\Or(z^{-2}).
\end{equation}
Taylor expansion gives the small $y$ estimate.
\end{proof}

\begin{lemma}\label{PFlemPathZ}
For any $x\geq \theta$, the function $g(x,y)$ is strictly increasing for $y>0$.\\
For $y\to\infty$ it holds $\partial_y g(x,y)\sim \kappa_{\theta} y$.
\end{lemma}
\begin{proof}
From (\ref{PFeqg2}) we have
\begin{equation*}
\frac{\partial g(x,y)}{\partial y} =\sum_{n=0}^\infty\left(\frac{y}{(\theta+n)^2}-\frac{y}{(x+n)^2+y^2}\right),
\end{equation*}
which is $0$ for $y=0$ and for $y>0$ is strictly positive. For large $y$, the second term goes to zero, leading to the estimate.
\end{proof}

\begin{lemma}\label{PFlemPoles}
Let $y>0$ be fixed. The function
\begin{equation*}
h(y,k):=g(\theta-y,y)-g(\theta-y+k,y)
\end{equation*}
satisfies $h(y,0)=0$, $h(y,k)$ is strictly decreasing for $k\in [0,y]$.\\
For any $\delta\in (0,\theta)$, $y\geq \delta$, $h(y,k)$ is strictly decreasing in $k\in [0,y+\delta/2]$.\\
For $y\to \infty$, $\partial_k h(y,k)|_{k=0}\sim -y \kappa_\theta$.\\
For $y\leq\theta$, $\partial_k h(y,k)\leq -\frac{k y c_\theta^3}{2}$ for $k\in [0,y]$.
\end{lemma}
\begin{proof}
From (\ref{PFeqg1}) we have
\begin{equation*}
\begin{aligned}
\frac{\partial h(y,k)}{\partial k}&=-g_1(\theta-y+k,y)=-\sum_{n=0}^\infty\left(\frac{\theta+n+y-k}{(\theta+n)^2}-\frac{\theta+n-y+k}{(\theta+n-y+k)^2+y^2}\right)\\
&=-\sum_{n=0}^\infty\frac{(\theta+n)(y^2-(k-y)^2)+(y-k)^3+y^2(y-k)}{(\theta+n)^2((\theta+n-y+k)^2+y^2)}
\end{aligned}
\end{equation*}
which strictly negative for $k\in [0,y]$.

The second statement follows from
\begin{equation*}
\begin{aligned}
& (\theta+n)(y^2-(k-y)^2)+(y-k)^3+y^2(y-k) \\
&\quad\geq \theta(y^2-(k-y)^2)+(y-k)^3+y^2(y-k)\\
&\quad\quad \geq \theta (\delta^2/2-\delta^2/4)+\theta y^2/2-\delta^3/8-\delta y^2/2\geq \delta^8/8.
\end{aligned}
\end{equation*}

To get the asymptotics of the derivative for large $y$, we can rewrite
\begin{equation*}
\frac{\partial h(y,k)}{\partial k}\bigg|_{k=0}=\Psi(\theta)-\Psi'(\theta) y-\frac12 \Psi(-y+\theta-\I y)-\frac12 \Psi(-y+\theta+\I y)
\end{equation*}
and use (\ref{PFeq51}) and $\Psi'(\theta)=\kappa_\theta$.

Moreover, for $k\in [0,y]$ and $y\leq \theta$ we have the bound
\begin{equation}\label{PFeq56}
\begin{aligned}
\frac{\partial h(y,k)}{\partial k}&\leq-\sum_{n=0}^\infty\frac{(y^2-(k-y)^2)}{(\theta+n)((\theta+n-y+k)^2+y^2)}\\
&\leq-k y\sum_{n=0}^\infty\frac{1}{(\theta+n)((\theta+n-y+k)^2+y^2)}
\leq -k y\sum_{n=0}^\infty\frac{1}{2(\theta+n)^3}=-\frac{k y c_\theta^3}{2}.
\end{aligned}
\end{equation}
\end{proof}

\subsection{Proof of Theorem~\ref{PFThmF2pert}(b)}
Now we turn to the proof of the Theorem with boundary perturbations.

Note that due to the ordering of the $a_i$'s, $b_1\geq b_2>\cdots$. Call $\bar{b}=b_1$. The scaling of the $a_i$'s implies that the contours $\mathcal{C}_w$ and $\mathcal{C}_z$ can be chosen as before except for a modification in a $N^{-1/3}$-neighborhood of the critical point, since they have to pass on the right of $\theta+\bar b c_{\theta}^{-1} N^{-1/3}$ (see Figure~\ref{PFFigPathsBBPproof}).
\begin{figure}
\begin{center}
\psfrag{-thetat}[cb]{$-\theta c_\theta N^{1/3}$}
\psfrag{w}[cb]{$w$}
\psfrag{Cw}[lb]{$\mathcal{C}_w$}
\psfrag{Cz}[lb]{$\mathcal{C}_z$}
\psfrag{bbar}[cb]{$\bar b$}
\psfrag{0}[cb]{$0$}
\psfrag{O1}[cb]{$\Or(1)$}
\includegraphics[height=6cm]{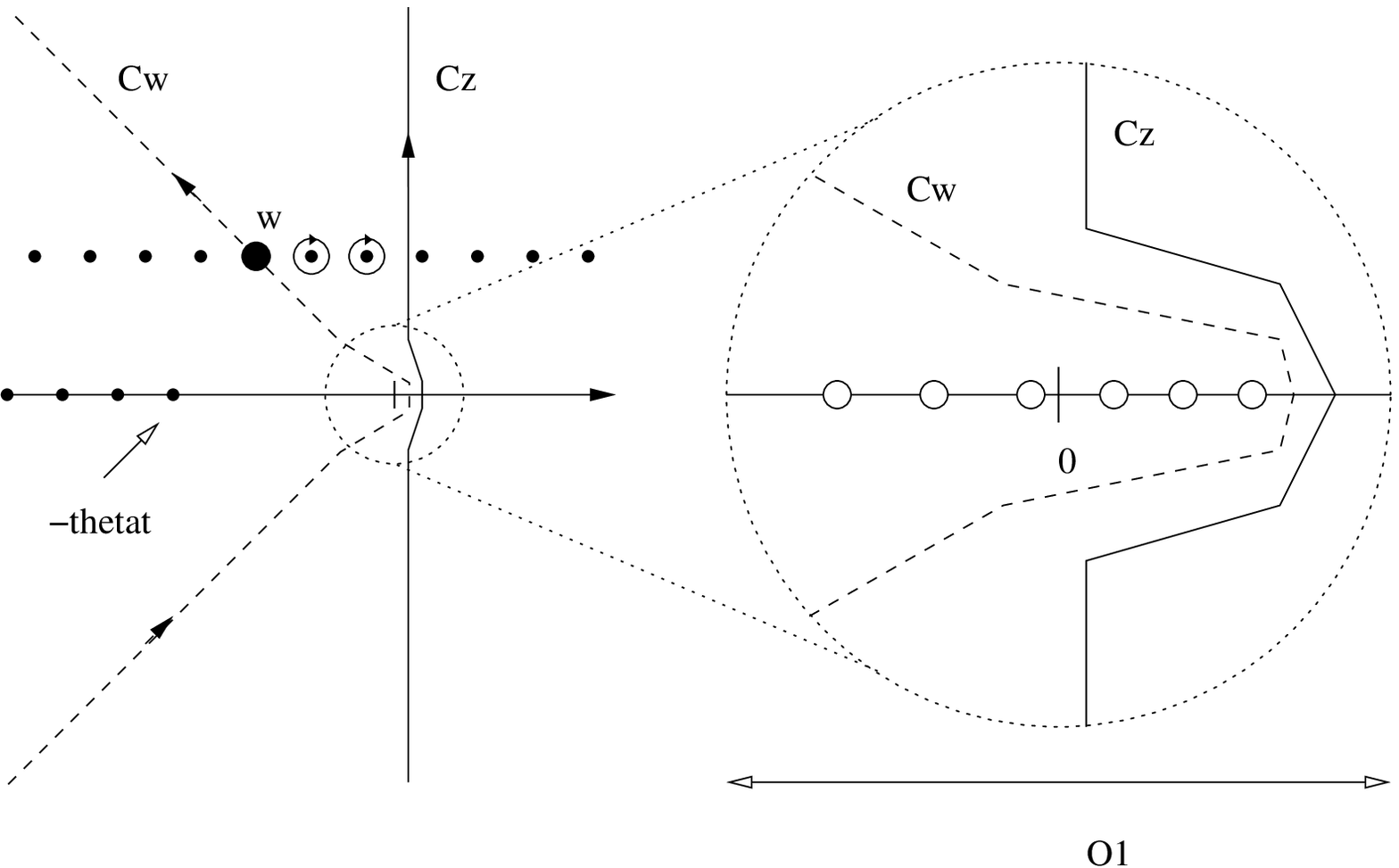}
\caption{Perturbation of the integration paths, compare with Figure~\ref{PFFigPathsTW} (right). The white dots on the right are the values of $b_1,\ldots,b_m$.}
\label{PFFigPathsBBPproof}
\end{center}
\end{figure}
Let us denote
\begin{equation}\label{PFeqPwza}
P(w,z,a):=\frac{\Gamma(\Phi(w)-a)}{\Gamma(\Phi(w))}\frac{\Gamma(\Phi(z))}{\Gamma(\Phi(z)-a)}.
\end{equation}
Then, the only difference with respect to the kernel (\ref{PFeqKernelPtF2}) is that in the $N\to\infty$ limit there might remains a factor coming from $\prod_{k=1}^m P(w,z,a_k)$.

Using $\frac{\Gamma(z+a)}{\Gamma(z+b)}\sim z^{a-b}(1+\Or(1/z))$ (see (6.1.47) of~\cite{AS84}), for any $w,z$ on $\mathcal{C}_w,\mathcal{C}_z$ we have the bound $|P(w,z,a_k)|\leq C e^{c|a_k| (|\Im(w)|+|\Im(z)|) N^{-1/3}}$ for some constants $C,c$.
The local modification of the paths has no influence on any of the bounds for large $w$ and $z$, so that the proof of pointwise convergence and of the bounds are minor modifications of Proposition~\ref{PFPropPtwiseConv} and Proposition~\ref{PFPropBound}. It remains to determine the pointwise limits of $P(w,z,a_k)$ as $N\to\infty$.\\
\emph{Case 1:} If $\limsup_{N\to\infty} (a_k(N)-\theta) N^{1/3} = -\infty$, then
\begin{equation*}
\lim_{N\to\infty}P(w,z,a_k)=1.
\end{equation*}
\emph{Case 2:} If $\limsup_{N\to\infty} (a_k(N)-\theta) N^{1/3} = b_k$, then
\begin{equation*}
\begin{aligned}
&\lim_{N\to\infty} \frac{\Gamma(\Phi(w)-\Phi(b_k))}{\Gamma(\Phi(w))}\frac{\Gamma(\Phi(z))}{\Gamma(\Phi(z)-\Phi(b_k))}\\
=& \lim_{N\to\infty} \frac{\Gamma((w-b_k) c_{\theta}^{-1} N^{-1/3})}{\Gamma(\theta+w c_{\theta}^{-1} N^{-1/3})}\frac{\Gamma(\theta+z c_{\theta}^{-1} N^{-1/3})}{\Gamma((z-b_k)c_{\theta}^{-1} N^{-1/3})}= \frac{z-b_k}{w-b_k}
\end{aligned}
\end{equation*}
because $\Gamma(z)=z^{-1}-\gamma_E+\Or(z)$ as $z\to 0$.

Therefore, one obtains
\begin{equation*}
\lim_{N\to\infty} \det(\Id+K_N)_{L^2(\mathcal{C}_w)} = \det(\Id+\widetilde K_{{\rm BBP},b})_{L^2(\mathcal{C}_w)}
\end{equation*}
where
\begin{equation}\label{eqBBPtilde}
\widetilde K_{{\rm BBP},b}(s,s')=\frac{-1}{2\pi\I} \int_{e^{-\pi \I/4}\infty}^{e^{\pi \I/4}\infty}dz \frac{e^{z^3/3-w^3/3}e^{r w-r z}}{(z-w)(z-w')}\prod_{k=1}^m \frac{z-b_k}{w-b_k}.
\end{equation}
The reformulation of Lemma~\ref{PFLemBBPreformuation} complete the proof.

\section{Details in the proof of Theorem~\ref{ThmIntDisAsy}}\label{proofCDRP}
As discussed at the beginning of Section~\ref{proofOConnellYorKPZ}, in the proof we parameterize using the position of the critical point $\theta$ instead of $\kappa$. Let us set $\T\in (0,\infty)$ and consider the scaling limit
\begin{equation*}
\theta:=\sqrt{N/\T},\quad \tau=\kappa_\theta N,\quad u=S e^{-N f_\theta}.
\end{equation*}
One has the following large-$\theta$ expansion of (\ref{PFeqKappa}) and (\ref{PFeqF}) gives
\begin{equation*}
\begin{aligned}
\kappa_\theta&=\frac{1}{\theta}+\frac{1}{2 \theta^2}+\frac{1}{6\theta^3}+\Or(\theta^{-5}),\\
f_\theta&=1-\ln(\theta)+\frac{1}{\theta}+\frac{1}{4\theta^2}+\Or(\theta^{-4}).
\end{aligned}
\end{equation*}
Thus,
\begin{equation*}
\begin{aligned}
\tau&=\kappa_\theta N=\sqrt{\T N}+\tfrac12 \T +\Or(N^{-1/2}),\\
u&=S e^{-N f_\theta}=S e^{-N-\frac12 N \ln(\T/N)+\sqrt{\T N}+\frac14 \T +\Or(N^{-1})}.
\end{aligned}
\end{equation*}
Equivalently, we can set $\tau=\sqrt{T N}$, then $\theta=\sqrt{N/T}+\frac12-\frac1{12}\sqrt{T/N}+\Or(N^{-3/2})$, so that
\begin{equation*}
\begin{aligned}
\T&=T-T^{3/2} /N^{1/2}+\tfrac{11}{12}T^2/N+\Or(N^{-3/2}),\\
u&=S e^{-N-\frac12 N\ln(T/N)-\tfrac12 \sqrt{T N}+T/4!+\Or(N^{-1})}.
\end{aligned}
\end{equation*}

As shown in Section~\ref{critpointCDRP}, what it remains is to prove Theorem~\ref{ThmCDRPdets}. We first prove the statement for the unperturbed case, and then we will show how the generalization is obtained.

\subsection{Proof of Theorem~\ref{ThmCDRPdets}(a).}
We have to determine is $\lim_{N\to\infty} \det(\Id+K_u)_{L^2(\mathcal{C}_v)}$. Consider the case of a drift vector $b=0$.
The path $\mathcal{C}_v$ is chosen as
\begin{equation*}
\begin{aligned}
\mathcal{C}_v&=\{\theta-1/4+\I r, |r|\leq r^*\}\cup\{\theta e^{\I t}, t^*\leq |t| \leq \pi/2\}\cup\{\theta-|y|+\I y,|y|\geq \theta\},\\
\end{aligned}
\end{equation*}
where $r^*=\sqrt{\theta/2-1/16}$, $t^*=\arcsin(\sqrt{1/2\theta-1/16\theta^2})$.
The path $\mathcal{C}_{\tilde z}$ is set as
\begin{equation}\label{eq5.11}
\mathcal{C}_{\tilde z}=\{\theta+p/4+\I \tilde y, \tilde y\in\R\}\cup \bigcup_{k=1}^{\ell} B_{v+k},
\end{equation}
where $B_{z}$ is a small circle around $z$ clockwise oriented and $p\in\{1,2\}$ depending on the value of $v$, see Figure~\ref{PFFigPathsKPZ}. More precisely, for given $v$, we consider the sequence of points $S=\{\Re(v)+1,\Re(v)+2,\ldots\}$ and we choose $p=p(v)$ and $\ell=\ell(v)$ as follows:\label{PFChoiceOfPathZKPZ}
\begin{itemize}
\item[(1)] If the sequence $S$ does not contain points in $[\theta,\theta+1/2]$, then let $\ell\in \N_0$ be such that $\Re(v)+\ell \in [\theta-1,\theta]$ and we set $p=1$.
\item[(2)] If the sequence $S$ contains a point in $[\theta,\theta+3/8]$, then let $\ell\in \N$ such that $\Re(v)+\ell \in [\theta,\theta+3/8]$ and set $p=2$.
\item[(3)] If the sequence $S$ contains a point in $[\theta+3/8,\theta+1/2]$, then let $\ell\in \N$ such that $\Re(v)+\ell \in [\theta-5/8,\theta-1/2]$ and set $p=1$.
\end{itemize}
With this choice, the singularity of the sine along the line $\theta+p/4+\I\R$ is not present, since the poles are at a distance at least $1/8$ from it. Also, the leading contribution of the kernel will come from situation (1) with $\ell=0$ and $p=1$.
\begin{figure}
\begin{center}
\psfrag{Phi1}[cb]{$\Phi^{-1}$}
\psfrag{0}[cb]{$0$}
\psfrag{theta}[cb]{$\theta$}
\psfrag{-thetat}[ct]{$-\frac{\theta}{\sigma}$}
\psfrag{+p4}[lt]{$\frac{p}{4\sigma}$}
\psfrag{-14}[ct]{$-\frac{1}{4\sigma}$}
\psfrag{theta+p4}[lt]{$\theta_+$}
\psfrag{theta-14}[ct]{$\theta_-$}
\psfrag{p}[lb]{$p$}
\psfrag{w}[cb]{$w$}
\psfrag{Cw}[lb]{$\mathcal{C}_w$}
\psfrag{v}[cb]{$v$}
\psfrag{Cv}[lb]{$\mathcal{C}_v$}
\psfrag{Cztilde}[lb]{$\mathcal{C}_{\tilde z}$}
\psfrag{Cz}[lb]{$\mathcal{C}_z$}
\includegraphics[height=6cm]{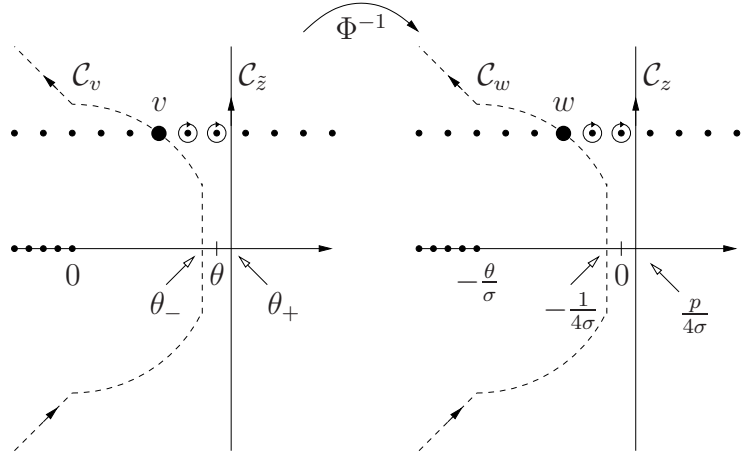}
\caption{Left: Integration paths $\mathcal{C}_v$ (dashed) and $\mathcal{C}_{\tilde z}$ (the solid line plus circles at $v+1,\ldots,v+\ell$), where $\theta_+=\theta+p/4$ and $\theta_-=\theta-1/4$, the small black dots are poles either of the sine or of the gamma function. Right:
Integration paths after the change of variables $\mathcal{C}_w$ (dashed) and $\mathcal{C}_z$ (the solid line plus circles at $w+1,\ldots,w+\ell$), with $p=p(w)\in\{1,2\}$}
\label{PFFigPathsKPZ}
\end{center}
\end{figure}
We denote
\begin{equation*}
\sigma:=(2/\T)^{1/3}
\end{equation*}
and we do the change of variable
\begin{equation*}
\{v,v',\tilde z\}=\{\Phi(w),\Phi(w'),\Phi(z)\}\quad \textrm{with}\quad\Phi(z):=\theta+z\sigma
\end{equation*}
and
\begin{equation}\label{PFeq62}
K_\theta(w,w'):=\sigma K_u(\Phi(w),\Phi(w'))=\frac{-1}{2\pi \I}\int_{\mathcal{C}_z}dz \frac{\sigma \pi S^{(z-w)\sigma}}{\sin(\pi (z-w)\sigma)}
\frac{e^{N G(\Phi(w))-N G(\Phi(z))}}{z-w'}.
\end{equation}
After this change of variable, the paths $\mathcal{C}_w=\Phi^{-1}(\mathcal{C}_v)$ and $\mathcal{C}_z=\Phi^{-1}(\mathcal{C}_{\tilde z})$ are given by
\begin{equation}\label{PFeqCwKPZ}
\begin{aligned}
\mathcal{C}_w=&\{-1/(4\sigma)+\I r/\sigma, |r|\leq r^*\}\cup\{(e^{\I t}-1)\theta/\sigma, t^*\leq |t|\leq \pi/2\}\cup\{-|y|+\I y,|y|\geq \theta/\sigma\},\\
\mathcal{C}_z=&\{p/(4\sigma)+\I y, y\in\R\}\cup \bigcup_{k=1}^{\ell} B_{w+k/\sigma},
\end{aligned}
\end{equation}
where $r^*=\sqrt{\theta/2-1/16}$, $t^*=\arcsin(\sqrt{1/2\theta-1/16\theta^2})$, and $B_{z}$ is a small circle around $z$ clockwise oriented. After this change of variable, we have
\begin{equation}\label{eq6.3}
\det(\Id+K_u)_{L^2(\mathcal{C}_v)} = \det(\Id+K_\theta)_{L^2(\mathcal{C}_w)}.
\end{equation}
Thus, we need to prove that
\begin{equation*}
\lim_{N\to\infty}\det(\Id+K_\theta)_{L^2(\mathcal{C}_w)}=\det(\Id-K_{{\rm CDRP}})_{L^2(\R_+)}
\end{equation*}
with $K_{{\rm CDRP}}$ given in Definition~\ref{2.18def}. The proof is very similar to the second part of the proof of Theorem~\ref{PFThmF2pert}(a), where this time the convergence of the kernel is in Proposition~\ref{PFPropPtConvKPZ} and the exponential bound in Proposition~\ref{PFPropBoundKPZ}. We then obtain
\begin{equation*}
\lim_{N\to\infty} \det(\Id+K_\theta)_{L^2(\mathcal{C}_w)} = \det(\Id+\widetilde K_{{\rm CDRP}})_{L^2(\mathcal{C}_w)}
\end{equation*}
with $\widetilde K_{{\rm CDRP}}$ given in (\ref{PFeqKPZtilde}). Lemma~\ref{PFLemKPZreformuation} shows that the limiting Fredholm determinant is equivalent to $\det(\Id-K_{{\rm CDRP}})_{L^2(\R_+)}$ and thus completes the proof of Theorem~\ref{ThmCDRPdets}(a).

\subsection{Pointwise convergence and bounds}
The leading contribution of the Fredholm determinant and in the kernel comes from $w,w',z$ of order $1$ away from $\theta\sim\Or(\sqrt{N})$. The scale for steep descent analysis is $N\theta$ instead of $N$ as in the case of the convergence to the GUE Tracy-Widom distribution function. So, the function whose real part has to be controlled is this time
\begin{equation*}
\widetilde G(Z):=\frac{G(\theta+\theta Z)}{\theta},
\end{equation*}
that satisfies
\begin{equation}\label{PFeqLargeThetaCircle}
\begin{aligned}
\widetilde G^{(3)}(0)&=-1+\Or(\theta^{-1}),\\
\widetilde G^{(4)}(0)&=2+\Or(\theta^{-1}),\\
\widetilde G^{(n)}(0)&=\Or(1),\quad n\geq 3,\\
G^{(n)}(\theta)&=\theta^{-n+1}\widetilde G^{(n)}(0).
\end{aligned}
\end{equation}
For asymptotic analysis we need to control the real part of $\widetilde G$, which we denote
\begin{equation*}
\widetilde g(X,Y):=\Re(\widetilde G(X+\I Y))=\frac{g(\theta+\theta X,\theta Y)}{\theta}.
\end{equation*}
In Lemmas~\ref{PFlemLargeThetaW},~\ref{PFlemLargeThetaZ}, and~\ref{PFlemLargeThetaPoles} we will analyze the steep descent properties for $\widetilde G$ (those are analogs of Lemmas~\ref{PFlemPathW}-\ref{PFlemPoles}), that we use to prove Proposition~\ref{PFPropPtConvKPZ} and Proposition~\ref{PFPropBoundKPZ} below.

\begin{proposition}\label{PFPropPtConvKPZ}
Uniformly for $w,w'$ in a bounded set of $\mathcal{C}_w$,
\begin{equation*}
\lim_{N\to\infty} K_\theta(w,w') = \widetilde K_{{\rm CDRP}}(w,w')
\end{equation*}
where
\begin{equation}\label{PFeqKPZtilde}
\widetilde K_{{\rm CDRP}}(w,w')= \frac{-1}{2\pi \I}\int_{\frac{1}{4\sigma}+\I \R}dz \frac{\sigma \pi S^{(z-w)\sigma}}{\sin(\pi (z-w)\sigma)} \frac{e^{z^3/3-w^3/3}}{z-w'}.
\end{equation}
\end{proposition}
\begin{proof}
First remark that the only dependence on $N$ in the kernel (\ref{PFeq62}) is in the factor
\begin{equation*}
\exp\left[N \left(G(\Phi(w))-G(\Phi(z))\right)\right]=\exp\left[N\theta \left(\widetilde G(w\sigma/\theta)-\widetilde G(z\sigma/\theta)\right)\right].
\end{equation*}
Let $w,w'$ be in a bounded set of $\mathcal{C}_w$ around the origin. For $N$ large enough and $w$ bounded in $\mathcal{C}_w$, $\Re(w\sigma+1)>1/2$ and $\Re((z-w)\sigma)\in (0,1)$ so that we have $\ell=0$ and $p=1$, i.e., in this case $\mathcal{C}_z=\{\frac{1}{4\sigma}+\I y,y\in\R\}$. We have
\begin{equation}\label{PFeq2.14}
\begin{aligned}
N G(\Phi(w))=N G(\theta+w\sigma)&=N G(\theta)+\frac{N}{6} G^{(3)}(\theta)\sigma^3 w^3+ \Or(N w^4/\theta^3)\\
&=N G(\theta)-\frac{N w^3 \sigma^3}{6\theta^2}+\Or(N w^4/\theta^3,N w^3/\theta^3)\\
&=N G(\theta)-\frac{w^3}{3}+\Or(w^4/\theta)
\end{aligned}
\end{equation}
where the $\theta$-dependence in the error term follows from $G^{(4)}(\theta)=\Or(\theta^{-3})$ and then we used the expansion (\ref{PFeqLargeThetaCircle}) for $G^{(3)}(\theta)$.

We divide the integral over $z$ into two parts: (a) $|\Im(z)|>\theta^{1/3}$  and \mbox{(b) $|\Im(z)|\leq \theta^{1/3}$}.

\smallskip
\emph{(a) Contribution of the integration over $|\Im(z)|>\theta^{1/3}$.} For $w,w'$ on $\mathcal{C}_w$ of order $1$ and $z\in \mathcal{C}_z$, $|z-w'|\geq \Or(1)$, $|\sin(\pi (z-w)\sigma)^{-1}|=\Or(1)$. So,
\begin{equation*}
|K_\theta(w,w')|\leq \Or(\theta)\int_{\theta^{-2/3}}^\infty dY \exp\left[N\theta \left(\widetilde g(0,0)- \widetilde g((4\sigma \theta)^{-1},Y)\right)\right].
\end{equation*}
From Lemma~\ref{PFlemLargeThetaZ} we have that $-\widetilde g((4\sigma \theta)^{-1},Y)$ is strictly decreasing with derivative going to $-\infty$ as $Y$ goes to infinity. Then, the integral over $Y$ is bounded and of leading order
\begin{equation}\label{PFeq74}
\exp\left[N\theta \left(\widetilde g(0,0)- \widetilde g((4\sigma \theta)^{-1},\theta^{-2/3})\right)\right].
\end{equation}
The estimates for small $Y$  of Lemma~\ref{PFlemLargeThetaZ} with $X=(4\sigma \theta)^{-1}$ and $Y=\theta^{-2/3}$ lead then to
\begin{equation*}
(\ref{PFeq74})\leq \exp\left[N\theta \left(\widetilde g(0,0)- \widetilde g((4\sigma \theta)^{-1},0)-\frac1{12\theta^{8/3}}+\Or(\theta^{-11/3})\right)\right]=\Or(1) \exp\left(-\frac{\T^{5/6} N^{1/6}}{12}\right)
\end{equation*}
where we also used the fact that $\widetilde g((4\sigma \theta)^{-1},0)=\widetilde g(0,0)+\Or(\theta^{-3})$. Thus, the contribution in the kernel from the integration over $|\Im(z)|\geq \theta^{1/3}$ is of order $e^{-c N^{1/6}}$ for some positive constant $c>0$.

\smallskip
\emph{(b) Contribution of the integration over $|\Im(z)|\leq \theta^{1/3}$.}
We need to estimate
\begin{equation}\label{PFeq2.20}
\frac{-1}{2\pi \I}\int_{\frac{1}{4\sigma}+\I y,|y|\leq\theta^{1/3}}dz \frac{\sigma \pi S^{(z-w)\sigma}}{\sin(\pi (z-w)\sigma)} \frac{e^{N G(\Phi(w))-N G(\Phi(z))}}{z-w'}.
\end{equation}

Unlike the scaling where we have proven the convergence to the GUE Tracy-Widom distribution, in this case the sine function survives in the limiting expression and we do not have to employ the quartic term in the estimates (since it was used only to control the error term of the sine).

First we verify that the convergence is controlled by the third order term. For this purpose, we set $z=\I y +1/(4\sigma)$. Then, using (\ref{PFeqLargeThetaCircle}) we obtain (as in (\ref{PFeq2.14}))
\begin{equation*}
-N G(\theta+\sigma z) =-N G(\theta)+\frac{z^3}{3}+\Or(z^4/\theta).
\end{equation*}
The real part of the cubic term is given by
\begin{equation}\label{PFeq69b}
\Re\left(\frac{z^3}{3}\right)=-\frac{y^2}{4\sigma^2}+\frac{1}{192\sigma^3}.
\end{equation}
In our situation we have $|y|\leq \theta^{1/3}$, therefore $-\frac{y^2}{4\sigma^2}$ dominates $\Or(z^4/\theta)$ for large $\theta$ (since $y^2=\Or(\theta^{2/3})$).

We have
\begin{equation*}
(\ref{PFeq2.20})=\frac{-1}{2\pi \I}\int_{\frac{1}{4\sigma}+\I y,|y|\leq\theta^{1/3}}dz \frac{\sigma \pi S^{(z-w)\sigma}}{\sin(\pi (z-w)\sigma)} \frac{e^{z^3/3-w^3/3+\Or(w^4/\theta;z^4/\theta)}}{z-w'}.
\end{equation*}
We divide the integration in (b.1) $\theta^{1/6}\leq |y|\leq \theta^{1/3}$ and (b.2) $|y|\leq \theta^{1/6}$. Since the quadratic term in $y$ from (\ref{PFeq69b}) dominates the others, the contribution of (b.1) is only of order
$\Or(e^{-c_1 \theta^{1/3}})=\Or(e^{-c_2 N^{1/6}})$ for some constants $c_1,c_2>0$. The contribution from (b.2) is given by
\begin{equation}\label{PFeq2.24}
\frac{-1}{2\pi \I}\int_{\frac{1}{4\sigma}+\I y,|y|\leq\theta^{1/6}}dz \frac{\sigma \pi S^{(z-w)\sigma}}{\sin(\pi (z-w)\sigma)} \frac{e^{z^3/3-w^3/3+\Or(w^4/\theta;z^4/\theta)}}{z-w'}.
\end{equation}
For $|y|\leq \theta^{1/6}$, $\Or(z^4/\theta)=\Or(\theta^{-1/3})$. Using $|e^x-1|\leq |x| e^{|x|}$ for $x=\Or(z^4/\theta)$ and then for $x=\Or(w^4/\theta)$ we can delete the error term by making an error of order $\Or(\theta^{-1/3})=\Or(N^{-1/6})$. Thus,
\begin{equation*}
(\ref{PFeq2.24})=\Or(N^{-1/6})+\frac{-1}{2\pi \I}\int_{\frac{1}{4\sigma}+\I y,|y|\leq\theta^{1/6}}dz \frac{\sigma \pi S^{(z-w)\sigma}}{\sin(\pi (z-w)\sigma)} \frac{e^{z^3/3-w^3/3}}{z-w'}.
\end{equation*}
Finally, extending the last integral to $\frac{1}{4\sigma}+\I\R$ we make an error of order $\Or(e^{-c_3 \theta^{1/3}})$ for some constant $c_3>0$.

Putting all the above estimates together we obtain that, for $w,w'\in \mathcal{C}_w$ in a bounded set around $0$,
\begin{equation*}
K_\theta(w,w')=\Or(N^{-1/6})+\frac{-1}{2\pi \I}\int_{\frac{1}{4\sigma}+\I \R}dz \frac{\sigma \pi S^{(z-w)\sigma}}{\sin(\pi (z-w)\sigma)} \frac{e^{z^3/3-w^3/3}}{z-w'}.
\end{equation*}
\end{proof}

\begin{proposition}\label{PFPropBoundKPZ}
For any $w,w'$ in $\mathcal{C}_w$, uniformly for all $N$ large enough,
\begin{equation*}
|K_\theta(w,w')|\leq C e^{-|\Im(w)|}
\end{equation*}
for some constant $C$.
\end{proposition}
\begin{proof}
First recall the expression of the kernel,
\begin{equation*}
\begin{aligned}
K_\theta(w,w')&=\frac{-1}{2\pi \I}\int_{\mathcal{C}_z}dz \frac{\sigma \pi S^{(z-w)\sigma}}{\sin(\pi (z-w)\sigma)}
\frac{e^{N G(\Phi(w))-N G(\Phi(z))}}{z-w'}\\
&=S^{-w\sigma}e^{N G(\Phi(w))-N G(\theta)} \frac{-1}{2\pi \I}\int_{\mathcal{C}_z}dz \frac{\sigma \pi S^{z\sigma}}{\sin(\pi (z-w)\sigma)} \frac{e^{N G(\theta)-N G(\Phi(z))}}{z-w'}.
\end{aligned}
\end{equation*}
As in the proof of Proposition~\ref{PFPropBound}, the dependence on $w'$ is marginal because (a) we can choose the integration variable $z$ such that $|z-w'|\geq 1/(4\sigma)$ and (b) we will get the bound through evaluating the absolute value of the integrand of (\ref{PFeq62}).

\medskip
\emph{Case 1: $w\in \{-1/(4\sigma)+\I y, |y|\leq r^*/\sigma\}$ with $r^*=\sqrt{\theta/2-1/16}$.} In this case, the integration path for $z$ is $1/(4\sigma)+\I\R$ and no extra contributions from poles of the sine are present. The factor $1/\sin(\pi (z-w)\sigma)$ is uniformly bounded from above. Doing the change of variable $z=\frac{1}{4\sigma}+\I \frac{Y\, \theta}{\sigma}$ we get
\begin{equation}\label{PFeq2.30}
|K_\theta(w,w')|\leq \Or(1) e^{N \Re(G(\theta+w\sigma))-N G(\theta)}
\int_{\R}dY e^{N\theta \left(\widetilde g(0,0)-\widetilde g(\tilde \e,Y)\right)} \theta
\end{equation}
with $\tilde \e=1/(4\theta)$. The estimates as in the proof of Proposition~\ref{PFPropPtConvKPZ} on the integral over $Y$ yield
\begin{equation}\label{PFeq2.31}
(\ref{PFeq2.30})\leq \Or(1)\times  e^{N \Re(G(\theta+w\sigma))-N G(\theta)} =\Or(1)\times  e^{N\theta \Re(\widetilde G(w\sigma/\theta))-N\theta\widetilde G(0)}.
\end{equation}
Since $|w\sigma/\theta|\leq \Or(\theta^{-1/2})$ is small, we can use Taylor expansion and with (\ref{PFeqLargeThetaCircle}) we obtain
\begin{equation*}
N\theta\widetilde G(w\sigma/\theta))-N\theta\widetilde G(0) = -\frac13 w^3(1+\Or(\theta^{-1}))+\frac{\sigma}{6\theta}w^4(1+\Or(\theta^{-1})),
\end{equation*}
substituting $w=-1/(4\sigma)+\I y$ and taking the real part we get
\begin{equation*}
N\theta \Re(\widetilde G(w\sigma/\theta))-N\theta\widetilde G(0) =  -\frac{1}{4\sigma} y^2+\frac{\sigma}{6\theta}y^4+\Or(1)+\Or(y^3/\theta,y^4/\theta^2).
\end{equation*}
Now, for $|y|\leq \sqrt{\theta/2}/\sigma$, $\frac{\sigma}{6\theta}y^4\leq \frac{1}{12\sigma} y^2$ and the quadratic term dominates $\Or(y^3/\theta,y^4/\theta^2)$ for large $\theta$. Therefore, for all $\theta$ large enough, we have
\begin{equation*}
N\theta \Re(\widetilde G(w\sigma/\theta))-N\theta\widetilde G(0) \leq  -\frac{1}{8\sigma} y^2+\Or(1).
\end{equation*}
Consequently,
\begin{equation*}
|K_\theta(w,w')|\leq \Or(1) e^{-\frac{1}{8\sigma} |\Im(w)|^2}\leq C e^{-|\Im(w)|}
\end{equation*}
for some finite constant $C$.

\medskip
\emph{Case 2: $w\in\{(e^{\I t}-1)\theta/\sigma, t^*\leq |t|\leq \pi/2\}\cup\{-|y|+\I y,|y|\geq \theta/\sigma\}$}.
We divide the estimation of the bound by dividing into the contributions from (a) integration over $\frac{p}{4\sigma}+\I\R$ with $p\in\{1,2\}$ depending on $w$ (see the definitions after (\ref{eq5.11})) and (b) integration over the circles $B_{w+k/\sigma}$, $k=1,\ldots,\ell$.

\medskip
\emph{Case 2(a).} First notice that the estimate (\ref{PFeq2.30}) of \emph{Case 1} still holds with the minor difference that $\tilde\e=p/(4\theta)$ where $p\in\{1,2\}$ depending on the value of $w$. Then, also (\ref{PFeq2.31}) still holds, so that we need only to estimate $N\theta \Re(\widetilde G(w\sigma/\theta))-N\theta\widetilde G(0)$.

For $w\in\{(e^{\I t}-1)\theta/\sigma, t^*\leq |t|\leq \pi/2\}$, in Lemma~\ref{PFlemPathWCircle} we show that $\widetilde g(\cos(t)-1,\sin(t))-\widetilde g(0,0)\leq -\sin(t)^4/16$. Replacing $\Im(w)=\sin(t) \theta/\sigma$ and using $|\Im(w)|\geq \sqrt{\theta/2-1/16}$ we obtain
\begin{equation*}
N\theta \Re(\widetilde G(w\sigma/\theta))-N\theta\widetilde G(0)\leq -c_1 |\Im(w)|^4/\theta\leq -c_2 |\Im(w)| \sqrt{\theta} \leq -|\Im(w)|
\end{equation*}
for all $\theta$ large enough, where $c_1,c_2$ are some (explicit) constants. This is the desired bound.

For $w\in \{-|y|+\I y,|y|\geq \theta/\sigma\}$, from Lemma~\ref{PFlemLargeThetaW} it follows that there exists a constant $c_3>0$ such that $\partial_Y\widetilde g(-Y,Y)\leq -c_3$ and from Lemma~\ref{PFlemPathWCircle} we know that $\widetilde g(-1,1)-\widetilde g(0,0)\leq -1/16$. Thus, for $c_4=\min\{\sigma/16,c_3\}$ it holds
$\widetilde g(-1,1)-\widetilde g(0,0)\leq -c_4 Y$ for all $|Y|\geq 1/\sigma$. This means that
\begin{equation*}
N\theta \Re(\widetilde G(w\sigma/\theta))-N\theta\widetilde G(0)\leq -c_4 N \theta |\Im(w)|/\theta\leq -|\Im(w)|
\end{equation*}
for $N$ large enough, giving us the needed bound.

\medskip
\emph{Case 2(b).} It remains to check that the extra contributions of the poles of the sine also tend to zero exponentially in $|\Im(w)|$. The contribution of the integration over $B_{w+k/\sigma}$ is (up to a $\pm$ sign depending on $k$) given by
\begin{equation*}
\frac{S^k e^{N G(\Phi(w))-N G(\Phi(w+k/\sigma))}}{w+k/\sigma-w'}.
\end{equation*}
Let us set $\widetilde h(Y,k):=\widetilde g(-Y,Y)-\widetilde g(-Y+k,Y)$. From Lemma~\ref{PFlemLargeThetaPoles} it follows that the largest contribution comes from the integration over $B_{w+1/\sigma}$. We have at most $\Or(|\Im(w)|)$ poles and also \mbox{$|w+k/\sigma-w'|\geq \Or(1/\theta)$} (the worst case is at the junction between the arc of circle and the straight lines). Thus, the contribution of all the poles is bounded by
\begin{equation}\label{PFeq2.40}
\Or(\theta |\Im(w)|) S^{|\Im(w)|} e^{N G(\Phi(w))-N G(\Phi(w+1/\sigma))} =\Or(\theta |\Im(w)|) S^{|\Im(w)|} e^{N\theta \widetilde h(Y,1)},
\end{equation}
where $Y=|\Im(w)|\sigma/\theta$. We consider separately the cases $1/\sqrt{2\theta}\leq Y\leq 1$ and $Y>1$:
\begin{itemize}
  \item[(1)] For $1/\sqrt{2\theta}\leq Y\leq 1$, the bound on $\partial_k \widetilde h(Y,k)$ leads to $\widetilde h(Y,1)\leq -Y/8$.
  \item[(2)] For $Y>1$ we know that $\widetilde h(Y,1)<0$ and $\partial_k \widetilde h(Y,k)|_{k=0}\simeq -Y$ as $Y\to\infty$. By the continuity of $\widetilde h(Y,1)$ in $Y$, there exists a constant $c_5>0$ such that $\widetilde h(Y,1)\leq -c_5 Y$ for all $Y\geq 1$.
\end{itemize}
Therefore, with $c_6=\min\{c_5,1/8\}$ and inserting $Y=|\Im(w)|\sigma/\theta$ we have
\begin{equation*}
(\ref{PFeq2.40})\leq \Or(\theta |\Im(w)|) S^{|\Im(w)|} e^{-N c_5\sigma |\Im(w)|} \leq \Or(1) e^{-|\Im(w)|}
\end{equation*}
for $N$ large enough. We have shown that also the contributions of the poles have the desired bound.
\end{proof}

\begin{lemma}\label{PFlemPathWCircle}
The function $\widetilde g(\cos(t)-1,\sin(t))$ is zero at $t=0$ and strictly decreasing for $t\in (0,\pi/2]$.\\
For $t\in [0,\pi/2]$ and $\theta$ large enough, $\partial_t \widetilde g(\cos(t)-1,\sin(t)) \leq -\sin(t)(1-\cos(t))/2$ so that
\begin{equation*}
\widetilde g(\cos(t)-1,\sin(t))- \widetilde g(0,0)\leq - \sin(t)^4/16.
\end{equation*}
\end{lemma}
\begin{proof}
We have $\widetilde g(\cos(t)-1,\sin(t))=\theta^{-1}g(\theta\cos(t),\theta\sin(t))$, thus
\begin{equation}\label{PFeq2.34}
\begin{aligned}
\frac{\partial \widetilde g(\cos(t)-1,\sin(t))}{\partial t}
&=\cos(t) g_2(\theta\cos(t),\theta\sin(t))-\sin(t) g_1(\theta\cos(t),\theta\sin(t))\\
&=-\sum_{n=0}^\infty \frac{2\theta^2\sin(t)(1-\cos(t))(2n\cos(t)+\theta)}{(2n\theta\cos(t)+n^2+\theta^2)(n+\theta)^2}
\end{aligned}
\end{equation}
is strictly negative for all $t\in (0,\pi/2]$, which shows the first result. We can further bound
\begin{equation*}
\begin{aligned}
(\ref{PFeq2.34})&\leq -\sin(t)(1-\cos(t)) \sum_{n=0}^\infty \frac{2\theta^3}{(2n\theta\cos(t)+n^2+\theta^2)(n+\theta)^2}
\\
&\leq -\sin(t)(1-\cos(t)) \sum_{n=0}^\infty \frac{2\theta^3}{(n+\theta)^4}=-\sin(t)(1-\cos(t))(\tfrac23+\Or(\theta^{-1})).
\end{aligned}
\end{equation*}
Thus, for large enough $\theta$, the derivative is bounded by $-\sin(t)(1-\cos(t))/2$, $t\in [0,\pi/2]$. Integrating over $[0,t]$ gives
\begin{equation*}
\widetilde g(\cos(t)-1,\sin(t))- \widetilde g(0,0)\leq -(1-\cos(t))^2/4\leq -\sin(t)^4/16
\end{equation*}
for $t\in [0,\pi/2]$.
\end{proof}

\begin{lemma}\label{PFlemLargeThetaZ}
For any $X\geq 0$, the function $\widetilde g(X,Y)$ is strictly increasing for $Y>0$, with
$\partial_Y \widetilde g(X,Y)\geq \partial_Y \widetilde g(0,Y)$.\\
For $Y\searrow 0$, $\partial_Y \widetilde g(0,Y)= Y^3/3+\Or(Y^3/\theta;Y^5)$, so that
\begin{equation*}
\widetilde g(X,Y)\geq \widetilde g(X,0)+Y^4/12+\Or(Y^4/\theta;Y^6).
\end{equation*}
For $Y\to\infty$ it holds $\partial_Y \widetilde g(X,Y)\sim Y$.
\end{lemma}
\begin{proof}
We have
\begin{equation*}
\frac{\partial \widetilde g(X,Y)}{\partial Y} =g_2(\theta+\theta X,\theta Y)=\sum_{n=0}^\infty\left(\frac{\theta Y}{(\theta+n)^2}-\frac{\theta Y}{(\theta+\theta X+n)^2+\theta^2 Y^2}\right),
\end{equation*}
which is $0$ for $Y=0$ and for $Y>0$ is strictly positive. The inequality
\begin{equation*}
\partial_Y \widetilde g(X,Y)\geq \partial_Y \widetilde g(0,Y)=\sum_{n=0}^\infty
\frac{\theta^3 Y^3}{(\theta+n)^2((\theta+n)^2+\theta^2 Y^2)},
\end{equation*}
whose expansion for small $Y$ and large $\theta$ is given by $Y^3/3+\Or(Y^3/\theta;Y^5)$.

For large $Y$, the second term becomes irrelevant with respect to the first, so that \mbox{$\partial_Y \widetilde g(X,Y)\sim \theta \kappa_\theta Y= Y(1+\Or(\theta^{-1}))$}.
\end{proof}

\begin{lemma}\label{PFlemLargeThetaW}
The function $\widetilde g(-Y,Y)$ is strictly decreasing for $Y>0$.\\
For $Y\to\infty$ it holds $\partial_Y \widetilde g(-Y,Y)\sim -\ln(Y)$.\\
For $Y\searrow 0$ we have $\widetilde g(-Y,Y)=\widetilde g(0,0)-\tfrac13 Y^3+\Or(Y^3/\theta;Y^4)$.
\end{lemma}
\begin{proof}
As in the proof of Lemma~\ref{PFlemPathW}, we use (\ref{PFeqg1}) and (\ref{PFeqg2}) to obtain
\begin{equation*}
\begin{aligned}
\frac{\partial \widetilde g(-Y,Y)}{\partial Y} &= g_2(\theta-\theta Y,\theta Y)-g_1(\theta-\theta Y,\theta Y)
=-\sum_{n=0}^\infty\frac{2 Y^2\theta^2}{(\theta+n)((\theta+n-\theta Y)^2+\theta^2 Y^2)}.
\end{aligned}
\end{equation*}
which is $0$ for $Y=0$ and strictly negative for $Y>0$.

We can rewrite the sum with the variable $\eta =n/\theta$. Then, for large (but still fixed) $\theta$ the sum over $\eta\in \{0,1,2,\ldots\}/\theta$ is very close to the integral over $\eta\in [0,\infty)$. From this one deduces that for large $Y$, $\partial_Y \widetilde g(-Y,Y)\sim -\ln(Y)$. The asymptotics for $Y\searrow 0$ is obtained by writing the Taylor series of $\widetilde G(Z)$ around $Z=1$ and taking the real part of it.
\end{proof}

\begin{lemma}\label{PFlemLargeThetaPoles}
Let $Y>0$ be fixed. The function
\begin{equation*}
\widetilde h(Y,k):=\widetilde g(-Y,Y)-\widetilde g(-Y+k,Y)
\end{equation*}
satisfies $\widetilde h(Y,0)=0$, $\widetilde h(Y,k)$ is strictly decreasing for $k\in [0,Y]$.\\
For any $\delta\in (0,1)$, $Y\geq \delta$, $\widetilde h(Y,k)$ is strictly decreasing in $k\in [0,Y+\delta/2]$.\\
For $Y\to \infty$, $\partial_k \widetilde h(Y,k)|_{k=0}\sim -Y$.\\
For $Y\leq 1$, $\partial_k \widetilde h(Y,k)\leq - kY/4$ for $k\in [0,Y]$.
\end{lemma}
\begin{proof}
The first statements follows directly from Lemma~\ref{PFlemPoles}. The asymptotics for large $Y$ can be obtained by approximating the sums in $\partial_k \widetilde h$ by integrals. The bound for $Y\leq 1$ follows from (\ref{PFeq56}) with $k\to \theta k$, $y\to \theta Y$, and $\sum_{n=0}^\infty (\theta+n)^{-3}\geq 1/(2\theta^2).$
\end{proof}

\subsection{Proof of Theorem~\ref{ThmCDRPdets}(b).}
Now we consider the perturbed case, where
\begin{equation*}
a_k:=\theta+b_k,\quad k=1,\ldots,m.
\end{equation*}
Then, the change of variable as in (\ref{PFeq62}) leads to the kernel
\begin{equation*}
K_\theta(w,w'):=\sigma K_u(\Phi(w),\Phi(w'))
=\frac{-1}{2\pi \I}\int_{\mathcal{C}_z}dz \frac{\sigma \pi S^{(z-w)\sigma}}{\sin(\pi (z-w)\sigma)}
\frac{e^{N G(\Phi(w))-N G(\Phi(z))}}{z-w'} \prod_{k=1}^m P(w,z,b_k)
\end{equation*}
where the perturbation term is
\begin{equation*}
P(w,z,b_k)=\frac{\Gamma(\sigma w-b_k) \Gamma(\Phi(z)) \theta^{\sigma w}}{\Gamma(\sigma z-b_k)\Gamma(\Phi(w))\theta^{\sigma z}}.
\end{equation*}
The difference from Theorem~\ref{ThmCDRPdets}(a) is that now (as it was the case for Theorem~\ref{PFThmF2pert}(b)), the paths $\mathcal{C}_z$ and $\mathcal{C}_w$ have to be locally modified around the critical point, $\theta$, so that they remains on the right of all the $b_1/\sigma,\ldots,b_m/\sigma$, see Figure~\ref{PFFigPathsKPZproof} for an illustration.
\begin{figure}
\begin{center}
\psfrag{-thetat}[cb]{$-\theta c_\theta N^{1/3}$}
\psfrag{w}[cb]{$w$}
\psfrag{Cw}[lb]{$\mathcal{C}_w$}
\psfrag{Cz}[lb]{$\mathcal{C}_z$}
\psfrag{bbar}[cb]{$\bar b$}
\psfrag{O1}[cb]{$\Or(1)$}
\psfrag{0}[cb]{$0$}
\psfrag{-thetat}[ct]{$-\frac{\theta}{\sigma}$}
\psfrag{+p4}[lt]{$\frac{p}{4\sigma}$}
\psfrag{-14}[ct]{$\frac{1}{4\sigma}$}
\includegraphics[height=6cm]{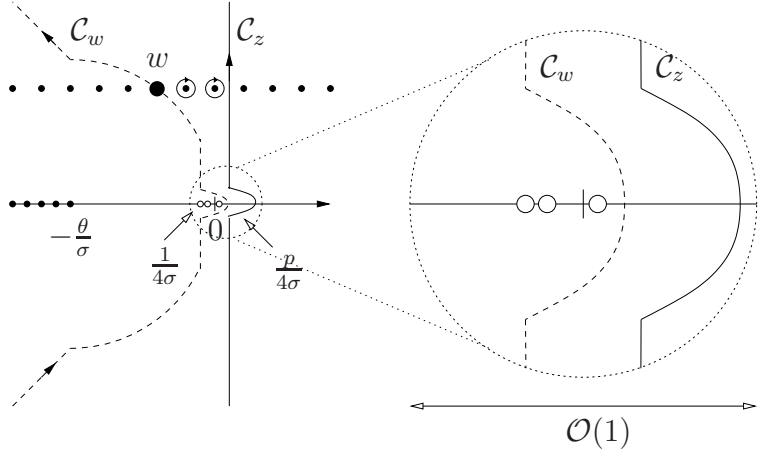}
\caption{Perturbation of the integration paths, compare with Figure~\ref{PFFigPathsKPZ} (right). The white dots on the right are the values of $b_1/\sigma,\ldots,b_m/\sigma$.}
\label{PFFigPathsKPZproof}
\end{center}
\end{figure}

We just have to show that
\begin{equation*}
\lim_{N\to\infty}\det(\Id+K_\theta)_{L^2(\mathcal{C}_w)}=\det(\Id-K_{{\rm CDRP},b})_{L^2(\R_+)}
\end{equation*}
with $K_{{\rm CDRP},b}$ as in (\ref{eq2.18}).

The proof is a minor modification of the one of Theorem~\ref{PFThmF2pert}(b). The local modification of the paths have no influence on the bounds for large $z$ and/or for large $w$. This is because $N G(\theta+b_k)-N G(\theta) = \Or(1)$ and the path for $z$ is the same away from a distance $\Or(1)$ from the origin.
What remains to be clarified is the limit kernel. We can choose the path $\mathcal{C}_w$ to be as before with a small perturbation (e.g.\ a circle) around $0$ so that it passes on the right of the all the $b_k$'s. Then, we modify the path $\mathcal{C}_z$ in the same way, i.e., by doing the same small perturbation but shifted to the right by $1/2\sigma$ (see Figure~\ref{PFFigPathsKPZproof} too). This ensures that we do not get extra poles from the sine. Finally, for the pointwise convergence of the kernel the new term remaining is
\begin{equation*}
\lim_{N\to\infty}P(w,z,b_k)=\frac{\Gamma(\sigma w-b_k)}{\Gamma(\sigma z-b_k)}.
\end{equation*}
Indeed, since $\Phi(z)=\theta+\Or(1)$ and $\Phi(w)=\theta+\Or(1)$, then $\lim_{N\to\infty}\frac{\Gamma(\Phi(z)) \theta^{\sigma w}}{\Gamma(\Phi(w))\theta^{\sigma z}}=1$.

Finally one reformulate the Fredholm determinant into one on $L^2(\R_+)$ in the same way as the unperturbed case of Lemma~\ref{PFLemKPZreformuation}.
The only small difference is that we the first step requires $\Re(z-w')>0$, which holds only for $b_k<\frac14$, for all $k$. Under this condition the rewriting holds. By looking at the final expressions one verifies that both sides are analytic in the parameters $b_1,\ldots,b_m$. Thus we have equality by analytic continuation.  This ends the proof of Theorem~\ref{ThmCDRPdets}(b).

\section{Details in the proof of Theorem~\ref{NeilPolymerFredDetThm}}\label{MacSec}

\subsection{Proof of Proposition~\ref{abgCor}}\label{abgCorproof}

This closely follows the proof of~\cite{BC11} Proposition~3.2.8 and Corollary 3.2.10. However, in that case the contour playing the role of $\CwPre{\tilde \alpha,\varphi}$ is bounded whereas it is unbounded presently. As such, some additional estimates must be made, so we include the entire proof here.

First observe that we may combine the $q$-moments $\mu_{k}=\langle q^{k\lambda_N}\rangle_{\MM_{t=0}(\tilde a_1,\dots,\tilde a_N;\rho)}$ (see definitions in Section~\ref{emergfreddet}) into a generating function
\begin{equation*}
G_{q}(\zeta)=\sum_{k\geq 0} \frac{(\zeta/(1-q))^k}{k_q!} \left\langle q^{k\lambda_N}\right\rangle_{\MM_{t=0}(\tilde a_1,\dots,\tilde a_N;\rho)}
\end{equation*}
where $k_q!= (q;q)_n/(1-q)^n$ and $(a;q)_k=(1-a)\cdots (1-aq^{k-1})$ (when $k=\infty$ the product is infinite, though convergent since $|q|<1$).
The convergence of the series defining $G_{q}(\zeta)$ follows from the fact that $q^{k\lambda_N}\leq 1$ and
\begin{equation*}
\frac{(\zeta/(1-q))^k}{k_q!}  = \frac{\zeta^k}{(1-q)\cdots (1-q^k)},
\end{equation*}
which shows geometric decay for large enough $k$. This justifies writing
\begin{equation*}
G_{q}(\zeta)=\left\langle \sum_{k\geq 0} \frac{(\zeta/(1-q))^k}{k_q!} q^{k\lambda_N}\right\rangle_{\MM_{t=0}(\tilde a_1,\dots,\tilde a_N;\rho)}  = \left\langle \frac{1}{\left(\zeta q^{\lambda_N};q\right)_{\infty}}\right\rangle_{\MM_{t=0}(\tilde a_1,\dots,\tilde a_N;\rho)}
\end{equation*}
where the second equality follows from the $q$-Binomial theorem~\cite{AAR04}.

It now suffices to show that $G_{q}(\zeta) =\det(\Id+K)$ as in the statement of the proposition. From now on, all contour integrals are along $\CwPre{\tilde \alpha,\varphi}$. Observe that we can rewrite the summation in the definition of $\mu_k$ so that
\begin{equation*}
\mu_k \frac{\zeta^k}{k_q!}= \sum_{L\geq 0}\sum_{\substack{m_1,m_2,\ldots\\ \sum m_i =L \\ \sum i m_i = k}} \frac{1}{(m_1+m_2+\cdots)!}\cdot \frac{(m_1+m_2+\cdots)!}{m_1! m_2! \cdots}\int\cdots \int I_L(\lambda;w;\zeta) \prod_{j=1}^{L} \frac{dw_j}{2\pi \I},
\end{equation*}
where $w=(w_1,\ldots, w_L)$ , $\lambda=(\lambda_1,\ldots, \lambda_L)$ and is specified by $\lambda=1^{m_1}2^{m_2}\cdots$, and where the integrand is
\begin{equation*}
I_L(\lambda;w;\zeta) = \det\left[\frac{1}{w_i q^{\lambda_i}-w_j}\right]_{i,j=1}^{L} \prod_{j=1}^{L}(1-q)^{\lambda_j}\zeta^{\lambda_j} f(w_j)f(qw_j)\cdots f(q^{\lambda_j-1}w_j).
\end{equation*}

The term $\tfrac{(m_1+m_2+\cdots)!}{m_1! m_2! \cdots}$ is a multinomial coefficient and can be removed by replacing the inner summation by
\begin{equation*}
\sum_{n_1,\ldots,n_L\in \mathcal{L}_{k,m_1,m_2,\ldots}} \int\cdots \int I_L(n;w;\zeta)\frac{dw_j}{2\pi \I},
\end{equation*}
with $n=(n_1,\ldots,n_L)$ and where
\begin{equation*}
\mathcal{L}_{k,m_1,m_2,\ldots} = \{n_1,\ldots,n_L\geq 1: \sum_i n_i = k \textrm{ and for each } j\geq 1, m_j \textrm{ of the } n_i \textrm{ equal } j\}.
\end{equation*}
This gives
\begin{equation*}
\mu_k \frac{\zeta^k}{k_q!} = \sum_{L\geq 0}\frac{1}{L!} \sum_{\substack{n_1,\ldots,n_L\geq 1\\ \sum n_i=k}} \int\cdots \int I_L(n;w;\zeta)\frac{dw_j}{2\pi \I}.
\end{equation*}

Now we may sum over $k$ which removes the requirement that $\sum_i n_i = k$. This yields that the left-hand side of equation (\ref{deteqn}) can be expressed as
\begin{equation}\label{fredexpabove}
\sum_{L\geq 0} \frac{1}{L!} \sum_{n_1,\ldots,n_L\geq 1} \int \cdots \int \det\left[\frac{1}{q^{n_i}w_i-w_j}\right]_{i,j=1}^{L} \prod_{j=1}^{L} (1-q)^{n_j}\zeta^{n_j} f(w_j)f(qw_j)\cdots f(q^{n_j-1}w_j) \frac{dw_j}{2\pi \I } .
\end{equation}
This is the definition of the Fredholm determinant expansion $\det(\Id+K)$, as desired. As these were purely formal manipulations, at this point to complete the proof we must justify the rearrangements in the above argument. In order to do this, we will show that the double summation of (\ref{fredexpabove}) is absolutely convergent. This is the point at which the unboundedness of the $\CwPre{\tilde \alpha,\varphi}$ contour introduces a slight divergence from the analogous proof of~\cite{BC11} Proposition~3.2.8 where the contour was bounded and of finite length.

Basically, the absolute convergence follows from the exponential decay of the function $f$ as the real part of $w$ increases to positive infinity, combined with Hadamard's inequality. Let us bound the absolute value of the integrand in (\ref{fredexpabove}). Note that by assumption $q^{n_i}w_i/w_j-1$ is bounded from 0 uniformly as $w_i$, $w_j$, and $n_i$ vary. Thus, it follows that for some finite constant $B_1$,
\begin{equation*}
\left| \det\left[\frac{1}{q^{n_i}w_i-w_j}\right]_{i,j=1}^{L}\right| \leq B_1^{L} L^{L/2}.
\end{equation*}
Since the function $f(w)$ is bounded as $w$ varies and has exponential decay with respect to the real part of $w$, we can replace
\begin{equation*}
\left|f(w_j)f(qw_j)\cdots f(q^{n_j-1}w_j)\right| \leq (B_2)^{n_j}e^{-c\Re(w_j)}
\end{equation*}
for constants $c>0$ and $B_2<\infty$. Thus we find that
\begin{equation}\label{eq7.12}
|(\ref{fredexpabove})|\leq \sum_{L\geq 0} \frac{1}{L!} B_1^L L^{L/2} \left(\sum_{n\geq 1} (B_2 (1-q) \zeta)^n \int \frac{|dw|}{2\pi} e^{-c \Re(w)}\right)^L.
\end{equation}
Since $w$ is being integrated along $\CwPre{\tilde \alpha,\varphi}$, the integral over $w$ is bounded by some constant $B_3<\infty$. Finally, for $|\zeta|$ small enough the geometric series converges and it is bounded by a constant $B_4$. Therefore
\begin{equation*}
(\ref{eq7.12})\leq \sum_{L\geq 0} \frac{(B_1 B_3 B_4)^{L} L^{L/2}}{L!}<\infty.
\end{equation*}
Thus we have shown that the double summation in (\ref{fredexpabove}) is absolutely convergent, completing the proof of Proposition~\ref{abgCor}.

\subsection{Proof of Theorem~\ref{PlancherelfredThm}}\label{PlancherelfredThmproof}
This theorem and its proof are adapted from~\cite{BC11} Theorem~3.2.11. However, in that theorem, the $w$-contour $\CwPre{\tilde \alpha,\varphi}$, was of finite length and the $s$-contour $\CsPre{w}$ was just a vertical line. The need for slightly more involved contours comes from the unboundedness of the $w$-contour and the necessity that $\tilde K_{\zeta}(w,w')$ goes to zero sufficiently fast as $|w|$ grows along the $w$-contour.

The starting point for this proof is Proposition~\ref{abgCor}. There are, however, two issues we must deal with. First, the operator in the proposition acts on a different $L^2$ space; second, the equality is only proved for $|\zeta|<C^{-1}$ for some constant $C>1$. We split the proof into three steps. \emph{Step 1:} We present a general lemma which provides an integral representation for an infinite sum. \emph{Step 2:} Assuming $\zeta\in \{\zeta:|\zeta|<C^{-1}, \zeta\notin\Rplus\}$ we derive equation (\ref{thmlaplaceeqn}). \emph{Step 3:} A direct inspection of the left-hand side of that equation shows that for all $\zeta\neq q^{-M}$ for $M\geq 0$ the expression is well-defined and analytic. The right-hand side expression can be analytically extended to all $\zeta\notin \Rplus$ and thus by uniqueness of the analytic continuation, we have a valid formula on all of $\C\setminus\Rplus$.

\subsubsection*{Step 1:}
The purpose of the next lemma is to change that $L^2$ space we are considering and to replace the summation in Proposition~\ref{abgCor} by a contour integral.
\begin{lemma}\label{gammasumlemma}
For all functions $g$ which satisfy the conditions below, we have the identity that for $\zeta\in \{\zeta:|\zeta|<1, \zeta\notin\Rplus\}$
\begin{equation*}
\sum_{n=1}^{\infty} g(q^n) (\zeta)^n  = \frac{1}{2\pi \I} \int_{C_{1,2,\ldots}} \Gamma(-s)\Gamma(1+s)(-\zeta)^s g(q^s) ds,
\end{equation*}
where the infinite contour $C_{1,2,\ldots}$ is a negatively oriented contour which encloses $1,2,\ldots$ and no singularities of $g(q^s)$, and $z^s$ is defined with respect to a branch cut along $z\in \R_-$. For the above equality to be valid the left-hand-side must converge, and the right-hand-side integral must be able to be approximated by integrals over a sequence of contours $C_{k}$ which enclose the singularities at $1,2,\ldots, k$ and which partly coincide with $C_{1,2,\ldots}$ in such a way that the integral along the symmetric difference of the contours $C_{1,2,\ldots}$ and $C_{k}$ goes to zero as $k$ goes to infinity.
\end{lemma}
\begin{proof}
The identity follows from $\Res{s=k}\Gamma(-s)\Gamma(1+s) = (-1)^{k+1}$.
\end{proof}

\subsubsection*{Step 2:} For this step let us assume that $\zeta\in \{\zeta:|\zeta|<C^{-1}, \zeta\notin\Rplus\}$. We may rewrite equation (\ref{abgKernel}) as
\begin{equation*}
K(n_1,w_1;n_2,w_2) = \zeta^{n_1} g_{w_1,w_2}(q^{n_1})
\end{equation*}
where $g$ is given in equation (\ref{gwwprimeeqn}).

Writing out the $M^{th}$ term in the Fredholm expansion we have
\begin{equation*}
\frac{1}{M!} \sum_{\sigma\in S_M} \sign(\sigma)\prod_{j=1}^{M} \int_{\CwPre{\tilde \alpha,\varphi}} \frac{dw_j}{2\pi \I} \sum_{n_j=1}^{\infty} \zeta^{n_j} g_{w_j,w_{\sigma(j)}}(q^{n_j}).
\end{equation*}

In order to apply Lemma~\ref{gammasumlemma} we need to define the sequence of contours $C_{k}$ (in fact we need only specify the contours for $k$ large). Let $C_{k}$ be composed of the union of two parts -- the first part is the portion of the contour $\CsPre{w}$ which lies within the ball of radius $k+1/2$ centered at the origin; the second part is the arc of the boundary of that ball which causes the union to be a closed contour which encloses $\{1,2,\ldots, k\}$ and no other integers. The contours $C_{k}$ are oriented positively and illustrated in the left-hand-side of Figure~\ref{SymDiffContours}. The infinite contour $C_{1,2,\ldots}$ is chosen to be $\CsPre{w}$ oriented as in the statement of the theorem (decreasing imaginary part). By the definition of the contours $\CwPre{\tilde \alpha,\varphi}$ and $\CsPre{w}$ we are assured that the contours $C_k$ do not contain any poles beyond those of the Gamma function $\Gamma(-s)$. This is due to the fact that the contours have been chosen such that as $s$ varies, $q^sw$ stays entirely to the left of $\CwPre{\tilde \alpha,\varphi}$ and hence does not touch $w'$.

\begin{figure}
\begin{center}
\psfrag{0}[cb]{$0$}
\psfrag{k1}[rt]{$k$}
\psfrag{k2}[lt]{$k+1$}
\psfrag{Cs}[lb]{$C_{k}$}
\psfrag{Carc}[lb]{$C^{arc}_k$}
\psfrag{Cseg}[lb]{$C^{seg}_k$}
\psfrag{R}[cb]{$R$}
\psfrag{2d}[lb]{$2d$}
\psfrag{12}[cb]{$\frac12$}
\includegraphics[height=5cm]{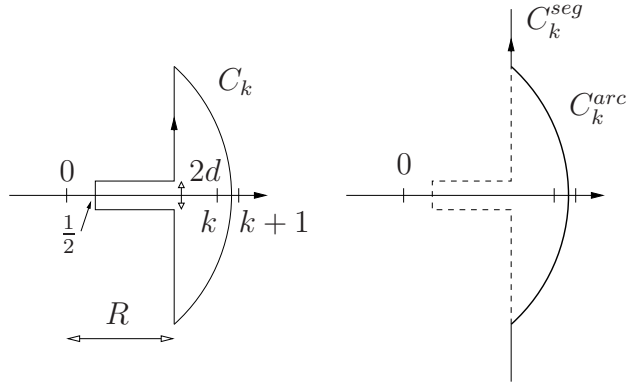}
\caption{Left: The contour $C_{k}$ composed of the union of two parts -- the first part is the portion of the contour $\CsPre{w}$ which lies within the ball of radius $k+1/2$ centered at the origin; the second part is the arc of that ball which causes the union to be a closed contour which encloses $\{1,2,\ldots, k\}$ and no other integers. Right: The symmetric difference between $C_k$ and $\CsPre{w}$ is given by two parts: a semi-circle arc which we call $C^{arc}_k$ and a portion of $R+\I \R$ with magnitude exceeding $k+1/2$ which we call $C^{seg}_k$.}\label{SymDiffContours}
\end{center}
\end{figure}

In order to apply the above lemma we must also estimate the integral along the symmetric difference. Identify the part of the symmetric difference given by the circular arc as $C^{arc}_k$ and the part given by the portion of $R+\I \R$ with magnitude exceeding $k+1/2$ as $C^{seg}_k$ (see the right-hand-side of Figure~\ref{SymDiffContours}). First observe that for $w_1,w_2$ fixed, $g_{w_1,w_2}(q^s)$ stays uniformly bounded as $s$ varies along these contours. Consider next $(-\zeta)^s$. If $-\zeta = r e^{i\sigma}$ for $\sigma\in (-\pi,\pi)$ and $r>0$ we have $(-\zeta)^s =r^s e^{\I s\sigma}$. Writing $s=x+\I y$ we have $|(-\zeta)^s| = r^{x}e^{-y\sigma}$. Note that our assumption on $\zeta$ corresponds to $r<1$ and $\sigma\in(-\pi,\pi)$. Concerning the product of Gamma functions, recall Euler's Gamma reflection formula
\begin{equation*}
\Gamma(-s)\Gamma(1+s) = \frac{\pi}{\sin (-\pi s)}.
\end{equation*}
One readily confirms that for all $s$: $\dist(s,\Z)>c$ for some $c>0$ fixed,
\begin{equation*}
\left| \frac{\pi}{\sin (-\pi s)} \right| \leq \frac{c'}{e^{\pi \Im(s)}}
\end{equation*}
for a fixed constant $c'>0$ which depends on $c$. Therefore, along the $C^{seg}_k$ contour where $s=R+\I y$,
\begin{equation*}
|(-\zeta)^s\Gamma(-s)\Gamma(1+s)|\sim r^R e^{-y\sigma}e^{-\pi|y|},
\end{equation*}
and since $\sigma\in(-\pi,\pi)$ is fixed, this product decays exponentially in $|y|$ and the integral goes to zero as $k$ goes to infinity. Along the $C^{arc}_k$ contour, the product of Gamma functions still behaves like $c'e^{-\pi |y|}$ for some fixed $c'>0$. Thus along this contour (again using the notation $s=x+\I y$)
\begin{equation*}
|(-\zeta)^s\Gamma(-s)\Gamma(1+s)| \sim e^{-y\sigma}r^x e^{-\pi|y|}.
\end{equation*}
Since $r<1$ and $-(\pi+\sigma)<0$ these terms behave like $e^{-c''(x+|y|)}$ ($c''>0$ fixed) along the circular arc. Clearly, as $k$ goes to infinity, the integrand decays exponentially in $k$ (versus the linear growth of the length of the contour) and the conditions of the lemma are met.

Applying Lemma~\ref{gammasumlemma} we find that the $M^{th}$ term in the Fredholm expansion can be written as
\begin{equation*}
\frac{1}{M!} \sum_{\sigma\in S_M} \sign(\sigma)\prod_{j=1}^{M} \int_{\CwPre{\tilde \alpha,\varphi}} \frac{dw_j}{2\pi \I}\int_{\CsPre{w_j}}\frac{ds_j}{2\pi \I} \,\Gamma(-s)\Gamma(1+s)(-\zeta)^{s} g_{w_j,w_{\sigma(j)}}(q^{s}).
\end{equation*}
Therefore the determinant can be written as $\det(\Id+\tilde K_{\zeta})$ as desired.

\subsubsection*{Step 3:} At this point we now make critical use of the choice for the contour $\CwPre{\tilde \alpha,\varphi}$ on which $w$ varies, since proving analyticity in $\zeta$ of the Fredholm determinant requires the decay properties of the kernel with respect to $w$ varying along $\CwPre{\tilde \alpha,\varphi}$.

In order to analytically extend our formula we must prove two facts. First, that the left-hand side of equation (\ref{thmlaplaceeqn}) is analytic for all $\zeta\notin\Rplus$; and second, that the right-hand side determinant is defined (i.e., its expansion is convergent) and analytic for all $\zeta\notin\Rplus$.

Expand the left-hand side of equation (\ref{thmlaplaceeqn}) as
\begin{equation*}
\sum_{n=0}^{\infty} \frac{ \PP(\lambda_N = n) } {(1-\zeta q^n)(1-\zeta q^{n+1})\cdots},
\end{equation*}
where $\PP=\PP_{\MM_{t=0}(\tilde a_1,\ldots,\tilde a_N;\rho)}$.

Observe that for any $\zeta\notin \{q^{-M}\}_{M=0,1,\ldots}$, within a neighborhood of $\zeta$ the infinite products are uniformly convergent and bounded away from zero. As a result the series is uniformly convergent in a neighborhood of any such $\zeta$ which implies that its limit is analytic, as desired.

Turning to the Fredholm determinant, we must show that the series denoted by $\det(\Id+\tilde K_{\zeta})$ is an analytic function of $\zeta$ away from $\Rplus$. For this we will appeal to the fact that limits of uniformly absolutely convergent series of analytic functions are themselves analytic. Recall that
\begin{equation*}
\det(\Id+\tilde K_{\zeta}) = 1 + \sum_{n=1}^{\infty} \frac{1}{n!} \int_{\CwPre{\tilde \alpha,\varphi}}\frac{dw_1}{2\pi \I} \cdots  \int_{\CwPre{\tilde \alpha,\varphi}} \frac{dw_n}{2\pi \I} \det(\tilde K_{\zeta}(w_i,w_j))_{i,j=1}^n.
\end{equation*}
It is clear from the definition of $\tilde K_{\zeta}$ that $\det(\tilde K_{\zeta}(w_i,w_j))_{i,j=1}^{n}$ is analytic in $\zeta$ away from $\Rplus$. Thus any partial sum of the above series is analytic in the same domain. What remains is to show that the series is uniformly absolutely convergent on any fixed neighborhood of $\zeta$ not including $\Rplus$.  Towards this end consider the $n^{th}$ term in the Fredholm expansion:
\begin{equation}\begin{aligned}\label{fredexpbound}
F_n(\zeta) &=& \frac{1}{n!} \int_{\CwPre{\tilde \alpha,\varphi}}\frac{dw_1}{2\pi \I}\cdots  \int_{\CwPre{\tilde \alpha,\varphi}}\frac{dw_n}{2\pi \I} \int_{\CsPre{w_1}} \frac{ds_1}{2\pi \I} \cdots \int_{\CsPre{w_n}} \frac{ds_n}{2\pi \I} \det\left(\frac{1}{q^{s_i}w_i -w_j}\right)_{i,j=1}^{n}\\
&&\times \prod_{j=1}^{n}  \left(\Gamma(-s_j)\Gamma(1+s_j) (-\zeta)^{s_j} \exp\big(\gamma w_j(q^{s_j}-1)\big)\prod_{m=1}^{N}\frac{(q^{s_j}w_j/\tilde a_m;q)_{\infty}}{(w_j/\tilde a_m;q)_{\infty}} \right).
\end{aligned}\end{equation}
We wish to bound the absolute value of this. We may pull the absolute values inside the integration. Now observe the following bounds which hold uniformly over all $w_j\in \CwPre{\tilde \alpha,\varphi}$ and then all $s_j\in \CsPre{w_j}$. For the first bound note that for all $z:|z|\leq 1$, there exists a constant $c_q$ such that $|(z;q)_{\infty}|<c_q$. From that it follows that for any $|z|>1$, $|(z;q)_{\infty}|\leq c_q|(z;q)_{k}|$ where $k$ is such that $|zq^k|\leq 1$. This $k$ is approximately $-\ln(|z|)/\ln(q)$ and hence bounded by $k\leq c_q' \ln|z|$ for some other constant $c_q'$. Finally $|(z;q)_{k}|\leq c_q'' |z|^k \leq c_q'' |z|^{c_q'\ln|z|}$. From this and the fact that $|q^{s_j}|<1$ (recall that $\Re(s_j)>0$ along  $\CsPre{w_j}$), it follows that for $|w_j|\geq 1$ along $\CwPre{\tilde \alpha,\varphi}$ we can bound
\begin{equation}\label{pochbound}
\left| \prod_{m=1}^{N} \frac{(q^{s_j}w_j/\tilde a_m;q)_{\infty}}{(w_j/\tilde a_m;q)_{\infty}}\right| \leq c_1 |w_{j}|^{N c_q'\ln(|w_j|/\tilde a)}
\end{equation}
for some constant $c_1$ and $\tilde a=\min_i\{\tilde a_i\}$. For $|w_j|\leq 1$ along $\CwPre{\tilde \alpha,\varphi}$, we can bound the above left-hand side by a constant and since $|w_j|$ is bounded from below along $\CwPre{\tilde \alpha,\varphi}$, it follows that the above bound (\ref{pochbound}) holds for all $w\in \CwPre{\tilde \alpha,\varphi}$, possibly with a modified value of $c_1$.

By Hadamard's inequality and the conditions we have imposed on $\CsPre{w_j}$ we get the crude bound
\begin{equation}\label{detbound}
\left| \det\left(\frac{1}{q^{s_i}w_i -w_j}\right)_{i,j=1}^{n} \right| \leq c_2^n n^{n/2}.
\end{equation}
for some fixed constant $c_2>0$.

Finally note that by the conditions we imposed in choosing the contours, for $w_j$ on $\CwPre{\tilde \alpha,\varphi}$ and $s_j$ on $\CsPre{w_j}$, we have $\Re\left(w_j(q^s-1)\right)\leq -c_{\varphi} |w_j|$ where $c_{\varphi}>0$ is a constant depending on $\varphi\in (0,\pi/4)$. From this it follows that
\begin{equation}\label{expbound}
\left|\exp\big(\gamma w_j(q^{s_j}-1)\big)\right| \leq \exp\big(-\gamma c_{\varphi}|w_j|\big).
\end{equation}

Taking the absolute value of (\ref{fredexpbound}) and bringing the absolute value all the way inside the integrals, we find that after plugging in the results of (\ref{pochbound}), (\ref{detbound}) and (\ref{expbound})\footnote{For a complex contour $C$ and a function $f:C\to \R$ we write $\int_C |dz|f(z)$ for the integral of $f$ along $C$ with respect to the arc length $|dz|$.}
\begin{equation}\label{almostthere}
|F_n(\zeta)| \leq  \frac{(c_1 c_2)^n n^{n/2}}{n!} \left(\int_{\CwPre{\tilde \alpha,\varphi}}\frac{|dw|}{2\pi}\int_{\CsPre{w}}\frac{|ds|}{2\pi} \left|\Gamma(-s)\Gamma(1+s) (-\zeta)^{s}\right| |w|^{N c_q'\ln(|w|/\tilde a)} \exp(-\gamma c_{\varphi}|w|)\right)^n.
\end{equation}

We integrate this in the $s$ variables first. For $\zeta\notin \Rplus$ we would like to bound

\begin{equation*}
 \int_{\CsPre{w}} \frac{|ds|}{2\pi} |\Gamma(-s)\Gamma(1+s) (-\zeta)^{s}|
\end{equation*}
for a neighborhood of $\zeta$ which does not touch $\Rplus$. We divide the contour of integral into two regions and bound the integral along each region: (1) The portion of the contour from $R-\I d$ to $1/2-\I d$ and then vertical to $1/2$ and its reflection through the real axis; (2) The portion of the contour which is infinite from $R-\I \infty$ to $R-\I d$ and then from $R+\I d$ to $R+\I \infty$. Recall that by Remark~\ref{contoursrem} we may assume that up to constants $d\approx |w|^{-1}$ and $R\approx \ln|w|$ for $|w|$ large enough.

Case (1): By standard bounds $|\Gamma(-s)\Gamma(1+s)|\leq 1/d \approx |w|$ (since $1/\sin(x)\approx 1/x$ near $x=0$). Calling $r$ the maximal modulus over the neighborhood of $|\zeta|$ in question, it follows that since the $s$ integral in Case (1) has length less than $2R$ (when $d<1/2$), the first part of the integral is bounded by a constant times $|w|\ln|w| r^{c_3\ln|w|}$ with a constant $c_3>0$.

Case (2): The product of Gamma functions decays exponentially in $s$ and so the integral is estimated by $r^{R}$ which, by Remark~\ref{contoursrem} is like $r^{c_3\ln|w|}$.

Summing up the above two cases we have that for $|w|$ large,
\begin{equation}\label{gammabound}
\int_{\CsPre{w}} \frac{|ds|}{2\pi} |\Gamma(-s)\Gamma(1+s) (-\zeta)^{s}| \leq c_4 r^{c_3\ln|w|} |w| \ln|w|.
\end{equation}

This estimate can be plugged in to the right-hand side of (\ref{almostthere}) to reduce the bound to just an integral in the $w_j$. This integral factors and thus we have
\begin{equation*}\begin{aligned}
|F_n(\zeta)| &\leq& \frac{(c_1 c_2)^n n^{n/2}}{n!} \left(\int_{\CwPre{\tilde \alpha,\varphi}}\frac{|dw|}{2\pi} |w|^{N c_q'\ln(|w|/\tilde a)} c_4  r^{c_3\ln|w|}|w|\ln|w|\exp(-\gamma c_{\varphi}|w|)\right)^n.
\end{aligned}\end{equation*}
The integral in $|w|$ is clearly convergent due to the exponential decay (which easily overwhelms the growth of $|w|^{Nc_q'\ln|w|}$ as well as the other terms). Thus the right-hand side above is bounded by $c_5^n n^{n/2}/n!$ for some constant $c_5$. Thus $F_n(\zeta)$ is absolutely convergent, uniformly over any fixed neighborhood of a $\zeta\notin \Rplus$. This implies that $\det(\Id+\tilde K_{\zeta})$ is analytic in $\zeta\notin\Rplus$ and hence completes the proof of \emph{Step 3} and hence the proof of the theorem.

\subsection{Proof of Proposition~\ref{postasymlemma}}\label{prop3sec}

By virtue of Lemma~\ref{exponentialdecaycutoff} it suffices to show that for some $c,C>0$,
\begin{equation}\label{kuineq1}
|K_u(v,v')|\leq Ce^{-c|v|}.
\end{equation}

Before proving this let us recall from Definition~\ref{CwPredef} the contours with which we are dealing. The variable $v$ lies on $\Cv{\alpha,\varphi}$ and hence can be written as $v=\alpha -\kappa \cos(\varphi)  \pm  \I \kappa \sin(\varphi)$, for $\kappa\in \R_+$, where the $\pm$ represents the two rays of the contour. The $s$ variables lies on $\Cs{v}$ which depends on $v$ and has two parts: The portion (which we have denoted $\Cs{v,\sqsubset}$) with real part bounded between $1/2$ and $R$ and imaginary part $\pm d$ for $d$ sufficiently small, and the vertical portion (which we have denoted $\Cs{v,\vert}$) with real part $R$. The condition on $R$ implies that $R\geq-\Re(v)+\alpha+1$ and for our purposes we will assume $R=-\Re(v)+\alpha+1$.

Let us denote by $h(s)$ the integrand in (\ref{kvvprimeC}), through which $K_u(v,v')$ is defined.
We split the proof into three steps. \emph{Step 1:} We show the integral of $h(s)$ over $s\in \Cs{v,\sqsubset}$ is bounded by an expression with exponential decay in $|v|$. \emph{Step 2:} We show the integral of $h(s)$ over $s\in \Cs{v,\vert}$ is bounded by an expression with exponential decay in $|v|$. \emph{Step 3:} We show that the integral of $h(s)$ over the entire contour $s\in \Cs{v}$ is bounded by a fixed constant, independent of $v$ or $v'$. The combination of these three steps imply the inequality (\ref{kuineq1}) and hence complete the proof.

\subsubsection*{Step 1:}
We the various terms in $h(s)$ separately and develop bounds for each. Let us write $s=x+\I y$ and note that along the contour $\Cs{v,\sqsubset}$, $y\in [-d,d]$ for $d$ small, and $x\in [1/2,R]$.

Let us start with $e^{v\tau s + \tau s^2/2}$.
The norm of the above expression is bounded by the exponential of the real part of the exponent. For $s$ along $\Cs{v,\sqsubset}$
\begin{equation*}
\Re(v s + s^2/2) = x\Re(v)+\frac{x^2}{2}-y\Im(v)-\frac{y^2}{2} .
\end{equation*}
We take $R=-\Re(v)+\alpha+1$, $d$ sufficiently small, and the bound $\Re(v)\leq \tilde c'-c' |v|$ for some constants $c',\tilde c'$ (depending on $\varphi$), to deduce
\begin{equation*}
\Re(v s + s^2/2) \leq \tilde c -c |v| x
\end{equation*}
for some constants $c,\tilde c>0$, from which
\begin{equation*}
|e^{v\tau s + \tau s^2/2}| \leq C e^{-c\tau |v|x}.
\end{equation*}
Let us now turn to the other terms in $h(s)$. We have
\begin{equation*}
|u^s|\leq e^{x\ln |u| - y\arg(u)}
\end{equation*}
and we may also bound
\begin{equation*}
\left|\frac{\Gamma(v-a_m)}{\Gamma(s+v-a_m)}\right| \leq C, \qquad \left|\frac{1}{v+s-v'}\right|\leq C,\qquad |\Gamma(-s)\Gamma(1+s)|\leq C,
\end{equation*}
for some constants $C>0$ (which may be different in each case). The first bound comes from the functional equation for the Gamma function, and the last from the fact that $s$ is bounded away from $\Z$.

Combining these together shows that for $|v|$ large, the portion of the integral of $h(s)$ for $s$ in $\Cs{v,\sqsubset}$ is bounded by (recall $s=x+\I y$)
\begin{equation*}
\int_{\Cs{v,\sqsubset}}|ds| C' e^{-c\tau |v|x + x\ln |u| -\arg(u)y}  \leq C e^{-c|v|}
\end{equation*}
for some constants $c,C>0$. Since for $|y|$ in a bounded set, everything starting from the integration path is clearly bounded, the bound holds.

\subsubsection*{Step 2:}
As above, we consider the various terms in $h(s)$ separately and develop bounds for each. Let us write $s=R+\I y$ and note that $s\in \Cs{v,\vert}$ corresponds to $y$ varying over all $|y|\geq d$.  As in \emph{Step 1}, the most important bound will be that of $e^{v\tau s + \tau s^2/2}$.

Observe that
\begin{equation*}
\Re(v s + s^2/2) = \Re(v)R - \Im(v) y + \frac{R^2}{2} -\frac{y^2}{2} = -\frac{(y+\Im(v))^2}{2} + \frac{\Im(v)^2}{2} +\frac{R^2}{2} +\Re(v)R.
\end{equation*}
Observe that because $\varphi\in (0,\pi/4)$ and $R=-\Re(v)+\alpha+1$,
\begin{equation}\label{eq6.9}
\frac{\Im(v)^2}{2} +\frac{R^2}{2} +\Re(v)R \leq \tilde c -c |v|^2
\end{equation}
for some constants $c,\tilde c>0$. Thus
\begin{equation}\label{step2bdd}
\Re(v s + s^2/2) \leq -\frac{(y+\Im(v))^2}{2} +\tilde c-c |v|^2.
\end{equation}

Let us now turn to the other terms in $h(s)$. We bound
\begin{equation*}
|u^s|\leq e^{R\ln |u| - y\arg(u)}.
\end{equation*}
By standard bounds for the large imaginary part behavior we can show
\begin{equation*}
\left|\frac{\Gamma(v-a_m)}{\Gamma(s+v-a_m)}\right| \leq C e^{\frac{\pi}{2} |y|}
\end{equation*}
for some constant $C>0$ sufficiently large. Also, $|1/(v+s-v')|\leq C$ for a fixed constant. Finally, the term
\begin{equation*}
|\Gamma(-s)\Gamma(1+s)|\leq Ce^{-\pi |y|},
\end{equation*}
for some constant $C>0$.

Combining these together shows that the integral of $h(s)$ over $s$ in $\Cs{v,\vert}$ is bounded by a constant time
\begin{equation}\label{eqnest}
\int_{\R} \exp\left(-\tau \frac{(y+\Im(v))^2}{2} -\tau c |v|^2 + R\ln|u| -y\arg(u) -\pi  |y| + N\frac{\pi}{2}|y|\right) dy.
\end{equation}

We can factor out the terms above which do not depend on $y$, giving
\begin{equation*}
\exp\left(-\tau c |v|^2 + R\ln|u| \right) \int_{\R} \exp\left(-\tau \frac{(y+\Im(v))^2}{2} -y \arg(u) + N\frac{\pi}{2}|y|\right) dy.
\end{equation*}
Notice that the prefactors on $y$ and $|y|$ in the integrand's exponential are fixed constants. We can therefore use the following bound that for $a$ fixed and $b\in \R$, there exists a constant $C$ such that
\begin{equation*}
\int_{\R} e^{-\beta(y+b)^2 + a|y|}dy \leq C e^{|ab|},\quad \beta>0.
\end{equation*}
Using this we find that we can upper-bound (\ref{eqnest}) by
\begin{equation}\label{6.12}
\exp\left(-\tau c |v|^2 + R\ln|u|+c'|v|\right).
\end{equation}
For $|v|$ large enough the Gaussian decay in the above bound dominates, and hence integral of $h(s)$ over $s$ in $\Cs{v,\vert}$ is bounded by
\begin{equation*}
C e^{-c|v|}
\end{equation*}
for some constants $c,C>0$.

\subsubsection*{Step 3:}
Since $v'$ only comes in to the term $1/(v+s-v')$ in the integrand, it is clear that the above arguments imply that the integral of $h(s)$ over the entire contour $s\in \Cs{v}$ is bounded by a fixed constant, independent of $v$ or $v'$. This completes the third step and hence completes the proof of Proposition~\ref{postasymlemma}.

\subsection{Proof of Proposition~\ref{finiteSprop}}\label{prop1sec}
Fix $\eta,r>0$. We are presently considering the Fredholm determinant of the kernels $K_u^{\e}$ and $K_u$ restricted to the fixed finite contour $\CvRLeq{\alpha,\varphi,<r}$. By Lemma~\ref{uniformptconvergence} we need only show convergence as $\e\to 0$ of the kernel $K_u^{\e}(v,v')$ to $K_u(v,v')$, uniformly in $v,v'\in \CvRLeq{\alpha,\varphi,<r}$. This is achieved via the following lemma.
\begin{lemma}
For all $\eta'>0$ there exists $\e_0>0$ such that for all $\e<\e_0$ and for all \mbox{$v,v'\in \CvRLeq{\alpha,\varphi,<r}$},
\begin{equation}\label{etapeqn}
\left|K_u^{\e}(v,v')-K_u(v,v')\right|\leq \eta'.
\end{equation}
\end{lemma}
\begin{proof}

The kernels $K_u^{\e}$ and $K_u$ are both defined via integrals over $s$. The contour on which $s$ is integrated can be fixed for ($\e<\e_0$) to equal $\Cs{v}$, which is the $s$ contour used to define $K_u$. The fact that the $s$ contours are the same for $K_u^{\e}$ and $K_u$ is convenient. The proof of this lemma will follow from three claims. The first deals with the uniformity of convergence of the integrand defining $K_u^{\e}$ to the integrand defining $K_u$ for $s$ restricted to any fixed compact set.

Before stating this claim, let us define some notation.

\begin{definition}\label{CsMdef}
Let $\Cs{v,>M}= \{s\in \Cs{v}: |s|\geq M\}$ be the portion of $\Cs{v}$ of magnitude greater than $M$ and similarly let  $\Cs{v,<M}= \{s\in \Cs{v}: |s|<M\}$. Let us assume $M$ is large enough so that $\Cs{v,>M}$ is the union of two vertical rays with fixed real part $R=-\Re(v)+\alpha+1$. Assuming this, we will write $s=R+\I y$. Then for $y_M=(M^2- (-\Re(v)+\alpha+1)^2)^{1/2}$, the contour $\Cs{v,>M}=\{R+\I y: |y|\geq y_M\}$.
\end{definition}

\begin{claim}\label{etappclaim}
For all $\eta''>0$ and $M>0$ there exists $\e_0>0$ such that for all $\e<\e_0$, for all $v,v'\in \CvRLeq{\alpha,\varphi,<r}$, and for all $s\in \Cs{v,<M}$,
\begin{equation}\label{uniflemmaeqn}
\left| h^q(s) - \Gamma(-s)\Gamma(1+s)  \prod_{m=1}^{N}\frac{\Gamma(v-a_m)}{\Gamma(s+v-a_m)} \frac{ u^s e^{v\tau s+\tau s^2/2}}{v+s-v'}\right|\leq \eta'',
\end{equation}
where $h^q$ is given in (\ref{446}).
\end{claim}
\begin{proof}
This is a strengthened version of the pointwise convergence in (\ref{pwlimits1}) through (\ref{pwlimits4}). It follows from the uniform convergence of the $\Gamma_q$ function to the $\Gamma$ function on compact regions away from the poles, as well as standard Taylor series estimates. The choice of contours is such that the pole arising from $1/(v+s-v')$ is uniformly avoided in the limiting procedure as well.
\end{proof}

It remains to show that for $M$ large enough, the integrals defining $K_u^{\e}(v,v')$ and $K_u(v,v')$ restricted to $s$ in $\Cs{v,>M}$, have negligible contribution to the kernel, uniformly over $v,v'$ and $\e$. This must be done separately for each of the kernels and hence requires two claims.

\begin{claim}
For all $\eta'>0$ there exists $M_0>0$ and $\e_0>0$ such that for all $\e<\e_0$, for all $v,v'\in \CvRLeq{\alpha,\varphi,<r}$, and for all $M>M_0$,
\begin{equation*}
\left|\int_{\Cs{v,>M}} ds h^q(s)\right|\leq \eta'.
\end{equation*}
\end{claim}
\begin{proof}
We will use the notation introduced in Definition~\ref{CsMdef} and assume $M_0$ is large enough so that $\Cs{v,>M}$ is only comprised of two vertical rays.

Let us first consider the behavior of the left-hand side of (\ref{pwlimits4}). The magnitude of this term is bounded by the exponential of
\begin{equation*}
\Re(\tau \e^{-1} s + \e^{-2} \tau q^v (q^s-1)).
\end{equation*}
This equation is periodic in $y$ (recall $s=R+\I y$) with a fundamental domain $y\in[-\pi\e^{-1},\pi \e^{-1}]$. For $\e^{-1}\pi>|y|>y_0$ for some $y_0$ which can be chosen uniformly in $v$ and $\e$, the following inequality holds
\begin{equation*}
\Re(\tau \e^{-1} s + \e^{-2} \tau q^v (q^s-1)) \leq -\tau y^2/6
\end{equation*}
This can is proved by careful Taylor series estimation and the inequality that for $x\in [-\pi,\pi]$, $\cos(x)-1\leq -x^2/6$.
This provides Gaussian decay in the fundamental domain of $y$.

Turning to the ratio of $q$-Gamma functions in (\ref{pwlimits3}), observe that away from its poles, the denominator
\begin{equation}\label{festatement}
\left|\frac{1}{\Gamma_q(s+v-a_m)}\right|\leq c e^{c' f^{\e}(s)}
\end{equation}
where $c,c'$ are positive constants independent of $\e$ and $v$ (as it varies in its compact contour) and $f^{\e}(s) = {\rm dist}(\Im(s),2\pi \e^{-1} \Z)$. This establishes a periodic bound on this denominator, which grows at most exponentially in the fundamental domain. The numerator $\Gamma_q(v-a_m)$ in (\ref{pwlimits3}) is bounded uniformly by a constant. This is because the $v$ contour was chosen to avoid the poles of the Gamma function, and the convergence of the $q$-Gamma function to the Gamma function is uniform on compact sets away from those poles.

Finally, the magnitude of (\ref{pwlimits1}) corresponds to  $|u^s|$ and behaves like $e^{-R\ln{|u|} + y\arg(u)}$. Thus, we have established the following inequality which is uniform in $v,v'$ and $\e$ as $y$ varies:
\begin{equation}\label{perbdd}
\left|\left(\frac{-\zeta}{(1-q)^N}\right)^s \frac{q^v \ln q}{q^{s+v} - q^{v'}} e^{\gamma q^v(q^{s}-1)} \prod_{m=1}^{N} \frac{\Gamma_q(v-\ln_q(\tilde a_m))}{\Gamma_q(s+v-\ln_q(\tilde a_m))} \right| \leq c''\, e^{-f^{\e}(s)^2/6+c'N|f^{\e}(s)|}
\end{equation}
for some constant $c''>0$. Notice that this inequality is periodic with respect to the fundamental domain for \mbox{$y\in[-\pi\e^{-1},\pi \e^{-1}]$}.

The last term to consider is $\Gamma(-s)\Gamma(1+s)$ which is not periodic in $y$ and decays like $e^{-\pi |y|}$ for $y\in \R$. Since $\Cs{v,>M}$ is only comprised of two vertical rays we must control the integral of $h^q(s)$ for $s=R+\I y$ and $|y|>y_M$. By making sure $M$ is large enough, we can use the periodic bound (\ref{perbdd}) to show that the integral over $y_M<|y|<\e^{-1} \pi$ is less than $\eta$ (with the desired uniformity in $v,v'$ and $\e$. For the integral over $|y|>\e^{-1}\pi$, we can use the above exponential decay of $\Gamma(-s)\Gamma(1+s)$. On shifts by $2\pi \e^{-1}\Z$ of the fundamental domain, the exponential decay of $\Gamma(-s)\Gamma(1+s)$ can be compared to the boundedness of the other terms (which is certainly true considering the bounds we established above). The integral of each shift can be bounded by a term in a convergent geometric series. Taking $\e_0$ small then implies that the sum can be bounded by $\eta'$ as well. Since $\eta'$ was arbitrary the proof is complete.
\end{proof}

\begin{claim}
For all $\eta'>0$ there exists $M_0>0$ such that for all $v,v'\in \CvRLeq{\alpha,\varphi,<r}$, and for all $M>M_0$,
\begin{equation*}
\left|\int_{\Cs{v,>M}} ds \Gamma(-s)\Gamma(1+s)  \prod_{m=1}^{N}\frac{\Gamma(v-a_m)}{\Gamma(s+v-a_m)} \frac{ u^s e^{v\tau s+\tau s^2/2}}{v+s-v'}\right|\leq\eta'.
\end{equation*}
\end{claim}
\begin{proof}
The desired decay here comes easily from the behavior of $vs+s^2/2$ as $s$ varies along $\Cs{v,>M}$. As before, assume that $M_0$ is large enough so that this contour is only comprised of two vertical rays and set $s=R+ \I y$ for $y\in \R$ for $|y|>y_M$. As in the proof of Proposition~\ref{postasymlemma} given in Section~\ref{prop3sec} one shows that
\begin{equation*}
|e^{v\tau s+\tau s^2/2}|\leq C e^{-cy^2}
\end{equation*}
uniformly over $v,v'\in \CvRLeq{\alpha,\varphi,<R}$, and for all $M>M_0$. This behavior should be compared to that of the other terms: $|\Gamma(-s)\Gamma(1+s)|\approx e^{-\pi |y|}$; $|u^s|= e^{-R\ln{|u|} + y\arg(u)}$; $\left|\frac{\Gamma(v-a_m)}{\Gamma(s+v-a_m)}\right|\leq C e^{|y| \pi/2}$; and $|1/(v+s+v')|\leq C$ as well. Combining these observations we see that the integral decays in $|y|$ at worst like $C e^{-cy^2+c' |y|}$. Thus, by choosing $M$ large enough so that $y_M\gg 1$ we can be assured that the integral over $|y|>y_M$ is as small as desired, proving the above claim.
\end{proof}

Let us now combine the above three claims to finish the proof of the Proposition~\ref{finiteSprop}. Choose $\eta'=\eta/3$ and fix $M_0$ and $\e_0'$ as specified by the second and third of the above claims. Fix some $M>M_0$ and let $L$ equal the length of the finite contour $\Cs{v,<M}$. Set $\eta''=\frac{\eta'}{3L}$ and apply Claim~\ref{etappclaim}. This yields an $\e_0$ (which we can assume is less than $\e_0'$) so that (\ref{uniflemmaeqn}) holds. This implies that for $\e<\e_0$, and for all $v,v'\in \CvRLeq{\alpha,\varphi,<r}$,
\begin{equation*}
\left| \int_{\Cs{v,<M}} h^q(s) ds - \int_{\Cs{v,<M}} \Gamma(-s)\Gamma(1+s)  \prod_{m=1}^{N}\frac{\Gamma(v-a_m)}{\Gamma(s+v-a_m)} \frac{ u^s e^{v\tau s+\tau s^2/2}}{v+s-v'}ds \right|\leq \eta'/3.
\end{equation*}
From the triangle inequality and the three factors of $\eta'/3$ we arrive at the claimed result of (\ref{etapeqn}) and thus complete the proof of the lemma and hence also Proposition~\ref{finiteSprop}.
\end{proof}

\subsection{Proof of Proposition~\ref{compactifyprop}}\label{prop2sec}
The proof of this proposition is essentially a finite $\e$ (recall $q=e^{-\e}$) perturbation of the proof of Proposition~\ref{postasymlemma} given in Section~\ref{prop3sec}. The estimates presently are a little more involved since the functions involved are $q$-deformations of classic functions. However, by careful Taylor approximation with remainder estimates, all estimates can be carefully shown. By virtue of Lemma~\ref{exponentialdecaycutoff} it suffices to show that for some $c,C>0$,
\begin{equation}\label{kueineq}
|K_u^{\e}(v,v')|\leq C e^{-c|v|}.
\end{equation}

Before proving this let us recall from Definition~\ref{cpctcontdef} the contours with which we are dealing. The variable $v$ lies on $\CvEpsR{\alpha,\varphi,r}$ for $\varphi\in (0,\pi/4)$. The $s$ variables lies on $\CsPre{v}$ which depends on $v$ and has two parts: The portion (which we have denoted $\CsPre{v,\sqsubset}$) with real part bounded between $1/2$ and $R$ and imaginary part $\pm d$ for $d$ sufficiently small, and the vertical portion (which we have denoted $\CsPre{v,\sqsubset}$) with real part $R$. The condition on $R$ is that $R\geq  -\Re(v)+\alpha+1$ and for our purposes we can take that to be an equality.

Let us recall that the integrand in (\ref{445}), through which $K_u^{\e}(v,v')$ is defined, is denoted by $h^q(s)$. We split the proof into three steps. \emph{Step 1:} We show the integral of $h^q(s)$ over $s\in \CsPre{v,\sqsubset}$ is bounded for all $\e<\e_0$ by an expression with exponential decay in $|v|$. \emph{Step 2:} We show the integral of $h^q(s)$ over $s\in \CsPre{v,\vert}$ is bounded for all $\e<\e_0$ by an expression with exponential decay in $|v|$. \emph{Step 3:} We show that for all $\e<\e_0$, the integral of $h^q(s)$ over the entire contour $\CsPre{v}$ is bounded by a fixed constant, independent of $v$ or $v'$. The combination of these three steps imply the inequality (\ref{kueineq}) and hence complete the proof.

\subsubsection*{Step 1:}
We consider the various terms in $h^q(s)$ separately (in particular we consider the left-hand sides of (\ref{pwlimits1})  through (\ref{pwlimits4})) and develop bounds for each which are valid uniformly for $\e<\e_0$ and $\e<0$ small enough. Let us write $s=x+\I y$ and note that along the contour $\CsPre{v,\sqsubset}$, $y\in [-d,d]$ for $d$ small, and $x\in [1/2,R]$.

Let us start with the left-hand side of (\ref{pwlimits4}) which can be rewritten as
\begin{equation*}
\exp\left(\tau \Re(\e^{-1}s + \e^{-2} q^v(q^s-1))\right).
\end{equation*}
The norm of the above expression is bounded by the exponential of the real part of the exponent.
For $\varphi\in (0,\pi/4)$ one shows (as a perturbation of the analogous estimate in \emph{Step 1} of the Proof of Proposition~\ref{postasymlemma}) via Taylor expansion with remainder estimates that
\begin{equation*}
\tau \Re(\e^{-1}s + \e^{-2} q^v(q^s-1))\leq \tilde c- \tau c|v| x,
\end{equation*}
for some constants $c,\tilde c$.
The above bound implies
\begin{equation*}
\left|\exp\left(\tau \Re(\e^{-1}s + \e^{-2} q^v(q^s-1))\right)\right|\leq C e^{-c|v|x}.
\end{equation*}

Let us now turn to the other terms in $h^q(s)$. We bound the left-hand side of (\ref{pwlimits1}) as
\begin{equation*}
\left|e^{-\tau s \e^{-1}}\left(\frac{-\zeta}{(1-q)^N}\right)^s\right| \leq C |u^s| \leq C e^{x\ln|u| - y\arg(u)}.
\end{equation*}
We may also bound the left-hand sides of (\ref{pwlimits2}) and (\ref{pwlimits3}), as well as the remaining product of Gamma functions by constants:
\begin{equation*}
\left|\frac{\Gamma_q(v-\ln_q(\tilde a_m))}{\Gamma_q(s+v-\ln_q(\tilde a_m))} \right|\leq C,\qquad
\left|\frac{q^v \ln q}{q^{s+v} - q^{v'}}\right|\leq C, \qquad
|\Gamma(-s)\Gamma(1+s)|\leq C,
\end{equation*}
for some constants $C>0$ (which may be different in each case). The first bound comes from the functional equation for the $q$-Gamma function, and the last from the fact that $s$ is bounded away from $\Z$.

Combining these together shows that for $|v|$ large,
\begin{equation*}
\left|\int_{\CsPre{v,\sqsubset}}h^q(s) ds\right| \leq \int_{\CsPre{v,\sqsubset}} C e^{-\tau c|v| \Re(s) + x\ln|u|-y \arg(u)} |ds| \leq C' e^{-c'|v|}
\end{equation*}
for some constants $c',C'>0$, while for bounded $|v|$ the integral is just bounded as well.

\subsubsection*{Step 2:}
As above, we consider the various terms in $h^q(s)$ separately and develop bounds for each. Let us write $s=R+\I y$ and note that $s\in \CsPre{v,\vert}$ corresponds to $y$ varying over all $|y|\geq d$. Three of the terms we consider (corresponding to the left-hand sides of (\ref{pwlimits2}), (\ref{pwlimits3}) and (\ref{pwlimits4})) are periodic functions in $y$ with fundamental domain $y\in [-\pi\e^{-1},\pi\e^{-1}]$. We will first develop bounds on these three terms in this fundamental domain, and then turn to the non-periodic terms.

We start by controlling the behavior of the left-hand side of (\ref{pwlimits4}) as $y$ varies in its fundamental domain.  For each $\varphi<\pi/4$ there exists a sufficiently small (yet positive) constant $c'$ such that as $y$ varies in its fundamental domain
\begin{equation*}
\tau \Re(\e^{-1}s + \e^{-2} q^v(q^s-1)) \leq c' \tau \Re(vs+s^2/2).
\end{equation*}
On account of this, we can use the bound (\ref{step2bdd}) from the proof of Proposition~\ref{postasymlemma}. This implies that
\begin{equation*}
\tau \Re(\e^{-1}s + \e^{-2} q^v(q^s-1)) \leq c' \tau \left(-\frac{(y+\Im(v))^2}{2} -c |v|^2\right).
\end{equation*}

Let us now turn to the other $y$-periodic terms in $h^q(s)$. By bounds for the large imaginary part behavior of the $q$-Gamma function we can show

\begin{equation*}
\left|\frac{\Gamma_q(v-\ln_q(\tilde a_m))}{\Gamma_q(s+v-\ln_q(\tilde a_m))} \right| \leq C e^{c f^{\e}(s+v)}
\end{equation*}
for some constants $c,C>0$ where $f^{\e}(s) = {\rm dist}(\Im(s),2\pi \e^{-1} \Z)$. Note that as opposed to (\ref{festatement}) when $|v|$ was bounded, in the above inequality we write $f^{\e}(s+v)$ in the exponential on the right-hand side. This is because we are presently considering unbounded ranges for $v$.

Also, we can bound
\begin{equation*}
\left|\frac{q^v \ln q}{q^{s+v} - q^{v'}}\right|\leq C
\end{equation*}
for some constant $C>0$.

The parts of $h^q(s)$ which are not periodic in $y$ can easily be bounded. We bound the left-hand side of (\ref{pwlimits1}) as in \emph{Step 1} by
\begin{equation*}
\left|e^{-\tau s \e^{-1}}\left(\frac{-\zeta}{(1-q)^N}\right)^s\right| \leq C |u^s| \leq C e^{x\ln|u| - y\arg(u)}.
\end{equation*}

Finally, the term
\begin{equation*}
|\Gamma(-s)\Gamma(1+s)|\leq Ce^{-\pi |y|},
\end{equation*}
for some constant $C>0$.

We may now combine the estimates above. The idea is to first prove that the integral on the fundamental domain $y\in [-\pi\e^{-1},\pi\e^{-1}]$ is exponentially small in $|v|$. Then, by using the decay of the two non-periodic terms above, we can get a similar bound for the integral as $y$ varies over all of $\R$. For $j\in \Z$, define the $j$ shifted fundamental domain as $D_j=j\e^{-1}2\pi + [-\e^{-1}\pi,\e^{-1}\pi]$. Let
\begin{equation*}
I_j:= \int_{D_j} h^q(R+\I y) dy
\end{equation*}
and observe that combining all of the bounds developed above, we have that
\begin{equation*}
|I_j|\leq  C \int_{-\e^{-1}\pi}^{\e^{-1}\pi} F_1(y) F_2(y) dy,
\end{equation*}
where
\begin{equation*}
\begin{aligned}
F_1(y) &= \exp\left(c' \tau \left(-\frac{(y+\Im(v))^2}{2} -c |v|^2\right) +c'' f^{\e}(s+v)  +x\ln|u|\right),\\
F_2(y) &= \exp\left(- (y+j\e^{-1}2\pi)\arg(u) -\pi |y+j\e^{-1}2\pi| \right).
\end{aligned}
\end{equation*}
The term $F_1(y)$ is from the periodic bounds while $F_2(y)$ from the non-periodic terms (hence explaining the $j\e^{-1}2\pi$ shift in $y$).
By assumption on $u$, we have $-\arg(u)-\pi=\delta\leq c$ for some $\delta$. Therefore
\begin{equation*}
F_2(y) \leq C e^{-c\e^{-1} |j|}
\end{equation*}
form some constants $c,C>0$. Thus
\begin{equation*}
|I_j|\leq C e^{-c\e^{-1} |j|} \int_{-\e^{-1}\pi}^{\e^{-1}\pi} F_1(y)dy.
\end{equation*}
Just as in the end of \emph{Step 2} in the proof of Proposition~\ref{postasymlemma} we can estimate the integral
\begin{equation*}
\int_{-\e^{-1}\pi}^{\e^{-1}\pi} F_1(y)dy \leq  \hat C e^{-\hat c|v|}
\end{equation*}
for some constants $\hat C,\hat c>0$. This implies
\begin{equation*}
|I_j|\leq \hat C C e^{-c\e^{-1} |j|} e^{-\hat c|v|}.
\end{equation*}
Finally, observe that
\begin{equation*}
\left|\int_{\CsPre{v,\vert}} h^{q}(s) ds\right| \leq \sum_{j\in \Z} |I_j| \leq \hat C C e^{-\hat c|v|} \sum_{j\in \Z}e^{-c\e^{-1} |j|} \leq  C' e^{-\hat c|v|}
\end{equation*}
where $C'$ is independent of $\e$ as long as $\e<\e_0$ for some fixed $\e_0$. This is the bound desired to complete this step.

\subsubsection*{Step 3:}
Since $v'$ only comes in to the term $\frac{q^v \ln q}{q^{s+v}-q^{v'}}$ in the integrand, it is clear that the above arguments imply that the integral of $h^q(s)$ over the entire contour $s\in \CsPre{v}$ is bounded by a fixed constant, independent of $v$ or $v'$. This completes the third step and hence completes the proof of Proposition~\ref{compactifyprop}.

\section{Appendix}

\subsection{Two probability lemmas}

\begin{lemma}[Lemma~4.1.38 of~\cite{BC11}]\label{problemma1}
Consider a sequence of functions $\{f_n\}_{n\geq 1}$ mapping $\R\to [0,1]$ such that for each $n$, $f_n(x)$ is strictly decreasing in $x$ with a limit of $1$ at $x=-\infty$ and $0$ at $x=\infty$, and for each $\delta>0$, on $\R\setminus[-\delta,\delta]$, $f_n$ converges uniformly to $\mathbf{1}_{x\leq 0}$. Consider a sequence of random variables $X_n$ such that for each $r\in \R$,
\begin{equation*}
\EE[f_n(X_n-r)] \to p(r)
\end{equation*}
and assume that $p(r)$ is a continuous probability distribution function. Then $X_n$ converges weakly in distribution to a random variable $X$ which is distributed according to \mbox{$\PP(X\leq r) = p(r)$}.
\end{lemma}

\begin{lemma}[Lemma~4.1.39 of~\cite{BC11}]\label{problemma2}
Consider a sequence of functions $\{f_n\}_{n\geq 1}$ mapping \mbox{$\R\to [0,1]$} such that for each $n$, $f_n(x)$ is strictly decreasing in $x$ with a limit of $1$ at \mbox{$x=-\infty$} and $0$ at $x=\infty$, and $f_n$ converges uniformly on $\R$ to $f$. Consider a sequence of random variables $X_n$ converging weakly in distribution to $X$.
Then
\begin{equation*}
\EE[f_n(X_n)] \to \EE[f(X)].
\end{equation*}
\end{lemma}

\subsection{Some properties of Fredholm determinants}

We give some important properties for Fredholm determinants. For a more complete treatment of this theory see, for example,~\cite{Sim00}.

\begin{lemma}[Proposition~1 of~\cite{TW08b}]\label{TWprop1}
Suppose $t\to \Gamma_t$ is a deformation of closed curves and a kernel $L(\eta,\eta')$ is analytic in a neighborhood of $\Gamma_t\times \Gamma_t\subset \C^2$ for each $t$. Then the Fredholm determinant of $L$ acting on $\Gamma_t$ is independent of $t$.
\end{lemma}

\begin{lemma}\label{exponentialdecaycutoff}
Consider the Fredholm determinant $\det(\Id+K)_{L^2(\Gamma)}$ on an infinite complex contour $\Gamma$ and an integral operator $K$ on $\Gamma$. Parameterize $\Gamma$ by arc length with some fixed point corresponding to $\Gamma(0)$. Assume that $|K(v,v')|\leq C$ for some constant $C$ and for all $v,v'\in \Gamma$ and that either of the following two exponential decay conditions holds: There exists constants $c,C>0$ such that
\begin{equation*}
|K(\Gamma(s),\Gamma(s'))|\leq Ce^{-c|s|},
\end{equation*}
Then the Fredholm series defining $\det(\Id+K)_{L^2(\Gamma)}$ is well-defined. Moreover, for any $\eta>0$ there exists an $r_0>0$ such that for all $r>r_0$
\begin{equation*}
|\det(\Id+K)_{L^2(\Gamma)} - \det(\Id+K)_{L^2(\Gamma_r)}|\leq \eta
\end{equation*}
where $\Gamma_r=\{\Gamma(s):|s|\leq r\}$.
\end{lemma}
\begin{proof}
The Fredholm series expansion (\ref{eqFredholm}) is given by
\begin{equation}\label{eq7.1}
\det(\Id+K)_{L^2(\Gamma)} = \sum_{n\geq 0}\frac{1}{n!} \int_{\Gamma}ds_1\cdots \int_{\Gamma}ds_n \det(K(\Gamma(s_i),\Gamma(s_j)))_{i,j=1}^n
\end{equation}
is well-defined since by using Hadamard's bound\footnote{Hadamard's bound: the determinant of a $n\times n$ matrix with entries of absolute value not exceeding $1$ is bounded by $n^{n/2}$.} one gets that
\begin{equation}\label{eq7.2}
\left|\det(K(\Gamma(s_i),\Gamma(s_j)))_{i,j=1}^n\right|\leq n^{n/2} C^n \prod_{j=1}^n e^{-c|s_j|}
\end{equation}
which is absolutely integrable / summable. To show is $\det(\Id+K)_{L^2(\Gamma_r)}\to \det(\Id+K)_{L^2(\Gamma)}$ as $r\to\infty$. From (\ref{eq7.1}) one immediately gets that
\begin{equation*}
\det(\Id+K)_{L^2(\Gamma_r)} = \det(\Id+P_r K)_{L^2(\Gamma)}.
\end{equation*}
where $P_r$ is the projection onto $\Gamma_r$. The kernel $(P_r K)(s_i,s_j)$ converges pointwise to $K(s_i,s_j)$ and (\ref{eq7.2}) provides a in $r$ uniform, integrable / summable bound for $\det(K(\Gamma(s_i),\Gamma(s_j)))_{i,j=1}^n$. Therefore, by dominated convergence as $r\to\infty$ the two Fredholm determinant converge.
\end{proof}

\begin{lemma}\label{uniformptconvergence}
Consider a finite length complex contour $\Gamma$ and a sequence of integral operators $K^{\e}$ on $\Gamma$, as well as an addition integral operator $K$ also on $\Gamma$. Assume that for all $\eta>0$ there exists $\e_0$ such that for all $\e<\e_0$ and all $z,z'\in \Gamma$, $|K^{\e}(z,z') - K(z,z')|\leq \eta$ and that there is some constant $C$ such that $|K(z,z')|\leq C$ for all $z,z'\in \Gamma$. Then
\begin{equation*}
\lim_{\e\to 0} \det(\Id+K^{\e})_{L^2(\Gamma)}  = \det(\Id+K)_{L^2(\Gamma)}.
\end{equation*}
\end{lemma}
\begin{proof}
As in Lemma~\ref{exponentialdecaycutoff} one writes the Fredholm series. Since $\Gamma$ is finite, the Fredholm determinants $\det(\Id+K)_{L^2(\Gamma)}$ is well-defined because $|K(z,z')|\leq C$ (use Hadamard's bound). By assumption, $K^\e$ converges pointwise to $K$ and we have the uniform bound \mbox{$|K^\e(z,z')|\leq C+\eta$}. This ensures that $\det(\Id+K^{\e})_{L^2(\Gamma)}$ is well-defined and that we can take the limit inside the Fredholm series, providing our result.
\end{proof}

\subsection{Reformulation of Fredholm determinants}\label{AppFredDet}

\begin{lemma}\label{PFLemTWreformuation}
Let $\widetilde K_{\rm Ai}$ as in (\ref{PFeq27}), $\mathcal{C}_w$ as in (\ref{PFeqCw}), and $K_{\rm Ai}$ the Airy kernel. Then it holds
\begin{equation*}
\det(\Id+\widetilde K_{\rm Ai})_{L^2(\mathcal{C}_w)}=\det(\Id-K_{\rm Ai})_{L^2(r,\infty)}.
\end{equation*}
\end{lemma}
\begin{proof}
The integration path in (\ref{PFeq27}) can be chosen to have $\Re(z)>0$ and since $\Re(w)<0$ for $w\in \mathcal{C}_w$, we can use
\begin{equation*}
\frac{1}{z-w}=\int_{\R_+}d\lambda e^{-\lambda(z-w)}
\end{equation*}
to write $\widetilde K_{\rm Ai}(w,w')=-(A B)(w,w')$ with $A:L^2(\mathcal{C}_w)\to L^2(\R_+)$ and \mbox{$B:L^2(\R_+)\to L^2(\mathcal{C}_w)$} have kernels
\begin{equation*}
A(w,\lambda)=e^{-w^3/3+w(r+\lambda)}, \quad B(\lambda,w')=\int_{e^{-\pi \I/4}\infty}^{e^{\pi \I/4}\infty} \frac{dz}{2\pi\I} \frac{e^{z^3/3-z(r+\lambda)}}{z-w'}.
\end{equation*}
We also have
\begin{equation*}
\begin{aligned}
(BA)(\eta,\eta')&=\frac{1}{2\pi\I}\int_{\mathcal{C}_w} dw B(\eta,w) A(w,\eta')\\
&=\frac{1}{(2\pi\I)^2} \int_{e^{-3\pi \I/4}\infty}^{e^{3\pi \I/4}\infty} dw \int_{e^{-\pi \I/4}\infty}^{e^{\pi \I/4}\infty} dz
\frac{1}{z-w}\frac{e^{z^3/3-z(r+\eta)}}{e^{w^3/3-w(r+\eta')}}= K_{\rm Ai}(\eta+r,\eta'+r).
\end{aligned}
\end{equation*}
Then, since $\det(\Id-AB)_{L^2(\mathcal{C}_w)}=\det(\Id-BA)_{L^2(\R_+)}=\det(\Id-K_{\rm Ai})_{L^2(r,\infty)}$ we get the claimed result. The first equality is a general result which applies as long as $AB$ and $BA$ are both trace-class operators \cite{Sim00}.
\end{proof}

\begin{lemma}\label{PFLemBBPreformuation}
Let $\widetilde K_{\rm BBP}$ as in (\ref{eqBBPtilde}), $\mathcal{C}_w$ as in Theorem~\ref{PFThmF2pert} (b), and $K_{\rm BBP}$ as in (\ref{eqBBP}). Then it holds
\begin{equation*}
\det(\Id+\widetilde K_{\rm BBP})_{L^2(\mathcal{C}_w)}=\det(\Id-K_{\rm BBP})_{L^2(r,\infty)}.
\end{equation*}
\end{lemma}
\begin{proof}
The proof is as the one of Lemma~\ref{PFLemTWreformuation}, except that in $A(w,\lambda)$ is multiplied by $\prod_{k=1}^m\frac{1}{w-b_k}$ and $B(\lambda,w')$ by $\prod_{k=1}^m (z-b_k)$.
\end{proof}

\begin{lemma}\label{PFLemKPZreformuation}
Let $\widetilde K_{{\rm CDRP}}$ as in (\ref{PFeqKPZtilde}), $\mathcal{C}_w$ as in (\ref{PFeqCwKPZ}), and $K_{{\rm CDRP}}$ the CDRP kernel given in (\ref{PFeqKPZ}). Then it holds
\begin{equation*}
\det(\Id+\widetilde K_{{\rm CDRP}})_{L^2(\mathcal{C}_w)}=\det(\Id-K_{{\rm CDRP}})_{L^2(\R_+)}.
\end{equation*}
\end{lemma}
\begin{proof}
Using
\begin{equation*}
\frac{1}{z-w'}=\int_{\R_+}d\eta e^{-\eta(z-w')}
\end{equation*}
we get
\begin{equation*}
\widetilde K_{{\rm CDRP}}(w,w')=\int_{\R_+} d\eta A(w,\eta) B(\eta,w')
\end{equation*}
with $B(\eta,w')=e^{\eta w'}$ and
\begin{equation*}
A(w,\eta)=\frac{-1}{2\pi\I}\int_{\frac{1}{4\sigma}+\I\R}dz \frac{\sigma \pi S^{(z-w)\sigma}}{\sin(\pi(z-w)\sigma)} e^{z^3/3-w^3/3-\eta z}.
\end{equation*}
Thus $\det(\Id+\widetilde K_{{\rm CDRP}})_{L^2(\mathcal{C}_w)}=\det(\Id-K_{{\rm CDRP}})_{L^2(\R_+)}$
where $K_{{\rm CDRP}}=-BA$, namely
\begin{equation}\label{PFeqA9}
\begin{aligned}
K_{{\rm CDRP}}(\eta,\eta')&=-\frac{1}{2\pi\I}\int_{-\frac{1}{4\sigma}+\I\R}dw B(\eta,w) A(w,\eta')\\
&=\frac{1}{(2\pi\I)^2}\int_{-\frac{1}{4\sigma}+\I\R} dw \int_{\frac{1}{4\sigma}+\I\R} dz \frac{\sigma \pi S^{(z-w)\sigma}}{\sin(\pi(z-w)\sigma)}  \frac{e^{z^3/3-z\eta'}}{e^{w^3/3-w\eta}}.
\end{aligned}
\end{equation}
The next step uses the following identity: for $0<\Re(u)<1$ it holds
\begin{equation*}
\frac{\pi\, S^{u}}{\sin(\pi u)}=\int_{\R} \frac{S e^{u t}}{S+e^t}dt
\end{equation*}
from which, for $0<\Re(u)<1/\sigma$ it holds
\begin{equation*}
\frac{\pi\, \sigma S^{\sigma u}}{\sin(\pi \sigma u)}=\int_{\R} \frac{S e^{-u t}}{S+e^{-t/\sigma}}dt.
\end{equation*}
We can use wit $u=z-w$ and obtain
\begin{equation}
\begin{aligned}
K_{{\rm CDRP}}(\eta,\eta')&=\int_{\R}dt \frac{S}{S+e^{-t/\sigma}} \bigg(\frac{1}{2\pi\I}\int_{-\frac{1}{4\sigma}+\I\R} dw  e^{-w^3/3+w(\eta+t)}\bigg)\bigg(\frac{1}{2\pi\I} \int_{\frac{1}{4\sigma}+\I\R} dz e^{z^3/3-z(\eta'+t)}\bigg)\\
&=\int_{\R}dt \frac{S}{S+e^{-t/\sigma}} \Ai(\eta+t)\Ai(\eta'+t),
\end{aligned}
\end{equation}
the expression of (\ref{PFeqKPZ}).
\end{proof}

\bibliographystyle{cpam}
\bibliography{Biblio}

\begin{thebibliography}{10}
\providecommand{\url}[1]{\texttt{#1}}
\providecommand{\urlprefix}{Available at: }
\providecommand{\eprint}[2][]{\url{#2}}

\bibitem{AS84}
Abramowitz, M.; Stegun, I. \emph{Pocketbook of Mathematical Functions}, Verlag
  Harri Deutsch, Thun-Frankfurt am Main, 1984.

\bibitem{AKQ10}
Alberts, T.; Khanin, K.; Quastel, J. The intermediate disorder regime for
  directed polymers in dimension 1+1. \emph{Phys. Rev. Lett.} \textbf{105}
  (2010), 090\,603.

\bibitem{AKQ12b}
Alberts, T.; Khanin, K.; Quastel, J. The continuum directed random polymer.
  \emph{arXiv:1202.4403}  (2012).

\bibitem{AKQ12}
Alberts, T.; Khanin, K.; Quastel, J. The intermediate disorder regime for
  directed polymers in dimension 1+1. \emph{arXiv:1202.4398}  (2012).

\bibitem{ACQ10}
Amir, G.; Corwin, I.; Quastel, J. {Probability distribution of the free energy
  of the continuum directed random polymer in 1+1 dimensions}. \emph{Comm. Pure
  Appl. Math.} \textbf{64} (2011), 466--537.

\bibitem{AAR04}
Andrews, G.; Askey, R.; Roy, R. \emph{{Special functions}}, Cambridge
  University Press, 2004.

\bibitem{Baik}
Baik, J. {Painlev\'{e} formulas of the limiting distributions for nonnull
  complex sample covariance matrices}. \emph{Duke J. Math.} \textbf{133}
  (2006), 205--235.

\bibitem{BBP05}
Baik, J.; {Ben Arous}, G.; P\'ech\'e, S. Phase transition of the largest
  eigenvalue for non-null complex sample covariance matrices. \emph{Ann.
  Probab.} \textbf{33} (2005), 1643--1697.

\bibitem{BDJ99}
Baik, J.; Deift, P.; Johansson, K. On the distribution of the length of the
  longest increasing subsequence of random permutations. \emph{J. Amer. Math.
  Soc.} \textbf{12} (1999), 1119--1178.

\bibitem{BR00}
Baik, J.; Rains, E. Limiting distributions for a polynuclear growth model with
  external sources. \emph{J. Stat. Phys.} \textbf{100} (2000), 523--542.

\bibitem{Bar01}
Baryshnikov, Y. {GUEs and queues}. \emph{Probab. Theory Relat. Fields}
  \textbf{119} (2001), 256--274.

\bibitem{BC09}
{Ben Arous}, G.; Corwin, I. {Current fluctuations for TASEP: a proof of the
  Pr\"ahofer-Spohn conjecture}. \emph{Ann. Probab.} \textbf{39} (2011),
  104--138.

\bibitem{BC95}
Bertini, L.; Cancrini, N. {The Stochastic Heat Equation: Feynman-Kac Formula
  and Intermittence}. \emph{J. Stat. Phys.} \textbf{78} (1995), 1377--1401.

\bibitem{BG97}
Bertini, L.; Giacomin, G. {Stochastic Burgers and KPZ equations from particle
  system}. \emph{Comm. Math. Phys.} \textbf{183} (1997), 571--607.

\bibitem{Bol89}
Bolthausen, E. A note on diffusion of directed polymers in a random
  environment. \emph{Comm. Math. Phys.} \textbf{123} (1989), 529--534.

\bibitem{BC11}
Borodin, A.; Corwin, I. Macdonald processes. \emph{arXiv:1111.4408}  (2011).

\bibitem{BFS09}
Borodin, A.; Ferrari, P.; Sasamoto, T. {Two speed TASEP}. \emph{J. Stat. Phys.}
  \textbf{137} (2009), 936--977.

\bibitem{BP07}
Borodin, A.; P\'ech\'e, S. Airy kernel with two sets of parameters in directed
  percolation and random matrix theory. \emph{J. Stat. Phys.} \textbf{132}
  (2008), 275--290.

\bibitem{CDR10}
Calabrese, P.; Doussal, P.~L.; Rosso, A. {Free-energy distribution of the
  directed polymer at high temperature}. \emph{EPL} \textbf{90} (2010),
  20\,002.

\bibitem{CC11}
Comets, F.; Cranston, M. {Overlaps and pathwise localization in the Anderson
  polymer model}. \emph{arXiv:1107.2011}  (2011).

\bibitem{CSY04}
Comets, F.; Shiga, T.; Yoshida, N. Probabilistic analysis of directed polymers
  in a random environment: a review. \emph{Advanced Studies in Pure
  Mathematics} \textbf{39} (2004), 115--142.

\bibitem{CJ06}
Comets, F.; Yoshida, N. {Directed polymers in random environment are diffusive
  at weak disorder}. \emph{Ann. Probab.}  (2006), 1746--1770.

\bibitem{Cor11}
Corwin, I. {The Kardar-Parisi-Zhang equation and universality class}.
  \emph{arXiv:1106.1596}  (2011).

\bibitem{COSZ11}
Corwin, I.; O'Connell, N.; Sepp{\"a}l{\"a}inen, T.; Zygouras, N. {Tropical
  combinatorics and Whittaker functions}. \emph{arXiv:1110.3489}  (2011).

\bibitem{CQ10}
Corwin, I.; Quastel, J. {Universal distribution of fluctuations at the edge of
  the rarefaction fan}. \emph{arXiv:1006.1338; To appear in Ann. Probab.}
  (2010).

\bibitem{CQ11}
Corwin, I.; Quastel, J. {Renormalization fixed point of the KPZ universality
  class}. \emph{arXiv:1103.3422}  (2011).

\bibitem{Dot10}
Dotsenko, V. {Replica Bethe ansatz derivation of the Tracy-Widom distribution
  of the free energy fluctuations in one-dimensional directed polymers}.
  \emph{J. Stat. Mech.}  (2010), P07\,010.

\bibitem{Eti99}
Etingof, P. {Whittaker functions on quantum groups and q-deformed Toda
  operators}. \emph{AMS Transl. Ser. 2} \textbf{194} (1999), 9--26.

\bibitem{Fer07}
Ferrari, P. {The universal Airy$_1$ and Airy$_2$ processes in the Totally
  Asymmetric Simple Exclusion Process}. in \emph{Integrable Systems and Random
  Matrices: In Honor of Percy Deift}, edited by J.~Baik; T.~Kriecherbauer;
  L.-C. Li; K.~McLaughlin; C.~Tomei, pp. 321--332, Contemporary Math., Amer.
  Math. Soc., 2008.

\bibitem{Fer10b}
Ferrari, P. {From interacting particle systems to random matrices}. \emph{J.
  Stat. Mech.}  (2010), P10\,016.

\bibitem{QM12}
Flores, G.~M.; Quastel, J. {Intermediate disorder for the O'Connell-Yor model}.
  \emph{In preparation}  (2012).

\bibitem{QMR12}
Flores, G.~M.; Quastel, J.; Remenik, D. \emph{In preparation}  (2012).

\bibitem{For93}
Forrester, P. The spectrum edge of random matrix ensembles. \emph{Nucl. Phys.
  B} \textbf{402} (1993), 709--728.

\bibitem{FNS77}
Forster, D.; Nelson, D.; Stephen, M. Large-distance and long-time properties of
  a randomly stirred fluid. \emph{Phys. Rev. A} \textbf{16} (1977), 732--749.

\bibitem{GLO11}
Gerasimov, A.; Lebedev, D.; Oblezin, S. {On a classical limit of q-deformed
  Whittaker functions}. \emph{arXiv:1101.4567}  (2011).

\bibitem{Giv97}
Givental, A. {Stationary phase integrals, quantum Toda lattices, flag manifolds
  and the mirror conjecture}. \emph{AMS Transl. Ser. 2} \textbf{180} (1997),
  103--116.

\bibitem{HH85}
Huse, D.; Henley, C. {Pinning and roughening of domain walls in Ising systems
  due to random impurities}. \emph{Phys. Rev. Lett.} \textbf{54} (1985),
  2708--2711.

\bibitem{SI04}
Imamura, T.; Sasamoto, T. Fluctuations of the one-dimensional polynuclear
  growth model with external sources. \emph{Nucl. Phys. B} \textbf{699} (2004),
  503--544.

\bibitem{SI11}
Imamura, T.; Sasamoto, T. {Replica approach to the KPZ equation with half
  Brownian motion initial condition}. \emph{J. Phys. A} \textbf{44} (2011),
  385\,001.

\bibitem{IS88}
Imbrie, J.; Spencer, T. Diffusion of directed polymers in a random environment.
  \emph{J. Stat. Phys.} \textbf{52} (1988), 609--626.

\bibitem{Jo00b}
Johansson, K. Shape fluctuations and random matrices. \emph{Comm. Math. Phys.}
  \textbf{209} (2000), 437--476.

\bibitem{Jo03b}
Johansson, K. Discrete polynuclear growth and determinantal processes.
  \emph{Comm. Math. Phys.} \textbf{242} (2003), 277--329.

\bibitem{Jo05}
Johansson, K. Random matrices and determinantal processes. in
  \emph{Mathematical Statistical Physics, Session LXXXIII: Lecture Notes of the
  Les Houches Summer School 2005}, edited by A.~Bovier; F.~Dunlop; A.~van
  Enter; F.~den Hollander; J.~Dalibard, pp. 1--56, Elsevier Science, 2006.

\bibitem{KPZ86}
Kardar, K.; Parisi, G.; Zhang, Y. Dynamic scaling of growing interfaces.
  \emph{Phys. Rev. Lett.} \textbf{56} (1986), 889--892.

\bibitem{Kos79}
Kostant, B.: Quantization and representation theory, in \emph{Representation
  Theory of Lie Groups, Proc. SRC/LMS Res. Symp., Oxford 1977. London Math.
  Soc. Lecture Notes Series}, vol.~34, 1979 pp. 287--316.

\bibitem{KK08}
Kriecherbauer, T.; Krug, J. {A pedestrian's view on interacting particle
  systems, KPZ universality, and random matrices}. \emph{J. Phys. A: Math.
  Theor.} \textbf{43} (2010), 403\,001.

\bibitem{Mac79}
Macdonald, I. \emph{{Symmetric functions and Hall polynomials}}, Clarendon
  Press Oxford, 1979.

\bibitem{MOC07}
Moriarty, J.; O'Connell, N. {On the free energy of a directed polymer in a
  Brownian environment}. \emph{Markov Process. Related Fields} \textbf{13}
  (2007), 251--266.

\bibitem{Mue91}
M{\"u}ller, C. On the support of solutions to the heat equation with noise.
  \emph{Stochastics} \textbf{37} (1991), 225--246.

\bibitem{NW93}
Nagao, T.; M.Wadati. Eigenvalue distribution of random matrices at the spectrum
  edge. \emph{J. Phys. Soc. Jpn.} \textbf{62} (1993), 3845--3856.

\bibitem{OCon09}
O'Connell, N. {Directed polymers and the quantum Toda lattice}.
  \emph{arXiv:0910.0069; To appear in Ann. Probab.}  (2009).

\bibitem{OCY01}
O'Connell, N.; Yor, M. {Brownian analogues of Burke's theorem}. \emph{Stoch.
  Proc. Appl.} \textbf{96} (2001), 285--304.

\bibitem{PS02}
Pr{\"a}hofer, M.; Spohn, H. Scale invariance of the {PNG} droplet and the
  {A}iry process. \emph{J. Stat. Phys.} \textbf{108} (2002), 1071--1106.

\bibitem{Q12}
Quastel, J. Lecture notes from current developments in mathematics, 2011.
  \emph{In preparation} .

\bibitem{Pec05}
{S. P\'ech\'e}. {The largest eigenvalue of small rank perturbations of
  Hermitian random matrices}. \emph{Probab. Theory Relat. Fields} \textbf{134}.

\bibitem{SS10}
Sasamoto, T.; Spohn, H. {One-dimensional Kardar-Parisi-Zhang equation: an exact
  solution and its universality}. \emph{Phys. Rev. Lett.} \textbf{104} (2010),
  230\,602.

\bibitem{Sep09}
Sepp{\"a}l{\"a}inen, T. {Scaling for a one-dimensional directed polymer with
  boundary conditions}. \emph{arXiv:0911.2446; To appear in Ann. Probab.}
  (2009).

\bibitem{SV10}
Sepp{\"a}l{\"a}inen, T.; Valko, B. Bounds for scaling exponents for a 1+1
  dimensional directed polymer in a brownian environment. \emph{ALEA, to
  appear} .

\bibitem{Sim00}
Simon, B. \emph{Trace Ideals and Their Applications}, American Mathematical
  Society, 2000, second edition ed.

\bibitem{Spo12}
Spohn, H. {KPZ scaling theory and the semi-discrete directed polymer model}.
  \emph{arXiv:1201.0645.} .

\bibitem{TS12}
Takeuchi, K.; Sano, M. {Evidence for geometry-dependent universal fluctuations
  of the Kardar-Parisi-Zhang interfaces in liquid-crystal turbulence}.
  \emph{arXiv:1203.2530.} .

\bibitem{TW94}
Tracy, C.; Widom, H. {Level-spacing distributions and the Airy kernel}.
  \emph{Comm. Math. Phys.} \textbf{159} (1994), 151--174.

\bibitem{TW08}
Tracy, C.; Widom, H. {A Fredholm determinant representation in ASEP}. \emph{J.
  Stat. Phys.} \textbf{132} (2008), 291--300.

\bibitem{TW08c}
Tracy, C.; Widom, H. {Integral formulas for the asymmetric simple exclusion
  process}. \emph{Comm. Math. Phys.} \textbf{279} (2008), 815--844.

\bibitem{TW08b}
Tracy, C.; Widom, H. {Asymptotics in ASEP with step initial condition}.
  \emph{Comm. Math. Phys.} \textbf{290} (2009), 129--154.

\bibitem{TW09b}
Tracy, C.; Widom, H. {On ASEP with step Bernoulli initial condition}. \emph{J.
  Stat. Phys.} \textbf{137} (2009), 825--838.

\bibitem{TW08cErratum}
Tracy, C.; Widom, H. {Erratum: Integral formulas for the asymmetric simple
  exclusion process}. \emph{Comm. Math. Phys.} \textbf{304} (2011), 875--878.

\bibitem{BKS85}
van Beijeren, H.; Kutner, R.; Spohn, H. Excess noise for driven diffusive
  systems. \emph{Phys. Rev. Lett.} \textbf{54} (1985), 2026--2029.

\end{thebibliography}

\end{document}